\RequirePackage{fix-cm}
\documentclass[smallextended,envcountsect]{svjour3}       
\smartqed  
\usepackage{graphicx}
\usepackage{amsmath,amssymb}
\usepackage{enumerate}
\usepackage{xcolor}
\usepackage[mathscr]{euscript}
\usepackage{url}
\usepackage[all]{xy}
\usepackage{booktabs}
\usepackage{graphicx}
\usepackage{epstopdf}
\usepackage{algorithmic}
\usepackage[subrefformat=parens, labelfont=up]{subcaption}
\usepackage{mathtools}
\usepackage{tablefootnote}
\usepackage{makecell}
\usepackage{multirow}
\usepackage{enumitem}

\usepackage{marvosym}
\usepackage{makecell}

\usepackage{multirow}

\usepackage[ruled,vlined]{algorithm2e}

\usepackage{hyperref}
\definecolor{cobalt}{rgb}{0.0, 0.28, 0.67}
\definecolor{darkblue}{rgb}{0.0, 0.0, 0.55}

\hypersetup
{
bookmarksopen=true,
pdftitle={Paper},
pdfauthor={Jianze Li},
pdfsubject={Rapport de th¨¨se},
pdfmenubar=true,
pdfhighlight=/O,
colorlinks=true,
pdfpagelayout=SinglePage,
pdffitwindow=true,
linkcolor=cobalt,
citecolor=cobalt,
urlcolor=cobalt
}

\newcommand{\tr}{\mbox{tr}}
\newcommand{\St}{\mathbf{St}}
\newcommand{\Gr}{\mathbf{Gr}}

\usepackage[capitalize,nameinlink,noabbrev]{cleveref}

\usepackage[scale=0.75]{geometry}
\usepackage{tikz}
\usetikzlibrary{shapes}
\usetikzlibrary{arrows}
\usetikzlibrary{matrix}

\newtheorem{assumption}{Assumption}

\definecolor{darkred}{rgb}{0.7,0,0}
\definecolor{darkgreen}{rgb}{0,0.46,0}
\definecolor{purple}{rgb}{0.6,0,0.5}
\definecolor{cholocate}{HTML}{d2691e}
\definecolor{slateblue}{HTML}{6a5acd}
\definecolor{dwtgreen}{HTML}{005f5f}

\definecolor{darkblue}{rgb}{0.0, 0.0, 0.55}

\newcommand{\tenss}[1]{\boldsymbol{\mathcal{#1}}}

\newcommand{\matr}[1]{\boldsymbol{#1}}

\newcommand{\vect}[1]{\boldsymbol{#1}}

\newcommand{\T}{{\sf T}}        

\makeatletter
\newcommand{\rank}[1]{\mathop{\operator@font rank}(#1)}
\newcommand{\colrank}[1]{\mathop{\operator@font colrank}\{#1\}}
\newcommand{\krank}[1]{\mathop{\operator@font krank}\{#1\}}
\newcommand{\trace}[1]{\mathop{\operator@font tr}\left(#1\right)}
\newcommand{\symmm}[1]{\mathop{\operator@font \bf symm}\left(#1\right)}
\newcommand{\symmo}[1]{\mathop{\operator@font sym}\left(#1\right)}
\newcommand{\skeww}[1]{\mathop{\operator@font skew}\left(#1\right)}
\newcommand{\skewww}[1]{\mathop{\operator@font \bf skew}\left(#1\right)}
\newcommand{\Diag}[1]{\mathop{\operator@font\bf Diag}\{#1\}}    
\newcommand{\diag}[1]{\textbf{\operator@font\bf diag}\left(#1\right)}    
\newcommand{\diagg}[2]{\mathop{\operator@font diag}_{#2}\{#1\}}    
\newcommand{\Span}[1]{\mathop{\operator@font Span}\{#1\}}    
\newcommand{\argmin}{\mathop{\operator@font argmin}}

\newcommand{\offdiag}[1]{\mathop{\operator@font offdiag}\{#1\}}    
\newcommand{\Proj}[2]{\mathop{\operator@font Proj_{#1}}{#2}}
\newcommand{\ProjGrad}[2]{\mathop{{\operator@font grad} }#1(#2)}

\newcommand{\expp}[1]{\mathop{\operator@font exp}\left(#1\right)}
\makeatother

\newcommand{\TangBundle}[1]{\mathbf{T}{#1}}

\newcommand{\eqdef}{\stackrel{\sf def}{=}}
\newcommand{\RR}{\mathbb{R}}

\newcommand{\NN}{\mathbb{N}}

\newcommand{\ON}[1]{\mathbf{O}_{#1}}

\newcommand{\contr}[1]{\mathop{\bullet_{#1}}}   

\DeclareMathOperator*{\argmax}{\arg\!\max}

\newcommand{\grad}{\operatorname{grad}} 
\newcommand{\TangSt}[1]{\mathbf{T}_{#1}{\St(r,n)}} 
\newcommand{\TangM}[1]{\mathbf{T}_{#1}{\mathcal{M}}} 

\newcommand{\TangMM}[1]{\mathbf{T}_{#1}{\mathcal{M}'}} 
\newcommand{\NormalSt}[1]{\mathbf{N}_{#1}{\St(r,n)}} 
\newcommand{\TangGr}[1]{\mathbf{T}_{#1}{\Gr(p,n)}} 
\newcommand{\NormalGr}[1]{\mathbf{N}_{#1}{\Gr(p,n)}} 
\newcommand{\mm}{\mathcal{M}} 
\newcommand{\NormalM}[1]{\mathbf{N}_{#1}{\mm}} 
\newcommand{\NormalMM}[1]{\mathbf{N}_{#1}{\mm'}} 
\newcommand{\NormalBB}{\mathbf{N}{\mm'}} 
\newcommand{\retr}{\operatorname{R}} 

\newcommand{\dist}{d} 
\newcommand{\openball}[2]{\mathcal{B}(#1;#2)} 
\newcommand{\closedball}[2]{\bar{\mathcal{B}}(#1;#2)} 
\newcommand{\DD}{\mathrm{D}} 
\newcommand{\ad}{\operatorname{ad}}

\newcolumntype{P}[1]{>{\centering\arraybackslash}p{#1}}

\graphicspath{{./}}

\begin{document}

\title
{Convergence analysis of the transformed gradient projection algorithms on compact matrix manifolds 
}

\titlerunning{Transformed gradient projection algorithms}
\author{Wentao Ding \and Jianze Li \and Shuzhong Zhang}
\authorrunning{W. Ding \and J. Li \and S. Zhang}

\institute{Wentao Ding \at
  Department of Systems Engineering and Engineering Management, The Chinese University of Hong Kong, Shatin,
N.T., Hong Kong  \\
\email{wentaoding@cuhk.edu.hk}
\and
Jianze Li 
\at
School of Science, Sun Yat-sen University, Shenzhen, Guangdong, China\\
\email{lijianze@mail.sysu.edu.cn}
\and
Shuzhong Zhang \at
Department of Industrial and Systems Engineering, University of Minnesota, Minneapolis, MN 55455, USA \\
\email{zhangs@umn.edu}
}

\date{Received: date / Accepted: date}

\maketitle

\begin{abstract}
In this paper, we study the optimization problem on a compact matrix manifold. 
While existing feasible algorithms can be broadly categorized into retraction-based and projection-based methods, compared to the more comprehensive and in-depth algorithmic and convergence research framework for retraction-based line-search (RetrLS) algorithms using only tangent vectors, the theoretical understanding and algorithmic design of projection-based line-search (ProjLS) algorithms remain limited, especially when general search directions and stepsizes are involved. To bridge this gap, we propose a unified algorithmic framework called the Transformed Gradient Projection (TGP) algorithm. The key idea is to construct the search direction as a transformed Riemannian (or Euclidean) gradient augmented by an additional normal component, allowing the framework to encompass and generalize numerous existing algorithms, such as the classical Euclidean gradient projection (EGP) algorithm, the shifted power method (Shifted PM),  and certain retraction-based algorithms, while also introducing several new special cases.
Then, we conduct a thorough exploration of the convergence properties of the TGP algorithms under various stepsizes, including the Armijo, Zhang-Hager type nonmonotone Armijo, and fixed stepsizes. 
To achieve this, we extensively analyze the geometric properties of the projection onto compact matrix manifolds, which may be of independent interest, allowing us to extend classical inequalities related to retractions from the literature. Building upon these insights, we establish the weak convergence, iteration complexity, and global convergence of TGP algorithms under three distinct stepsizes. 
In particular, we establish the global convergence of the Zhang-Hager type nonmonotone Armijo stepsize scheme, which, to our knowledge, has not been addressed  before in the nonconvex setting. 
In cases where the compact matrix manifold is the Stiefel or Grassmann manifold, our convergence results either encompass or surpass those found in the literature. 
Finally, through a series of numerical experiments and theoretical analysis, we observe that different choices of scaling matrices and normal components in the search direction of TGP algorithms can lead to significantly different performance in practice. 
As a result, the increased flexibility in choosing these components allows TGP algorithms to outperform classical EGP and retraction-based line-search algorithms in several scenarios. 

\vspace{0.8em}

\end{abstract}

\section{Introduction}

\subsection{Problem formulation}

Since the 1990s, optimization problems constrained on matrix manifolds have increasingly attracted attention in optimization theory and applied mathematics, due to their wide range of applications in various fields, including \emph{signal processing} \cite{absil2009optimization,Como10:book,li2019polar}, \emph{machine learning} \cite{journee2010generalized,bansal2018can,wang2020orthogonal}, \emph{numerical linear algebra} \cite{qi2017tensor,qi2009z} and \emph{data analysis}
\cite{Anan14:latent,Cichocki15:review,kolda2009tensor,sidiropoulos2017tensor}.
To facilitate the subsequent discussions, we first introduce the general form of the optimization problem considered in this paper. 
Let $ \mm \subseteq \RR^{n\times r}$ be a compact matrix submanifold of class \(C^3\) with $1\leq r\leq n$.
We mainly consider the following optimization problem:
\begin{align}\label{eq:objec_func_g}
\min_{\matr{X}\in\mm} f(\matr{X}),
\end{align}
where $f: \RR^{n \times r} \to \RR$ has Lipschitz continuous gradient in the convex hull of $\mm$. 

A diverse range of algorithms have been developed to address problem \eqref{eq:objec_func_g} in the literature, including both \emph{infeasible} and \emph{feasible} approaches. Infeasible methods encompass techniques such as the splitting methods \cite{Lai2014ASM}, and penalty methods \cite{Xiao2020ACO,xiao2021ExactPenaltyFunction,Wen2016TracePenaltyMF}.
Feasible methods mainly include two classes.
The first class stems from exploiting the geometric structure of $\mm$, allowing for the direct implementation of various Riemannian optimization algorithms by making use of differential-geometric principles like \emph{geodesic} and \emph{retraction}; see e.g. \cite{absil2009optimization,boumal2023intromanifolds,hu2019brief}.
The second class\footnote{In this paper, we will focus on the second class of feasible methods.} of feasible methods ensures that iterations consistently remain within the manifold by using the \emph{projection} onto compact matrix manifolds.

\subsection{Retraction-based line-search algorithms}\label{subsec:retrac_Riema_optim}
To solve optimization problems constrained on matrix manifolds, such as problem (1) above, early research efforts often attempted to adopt \emph{geodesics} \cite{luenberger1984linear} as the iteration path, which is a fundamental concept in the theory of differential manifolds. 
It, in essence, embodies a locally shortest path on the manifold, offering a generalization of straight lines in Euclidean space \cite{do2016differential,gabay1982minimizing}.
Note that,
in classical unconstrained optimization, a thoroughly examined class of line-search algorithms explores along straight lines in each iteration \cite{luenberger1984linear}.
Therefore, it is natural to extend these classical line-search methods to manifolds through the utilization of geodesics. These types of algorithms have been extensively studied in early works; see
\cite{luenberger1984linear,gabay1982minimizing,smith1994optimization,yang2007globally}
and the references therein. It is worth noting that the algorithms proposed therein usually presume the explicit calculation of geodesics along a given direction.

While closed-form expressions of geodesics are available only for certain manifolds \cite{edelman1998geometry,smith1993geometric}, the computation of geodesics can be computationally expensive or even impractical in general, as shown in \cite{botsaris1981constrained,absil2012projection}. To address this challenge, it has been suggested to approximate exact geodesics using computationally efficient alternatives \cite{botsaris1981constrained,gao2022optimizationon}. For example, in the context of the Stiefel manifold, various specially designed curves along search directions have been constructed with low computational cost, and curvilinear search algorithms have subsequently been developed based on these curves \cite{wen2013feasible,wen2013adaptive,jiang2015framework}.

Note that geodesics can often be computed through an exponential map \cite{edelman1998geometry}. To approximate these geodesics effectively,  it suffices to find an approximation of the exponential map, which gives rise to the concept of \emph{retraction} \cite{absil2009optimization,shub1986some}. Let $\mm'\subseteq\RR^{m}$ be a submanifold. A smooth map $ \retr $ from the tangent bundle $ \TangBundle{\mm'} $ to $ \mm' $ is said to be a \emph{retraction} on $ \mm' $ if it satisfies the following properties:
\begin{itemize}
\item[(i)] $ \retr(\vect{x}, \vect{0}_{\vect{x}}) = \vect{x} $ for all $ \vect{x} \in \mm' $, where $\vect{0}_{\vect{x}} $ denotes the zero element in the \emph{tangent space}\footnote{See \cref{Concepts_RM} for a definition.} $ \TangMM{\vect{x}} $;
\item[(ii)] The differential of $ \retr_{\vect{x}} $ at $ \vect{0}_{\vect{x}} $ is the identity map on $ \TangMM{\vect{x}} $. Here $ \retr_{\vect{x}}: \TangMM{\vect{x}} \to \mm' $ denotes the restriction of $ \retr $ to $ \TangMM{\vect{x}} $, \emph{i.e.}, $ \retr_{\vect{x}}(\cdot) \eqdef \retr(\vect{x}, \cdot) $.
\end{itemize}
Define a curve $ c(t) $ on $ \mm' $ passing through $ \vect{x} $ for some $ \vect{x} \in \mm' $ and $ \vect{v} \in \TangMM{\vect{x}} $ by $ c(t) \eqdef \retr_{\vect{x}}(t\vect{v}) $.
The above definition of retraction implies that $ c'(0) = \vect{v} $, and thus this curve $ c(t) $ serves as a first-order approximation of the geodesic passing through $ \vect{x} $ along the direction $ \vect{v} $ \cite{absil2012projection}.

Over recent decades, numerous retractions have been developed for commonly used manifolds, many of which can be computed efficiently or have closed-form solutions; see \cite{absil2009optimization,boumal2023intromanifolds,hu2019brief}.
The derivation of retractions enables the adoption of classical algorithms in unconstrained optimization to general Riemannian manifolds.
Up to now, various retraction-based Riemannian optimization algorithms have been developed, including
\emph{Riemannian gradient descent} (RGD) \cite{absil2009optimization,boumal2019GlobalRatesConvergence,liu2019QuadraticOptimizationOrthogonality,sheng2022riemannian},
\emph{Newton-type} \cite{hu2018adaptive,huang2015broyden,huang2018riemannian,zhang2018cubic,zhang2023riemannian} and \emph{trust region} \cite{absil2007trust,boumal2019GlobalRatesConvergence}.
In particular, in the context of addressing problem \eqref{eq:objec_func_g}, the update scheme of \emph{retraction-based line-search} (RetrLS) algorithms can be represented as:
\begin{equation}\label{eq:iter_retr}
\matr{X}_{k+1} = \retr_{\matr{X}_k} (\tau_k\matr{V}_k).
\end{equation}
Here, $ \matr{V}_k \in \TangM{\matr{X}_k} $ is a \emph{descent direction} such as $ -\grad f(\matr{X}_k) $, where $\grad f(\matr{X}_k) $ is the \emph{Riemannian gradient} \cite{absil2009optimization} of $f$ at $\matr{X}_k$, \(\tau_k>0\) is the \emph{stepsize} selected by certain rules and $ \retr $ is a retraction on $\mm$. 
As a special case, the classical \emph{Riemannian gradient descent} (RGD) algorithm \cite{absil2009optimization,boumal2019GlobalRatesConvergence,liu2019QuadraticOptimizationOrthogonality,sheng2022riemannian} is 
\begin{equation}\label{eq:iter_retr_RGD} 
\matr{X}_{k+1} = \retr_{\matr{X}_k} (-\tau_k\grad f(\matr{X}_k)).
\end{equation}

In recent years, there has been a growing interest in the convergence analysis of the RetrLS update scheme \eqref{eq:iter_retr}
\cite{boumal2019GlobalRatesConvergence,zhang2016riemannian,zhang2016first}.
Notably, the \emph{weak convergence}\footnote{Every accumulation point of the iterates is a stationary point,  \emph{i.e.}, the Riemannian gradient of the cost function at this point is $\vect{0}$.} of general first-order line-search algorithms on a general manifold has been established in \cite{absil2009optimization} under a \emph{gradient-related} assumption.
The research conducted in \cite{liu2019QuadraticOptimizationOrthogonality,sheng2022riemannian} demonstrates the \emph{global convergence}\footnote{For any starting point, the iterates converge as a whole sequence.} of these algorithms on the Stiefel manifold.
It has also been shown in \cite{liu2019QuadraticOptimizationOrthogonality} that the sequence generated by these algorithms exhibits linear convergence for quadratic optimization on the Stiefel manifold.
The work presented in \cite{bento2017iteration,boumal2019GlobalRatesConvergence} derived the \emph{iteration complexity}\footnote{In \cite{boumal2019GlobalRatesConvergence}, the term \emph{convergence rate} was used to denote what we refer to as \emph{iteration complexity}, \emph{i.e.}, the number of iterations required to reach a certain solution accuracy.} of the gradient descent type algorithm under specific conditions.
In particular, the RGD algorithm \eqref{eq:iter_retr_RGD} attains a first-order \(\epsilon\)-stationary point within $\mathcal{O}(\epsilon^{-2})$
iterations on a general compact submanifold of Euclidean space.
Although not as extensive as the convergence studies on the Stiefel manifold, for the Grassmann manifold in the form of the set of projection matrices, the weak convergence of the RGD algorithm \eqref{eq:iter_retr_RGD} was also established \cite{sato2014optimization}.

\subsection{Gradient projection method}\label{subsec:grad_proj_meth}

In addition to the RetrLS algorithms \eqref{eq:iter_retr} discussed in \cref{subsec:retrac_Riema_optim}, the other feasible approach to addressing problem \eqref{eq:objec_func_g} is through the classical \emph{Euclidean gradient projection} (EGP) algorithm, which selects the next iterate by
\begin{equation}\label{eq:iter_prjec}
\matr{X}_{k+1} = \mathcal{P}_{\mathcal{M}} (\matr{X}_{k}-\tau_k\nabla f(\matr{X}_{k})),
\end{equation}
where $\mathcal{P}_{\mathcal{M}}:\RR^{n\times r}\rightarrow \mm$ denotes the \emph{projection} mapping onto $\mm$ computing the best approximation, and $\tau_{k}>0$ is the stepsize.
Since \( \nabla f(\matr{X}_{k}) \) does not necessarily lie in the tangent space, standard retraction techniques are no longer applicable for ensuring feasibility. 
Therefore, we need to use a projection operator $\mathcal{P}_{\mathcal{M}}$ to ensure feasibility on the compact manifold. 

In recent years, there has been extensive research on the convergence of the EGP algorithm \eqref{eq:iter_prjec} for addressing \emph{phase synchronization}  \cite{liu2017estimation,ling2022improved,zhu2021orthogonal} and \emph{tensor approximation} problems \cite{chen2009tensor,yang2019epsilon,hu2019LROAT,li2019polar}, where the feasible set can be the Stiefel manifold, the product of Stiefel manifolds, or the Grassmann manifold.
In addition, various variants of the EGP algorithm \eqref{eq:iter_prjec} have also been studied, as well as their convergence properties. For example, an algorithm combining \eqref{eq:iter_prjec} with a correction step was proposed in \cite{gao2018new}.
In the long line of work presented in \cite{oviedo2019scaled,oviedo2021two,oviedo2019non}, the term $ \matr{X}_k - \tau_k \nabla f(\matr{X}_k) $ in \eqref{eq:iter_prjec} was substituted with various forms incorporating different stepsizes, and the weak convergence of them was established\footnote{See \cref{subsec:special_literature} for more details.}.
In this paper, these algorithms are collectively referred to as \emph{projection-based line-search} (ProjLS) algorithms, and the update scheme of them can be summarized as \begin{equation}\label{eq:iter_prjec-h}
\matr{X}_{k+1} = \mathcal{P}_{\mathcal{M}} (\matr{X}_{k}-\tau_k\matr{H}_{k}),
\end{equation}
where $\matr{H}_{k}\in\RR^{n\times r}$ is the \emph{search direction}. 
Here, to facilitate the discussion that follows, we divide the search direction $\matr{H}_{k}$ into two parts as follows:
\begin{equation}\label{eq:iter_prjec-h-sum}
\matr{H}_k = \tilde{\matr{H}}_k + \hat{\matr{H}}_k, 
\end{equation}
where \( \tilde{\matr{H}}_k\eqdef\mathcal{P}_{\TangM{\matr{X}_k}}(\matr{H}_k)\in \TangM{\matr{X}_k} \) is a \emph{tangent component}, and \( \hat{\matr{H}}_k \eqdef\mathcal{P}_{\NormalM{\matr{X}_k}}(\matr{H}_k) \in \NormalM{\matr{X}_k} \) is a \emph{normal component}\footnote{See \cref{Concepts_RM} for a definition.}. 

Although the update schemes \eqref{eq:iter_retr} and \eqref{eq:iter_prjec-h} both keep the iterates in the feasible region $ \mm $ and, for tangent vectors  $\matr{V}\in\TangM{\matr{X}}$, the map $ \operatorname{R}_{\matr{X}}:\matr{V}\mapsto \mathcal{P}_{\mathcal{M}}(\matr{X} + \matr{V}) $ forms a retraction
\cite{absil2012projection}, there still exist fundamental differences between them.
For example, the search direction $\matr{H}_{k}$ in \eqref{eq:iter_prjec-h} is not necessarily tangent to $ \matr{X} $ in general, and there also exist other choices of retraction in \eqref{eq:iter_retr} besides the projection. 
Therefore, the existing analysis of RetrLS algorithms \eqref{eq:iter_retr} cannot be directly applied to the projection-based ones in \eqref{eq:iter_prjec-h}.

Regarding the practical performance, compared to the RetrLS algorithms \eqref{eq:iter_retr}, the above ProjLS algorithms \eqref{eq:iter_prjec-h} may achieve better experimental results in several scenarios, as their search direction can also include various normal components, which increases the likelihood of finding a global optimal solution or accelerating the algorithm in certain cases.
In fact, the superior empirical performance of projection-based algorithms compared to retraction-based ones has also been reported in the literature; see, for example, \cite{dalmau2018projected,gao2018new}.
In \cref{sec:numer_exper}, we will further illustrate this phenomenon through several numerical results, as well as theoretical analysis in some concrete examples.

\subsection{Limitations in the study of projection-based line-search methods}\label{subsec:limitations_problem}

Although several research works in the literature have explored ProjLS algorithms \eqref{eq:iter_prjec-h} on compact manifolds, as discussed in \cref{subsec:grad_proj_meth}, most of these works focus on specific cases, such as a specific manifold or a specific search direction.
In contrast to the more comprehensive and in-depth algorithmic and convergence research framework for  RetrLS algorithms \eqref{eq:iter_retr}  using only tangent vectors \cite{absil2009optimization,boumal2023intromanifolds},
 ProjLS algorithms \eqref{eq:iter_prjec-h} still lack a unified framework and convergence analysis for general cases.
This limitation can be more precisely illustrated through the following three aspects: (i) To our knowledge, the convergence of ProjLS algorithms \eqref{eq:iter_prjec-h} is mostly concerned with \emph{weak convergence}
\cite{chen2009tensor,gao2018new,oviedo2022global,oviedo2023worst,oviedo2019scaled,oviedo2021two,oviedo2019non},
and the limited \emph{global convergence} results are all for fixed step sizes \cite{hu2019LROAT,li2019polar,yang2019epsilon},
typically derived using the Taylor expansion.
(ii) Little attention has been paid to the impact of the normal component in the search direction \eqref{eq:iter_prjec-h-sum} on the performance of ProjLS algorithms \eqref{eq:iter_prjec-h}. In fact, this impact can be implicitly observed, such as in the iteration complexity of the \emph{shifted power method} (Shifted PM), which is a special case\footnote{See the discussion around \eqref{eq:iter_power_method} for details.} of the ProjLS algorithms \eqref{eq:iter_prjec-h}. However, no existing algorithmic framework currently allows for the adaptive adjustment of the normal direction.
(iii) In many ProjLS algorithms \eqref{eq:iter_prjec-h}, the tangent component in the search direction \eqref{eq:iter_prjec-h-sum} is often chosen to be the Riemannian gradient. It remains unclear whether an algorithm will still converge if the tangent component of the search direction is not the Riemannian gradient. 
A more systematic investigation is needed to understand the implications of this deviation.

In summary, the above discussions and limitations can be summarized as the following key question:

\begin{itemize}
\item[] \emph{Can one develop a unified framework for ProjLS algorithms that both generalizes existing algorithms and enables stronger convergence guarantees?}
\end{itemize}

\subsection{Contributions}\label{subsec:contri_TGP}

In this paper, we aim to address the above question.
To be more specific, to ensure that the tangent component of the search direction \eqref{eq:iter_prjec-h-sum} encompasses a broad class of existing projection-based algorithms summarized in \cref{subsec:special_literature}, we construct it in the form of a \emph{transformed Riemannian gradient} as follows:
\begin{align}\label{eq:c_tilde_hat_transf}
\matr{L}(\matr{X}_k) \grad f(\matr{X}_k) \matr{R}(\matr{X}_k), 
\end{align}
where $\matr{L}(\matr{X}_k)\in \symmm{\RR^{n\times n}}$ and $\matr{R}(\matr{X}_k)\in\symmm{\RR^{r\times r}}$ are \emph{left} and \emph{right scaling matrices}, respectively.
In \cref{subsec:TGP_framework}, we will impose \cref{asp:H-tilde-grad-equiv} on the left and right scaling matrices to ensure that the expression in \eqref{eq:c_tilde_hat_transf} remains in the tangent space and thus still represents the tangent component $\tilde{\matr{H}}_k$. 
At the same time, this assumption guarantees that the expression in \eqref{eq:c_tilde_hat_transf} is \emph{scale-equivalent} to and \emph{directionally aligned} with the Riemannian gradient.\footnote{See \cref{lem:proper_c_k}(ii) for a detailed explanation.}
In this sense, it will be seen that the search direction also admits an alternative decomposition that is more closely related to the classical EGP algorithm in \eqref{eq:iter_prjec}, as shown below: 
\begin{align}
\matr{H}_k &\eqdef \matr{L}(\matr{X}_{k})\grad f(\matr{X}_k)\matr{R}(\matr{X}_k)+\matr{N}(\matr{X}_k)\label{eq:Trans_RGP}, \\
&=\matr{L}(\matr{X}_{k})\nabla f(\matr{X}_{k})\matr{R}(\matr{X}_k)+\matr{N}'(\matr{X}_k),\label{eq:Trans_EGP}
\end{align}
where  $\matr{N}'(\matr{X}_k)=\matr{L}(\matr{X}_{k})(\nabla f(\matr{X}_{k})-\grad f(\matr{X}_k))\matr{R}(\matr{X}_k)+\matr{N}(\matr{X}_k)$ is also a normal vector; as will be shown in \cref{lem:proper_c_k}(i).
In this paper, we refer to the ProjLS algorithms that use the search direction \eqref{eq:Trans_RGP} (or equivalently, \eqref{eq:Trans_EGP}) as the \emph{Transformed Gradient Projection} (TGP) algorithm (see \cref{alg:TGP}).
This can be interpreted as a \emph{transformed} variant of both RGD in \eqref{eq:iter_retr_RGD} and EGP in \eqref{eq:iter_prjec},  augmented by an additional normal matrix component.
Throughout the paper, we will make use of both equivalent formulations as needed.
The first form \eqref{eq:Trans_RGP} is more convenient for convergence analysis, as it clearly separates the tangent and normal components. 
In contrast, the second form \eqref{eq:Trans_EGP} more clearly reveals the connection between our framework and existing projection-based algorithmic updates, as it explicitly includes the Euclidean gradient term. 

The main contributions of this paper can be summarized as follows:
\begin{itemize}
\item \emph{Generality}: Our TGP algorithmic framework in \cref{alg:TGP} is quite general, encompassing the classical EGP algorithm \eqref{eq:iter_prjec} and Shifted PM method \eqref{eq:iter_power_method} as special cases, and intersecting the RetrLS algorithms \eqref{eq:iter_retr}.
It is a subclass of the ProjLS algorithms \eqref{eq:iter_prjec-h}.
An illustration of the relationships among these algorithms can be found in \cref{fig:compare-retra-proj-alg}. 
\item \emph{Important special cases}: For problem \eqref{eq:objec_func_g} and the proposed TGP algorithmic framework, our specific emphasis lies on the Stiefel or Grassmann manifold. It is evident that many important algorithms in the literature can be viewed as special cases of \cref{alg:TGP}, as detailed in \cref{subsec:special_literature}.
It also introduces several new special cases that have not been studied in the literature; see \cref{exa:LR_St,exa:LR_Gr}. In particular, by specifying the scaling matrices and normal vector, we obtain new algorithms TGP-$\ast$-E/R \eqref{eq:H_TGP_R_E}, TGP-$\ast$-DE/DF in \eqref{eq:TGP-DEF} and TGP-A-Eigen \eqref{eq:H-TGP-eigen}, which may outperform the classical algorithms in certain scenarios, as shown in \cref{sec:numer_exper}. 
\item \emph{Geometric properties of the projection}: In the convergence analysis of RetrLS algorithms \eqref{eq:iter_retr}, certain geometric inequalities related to the retraction play a crucial role \cite{boumal2019GlobalRatesConvergence,li2021weakly,liu2019QuadraticOptimizationOrthogonality}. However, similar inequalities for projection operators were previously lacking in the literature. Therefore, in this paper, to analyze the convergence of the TGP algorithms, we need to develop several new inequalities related to the projection onto a compact matrix manifold $\mm$, which characterize the variations in distance and function values under projection.
These results are crucial for the subsequent investigation of the convergence properties of TGP algorithms, and extend certain inequalities found in the literature concerning retractions \cite{boumal2019GlobalRatesConvergence,li2021weakly,liu2019QuadraticOptimizationOrthogonality}, when the retraction is constructed using the projection. A summary of our key geometric results is provided in Table~\ref{tab:comparison_proximally_smooth}. See \cref{sec:Prop-proj} for more details. 
\item \emph{Convergence properties}: We conduct a systematic exploration of the convergence properties of TGP algorithms across various stepsizes, encompassing the Armijo, Zhang-Hager type nonmonotone Armijo and fixed stepsizes, and establish their weak convergence, iteration complexity bound and global convergence under \cref{asp:H-tilde-grad-equiv} and \cref{asp:H-boundedness} (see \cref{sec:TGP_alg_frame,sec:TGP-A}).
The Armijo stepsize is a very natural option, and has been widely adopted in various types of algorithms.
In order to better understand the behavior of the proposed TGP algorithm and improve its practical performance, we also consider the Zhang-Hager type nonmonotone stepsize and a fixed stepsize as alternative strategies. 
When the compact matrix manifold is the Stiefel or Grassmann manifold, the convergence results we obtain in these specific cases either encompass or surpass the existing results in the literature.
In particular, we prove the global convergence of \cref{alg:TGP} under Zhang-Hager type nonmonotone Armijo stepsize. To our knowledge, this is the first time that the global convergence is established for the Zhang-Hager type nonmonotone Armijo stepsize in the nonconvex setting. 
\cref{lem:nonmonotone-global-convergence}, which we prove, will also contribute to establishing the global convergence of other analogous nonmonotone algorithms. 
\item \emph{Experimental efficiency}:  Through a series of numerical experiments and theoretical analysis, we show that different choices of scaling matrices and normal components in the search direction of \cref{alg:TGP} can lead to significantly different performance in practice.
As a result, the increased flexibility in choosing these components allows TGP algorithms to achieve superior experimental results compared to RetrLS algorithms \eqref{eq:iter_retr} and the EGP algorithm \eqref{eq:iter_prjec} in several scenarios.
\end{itemize}  

In this paper, as two significant examples of the problem \eqref{eq:objec_func_g}, we mainly focus on the \emph{Stiefel manifold} \cite{edelman1998geometry,wen2013feasible} and the \emph{Grassmann manifold} \cite{bendokat2020grassmann}.
In fact, it is worth noting that the proposed TGP algorithms and convergence results of this paper also apply to other compact matrix manifolds as well, \emph{e.g.}, the \emph{oblique manifold} \cite{trendafilov2002multimode} and the \emph{product of Stiefel manifolds} \cite{hu2018adaptive,li2019polar}, although we do not dive into the details in this paper.

\subsection{Organization}
The paper is organized as follows.
In \cref{sec:prelimi}, we recall several concepts for the Riemannian manifold, as well as the \L{}ojasiewicz gradient inequality.
In \cref{sec:TGP_alg_frame}, we introduce the TGP algorithmic framework, propose the assumptions concerning the scaling matrices $\matr{L}(\matr{X}_{k})$ and $\matr{R}(\matr{X}_{k})$, review related existing algorithms from the literature, and summarize the convergence results, which we will obtain, in \cref{table-example-3-0-9}.
In \cref{sec:Prop-proj},
we study the geometric properties of the projection $\mathcal{P}_{\mathcal{M}}$, which will play a crucial role in the subsequent convergence analysis of TGP algorithms.
In \cref{sec:TGP-A,sec:nonmonotone_TGP,sec:fixed_TGP},
we focus on the investigation of TGP algorithms employing Armijo, Zhang-Hager type nonmonotone Armijo and fixed stepsizes, respectively.
We establish their weak convergence, iteration complexity and global convergence using the geometric properties of the projection.
In \cref{sec:numer_exper}, we conduct several numerical experiments to verify the efficiency of TGP algorithms.
In \cref{sec:conclu}, we provide a summary of the paper and discuss potential directions for future research. 

\section{Preliminaries}\label{sec:prelimi}

\subsection{Notation}
In this paper, we endow the Euclidean space $\RR^{n \times r}$ with the standard Euclidean inner product defined as $\langle \matr{X}, \matr{Y} \rangle \eqdef \tr(\matr{X}^{\T}\matr{Y})$ for $\matr{X}, \matr{Y} \in \RR^{n \times r}$.
For a matrix $\matr{X}\in \RR^{n\times r}$, we denote its Frobenius norm by $\|\matr{X}\| \eqdef \sqrt{\langle \matr{X}, \matr{X} \rangle}$ and its Schatten $p$-norm by $\| \matr{X} \|_p$. In particular, $ \| \matr{X} \|_{\infty} $ represents its spectral norm. The smallest and largest singular values of $\matr{X}$ are denoted by $\sigma_{\min}(\matr{X})$ and $\sigma_{\max}(\matr{X})$, respectively.
Let \( \symmm{\RR^{m \times m}} \) and $\skewww{\RR^{m \times m}}$ denote the sets of \emph{symmetric} and \emph{skew-symmetric} matrices in $\RR^{m \times m}$, respectively.
For $ \matr{X}\in\symmm{\RR^{m\times m}}$, $\lambda_{\min}(\matr{X})$ and $\lambda_{\max}(\matr{X})$ refer to its smallest and largest eigenvalues, respectively.
For two matrices $ \matr{X}, \matr{Y}\in \symmm{\RR^{m\times m}}$, $ \matr{X} \succeq \matr{Y}$ indicates that $ \matr{X} - \matr{Y} $ is positive semi-definite.
For a matrix $\matr{X}\in\RR^{m\times m}$, we denote
$\symmo{\matr{X}}\eqdef\frac{1}{2}\left(\matr{X}+\matr{X}^{\T}\right)$ and $\skeww{\matr{X}}\eqdef\frac{1}{2}\left(\matr{X}-\matr{X}^{\T}\right)$.
We denote by \( \matr{I}_m \in \RR^{m \times m} \) the identity matrix and by \( \matr{0}_{m \times m'} \in \RR^{m \times m'} \) the zero matrix. If the dimension is clear from the context, we will abbreviate it as \( \matr{0} \).
Given multiple square matrices \( \matr{X}_1, \matr{X}_2 , \ldots , \matr{X}_{L} \in \RR^{m \times m} \), we denote by \( \Diag{\matr{X}_1 , \matr{X}_2, \ldots , \matr{X}_L} \in \RR^{Lm \times Lm} \) the square block diagonal matrix consisting of the given matrices.

For a point $ \vect{x}\in\RR^{m} $ and a subset $\mathcal{S}\subseteq\RR^{m}$, we denote by $ \dist(\vect{x}, \mathcal{S}) \eqdef \inf_{\vect{y} \in \mathcal{S}} \| \vect{x} - \vect{y} \| $ the distance between them. The projection of \( \vect{x} \) onto \( \mathcal{S} \) is denoted by \( \mathcal{P}_{\mathcal{S}}(\vect{x}) \). When the projection is not unique, \( \mathcal{P}_{\mathcal{S}}(\vect{x}) \) denotes an arbitrary projection of \( \vect{x} \). We define $ \openball{\vect{x}}{\rho}\eqdef \{ \vect{y}: \| \vect{y}-\vect{x} \| < \rho \} $ as the open ball centered at $ \vect{x} $ with radius $ \rho>0$, and similarly define \( \openball{\mathcal{S}}{\rho} \eqdef \{ \vect{y}: \dist(\matr{y}, \mathcal{S}) < \rho \} \).
Given a differentiable function $ F: \mathcal{E} \to \mathcal{E}'$ between two linear spaces $ \mathcal{E} $ and $\mathcal{E}'$, its differential at $\vect{x} \in \mathcal{E}$ is the linear mapping $\DD F(\vect{x})$ from $\mathcal{E}$ to $\mathcal{E}'$ defined as:
\[ \DD F(\vect{x})[\vect{v}] \eqdef \lim_{t \to 0} \frac{F(\vect{x}+t \vect{v}) - F(\vect{x})}{t}, \text{ for all } \vect{v} \in \mathcal{E}. \]
For the cost function $f$ in \eqref{eq:objec_func_g},
we denote $ f^{*} \eqdef  \min_{\matr{X} \in \mm} f(\matr{X}) $.

\subsection{Basic concepts for Riemannian manifold}\label{Concepts_RM}
In this paper, we consider $\mm$ as a submanifold of the ambient space $\RR^{n \times r}$, and endow $\mm$ with the Riemannian metric induced from the Euclidean metric on $\RR^{n \times r}$. To be more specific, the inner product on the \emph{tangent space} to $\mm$ at \( \matr{X} \in \mm \), denoted by $\TangM{\matr{X}}$, is defined as:
\[ \langle \matr{V}, \matr{V}' \rangle_{\matr{X}} \eqdef \langle \matr{V}, \matr{V}' \rangle = \tr(\matr{V}^{\T}\matr{V}'), \text{ for all } \matr{V}, \matr{V}' \in \TangM{\matr{X}}.  \]
We use $\NormalM{\matr{X}}$ to represent the \emph{normal space} to $\mm$ at $\matr{X}$, which is the orthogonal complement of the tangent space $\TangM{\matr{X}}$.
For the cost function $f$ in \eqref{eq:objec_func_g}, its \emph{Riemannian gradient} at $\matr{X} \in \mm$ is defined to be the unique tangent vector $\grad f(\matr{X}) \in \TangM{\matr{X}}$ satisfying
\[ \DD f(\matr{X})[\matr{V}] = \langle \grad f(\matr{X}), \matr{V} \rangle_{\matr{X}}, \text{ for all } \matr{V} \in \TangM{\matr{X}}. \]
In our setting where $\mm$ is a Riemannian submanifold of $\RR^{n \times r}$, the Riemannian gradient is equal to the projection of the classical Euclidean gradient \( \nabla f(\matr{X}) \) onto the tangent space to $\mm$ at $\matr{X}$ \cite[Eq. 3.37]{absil2009optimization}, \emph{i.e.},
\begin{equation}\label{eq:Riemannian-gradient-of-submanifold}
\grad f(\matr{X}) = \mathcal{P}_{\TangM{\matr{X}}}(\nabla f(\matr{X})).
\end{equation}
\begin{example}[Stiefel manifold]
The \emph{Stiefel manifold} is defined as $\St(r,n) \eqdef \{\matr{X}\in\RR^{n\times r}: \matr{X}^{\T}\matr{X}=\matr{I}_r\}$ \cite{stiefel1935richtungsfelder}.
If $r=1$, it is the unit sphere $\mathbb{S}^{n-1} \subseteq \RR^{n}$, and when $r=n$, it becomes the $n$-dimensional orthogonal group $\ON{n}\subseteq\RR^{n\times n}$. 
As demonstrated in \cite[Ex. 3.5.2]{absil2009optimization}, the tangent space to $\St(r, n)$ at $\matr{X} \in \St(r, n)$ can be expressed as
\begin{align}
\TangSt{\matr{X}} & =\{ \matr{V} \in \RR^{n \times r}: \matr{X}^{\T} \matr{V} + \matr{V}^{\T}\matr{X} = \matr{0}\}\label{def-St-tangent-space-product}\\
& = \{ \matr{X}\matr{A} + \matr{X}_{\perp}\matr{B}: \matr{A} \in \skewww{\RR^{r \times r}}, \matr{B} \in \RR^{(n-r) \times r} \}\label{def-St-tangent-space-decomp},
\end{align}
where \( \matr{X}_{\perp} \in \St(n-r, n) \) satisfies $[\matr{X},\matr{X}_{\perp}]\in\ON{n}$.
The normal space to \(\St(r, n)\)  at \( \matr{X} \in \St(r, n) \) satisfies
\begin{equation}\label{def-St-normal-space}
\NormalSt{\matr{X}} = \{ \matr{X}\matr{S}: \matr{S} \in \symmm{\RR^{r \times r}} \}.
\end{equation}
Moreover, the orthogonal projection of an arbitrary point $\matr{Y}\in\RR^{n\times r}$ onto the tangent space $\TangSt{\matr{X}}$ can be computed by
\begin{equation}\label{eq:proj_tangent_St}
\mathcal{P}_{\TangSt{\matr{X}}}(\matr{Y}) = (\matr{I}_{n}-\matr{X}\matr{X}^{\T})\matr{Y} + \matr{X}\skeww{\matr{X}^{\T}\matr{Y}}
=\matr{Y} - \matr{X}\symmo{\matr{X}^{\T}\matr{Y}}.
\end{equation}
Let $f$ be the cost function in \eqref{eq:objec_func_g}.
It follows from \eqref{eq:Riemannian-gradient-of-submanifold} and \eqref{eq:proj_tangent_St} that
\begin{equation}\label{eq:St_grad}
\ProjGrad{f}{\matr{X}} = \mathcal{P}_{\TangSt{\matr{X}}}( \nabla f(\matr{X}))
= \nabla f(\matr{X}) -\matr{X}\symmo{\matr{X}^{\T}\nabla f(\matr{X})}.
\end{equation}
For the Stiefel manifold $\St(r,n)$, the retraction can be chosen as the exponential map, QR decomposition, polar decomposition or Cayley transform; see \cite{hu2019brief} and the references therein.
It is well-known that the projection $\mathcal{P}_{\mathcal{M}}$ can be computed via the polar decomposition when $\mm = \St(r, n) $; see \cite[Lem. 5]{li2019polar} and the references therein.
\end{example}

\begin{example}[Grassmann manifold]
The \emph{Grassmann manifold} is defined as $\Gr(p,n) \eqdef \{\matr{X}\in\RR^{n\times n}: \matr{X}^{\T}=\matr{X}, \matr{X}^2=\matr{X}, \rank{\matr{X}}=p\}$  \cite{batzies2015geometric}, which is a set of projection matrices satisfying $\rank{\matr{X}}=p$.
It is also isomorphic\footnote{In this paper, for the sake of convenience in presentation, we will interchangeably use these two equivalent forms.} to $\St(p,n)/\ON{p}$, the quotient manifold of $\St(p,n)$ by the orthogonal group $\ON{p}$ \cite{bendokat2020grassmann}. 
Note that the ambient space of the Grassmann manifold $\Gr(p, n)$ is $\RR^{n \times n}$. It follows from the results obtained in \cite{bendokat2020grassmann,batzies2015geometric,sato2014optimization} that the tangent space to $\Gr(p, n)$ at $\matr{X} \in \Gr(p, n)$ can be represented by
\begin{align}
\TangGr{\matr{X}} & = \{ \matr{V} \in \symmm{\RR^{n \times n}} : \matr{V} = \matr{V} \matr{X} + \matr{X}\matr{V}\} \\
& = \{ \matr{\Omega} \matr{X} - \matr{X} \matr{\Omega}: \matr{\Omega} \in \skewww{\RR^{n \times n}} \}\\
& = \left\{ \matr{Q} \begin{bmatrix}\matr{0} & \matr{J}^{\T}\\ \matr{J} & \matr{0} \end{bmatrix} \matr{Q}^{\T} : \matr{X} = \matr{Q} \Diag{\matr{I}_p, \matr{0}}\matr{Q}^{\T} , \matr{J} \in \RR^{(n-p) \times p}, \matr{Q} \in \ON{n} \right\}.\label{eq:def-Gr-tangent-decomp}
\end{align}
The normal space to $\Gr(p, n)$ at $\matr{X}$ satisfies
\[ \NormalGr{\matr{X}} =  \{ \matr{S} - \ad_{\matr{X}}^2(\matr{S}): \matr{S} \in \symmm{\RR^{n \times n}}  \}, \]
where $ \ad_{\matr{X}}(\matr{S}) \eqdef \matr{X}\matr{S} - \matr{S} \matr{X}$.
As shown in \cite[Prop. 2.2]{sato2014optimization}, the orthogonal projection of $ \matr{Y} \in \RR^{n \times n} $ onto the tangent space to $\Gr(p, n)$ at $\matr{X}$ is
\begin{equation}\label{eq:proj_tangent_Gr}
\mathcal{P}_{\TangGr{\matr{X}}}(\matr{Y}) = 2\symmo{\matr{X}\symmo{\matr{Y}}(\matr{I}_n - \matr{X})}.
\end{equation}
Let $f$ be the cost function in \eqref{eq:objec_func_g}.
By equations \eqref{eq:Riemannian-gradient-of-submanifold} and \eqref{eq:proj_tangent_Gr}, we have
\begin{equation}\label{eq:Gr_grad}
\ProjGrad{f}{\matr{X}} = \mathcal{P}_{\TangGr{\matr{X}}}( \nabla f(\matr{X}))
= 2\symmo{\matr{X}\symmo{\nabla f(\matr{X})}(\matr{I}_n - \matr{X})}.
\end{equation}
For the Grassmann manifold $\Gr(p,n)$ in the form of the quotient manifold \( \St(p, n) / \ON{p} \), each retraction on the Stiefel manifold induces a corresponding retraction on $\Gr(p,n)$ \cite[Prop. 4.1.3]{absil2009optimization}. When the Grassmann manifold is represented as the set of projection matrices satisfying $\rank{\matr{X}}=p$, available retraction options involve utilizing QR decomposition \cite{sato2014optimization} and the exponential map \cite{bendokat2020grassmann}.
We will demonstrate later in \cref{lem:proj-Gr} that, when $\mm = \Gr(p, n)$,  the projection $\mathcal{P}_{\mathcal{M}}$ can be obtained from the eigenvalue decomposition.
\end{example}

\subsection{\L{}ojasiewicz gradient inequality}
In this subsection, we present some results about the \L{}ojasiewicz gradient inequality \cite{AbsMA05:sjo,loja1965ensembles,lojasiewicz1993geometrie,SU15:pro,Usch15:pjo}, which has been used in \cite{LUC2017globally,li2019polar,liu2019QuadraticOptimizationOrthogonality,ULC2019} and will also help us to establish the global convergence of TGP algorithms in this paper.

\begin{definition} [{\cite[Def. 2.1]{SU15:pro}}]\label{def:Lojasiewicz}
Let $\mm' \subseteq \RR^m$ be a Riemannian submanifold,
and $g: \mm' \to \RR$ be a differentiable function.
The function $g$ is said to satisfy a \emph{\L{}ojasiewicz gradient inequality} at $\vect{x} \in \mm'$, if there exist
$\varsigma>0$, $\theta\in [\frac{1}{2}, 1)$ and a neighborhood $\mathcal{U}$ in $\mm'$ of $\vect{x}$ such that for all $\vect{y}\in\mathcal{U}$, it follows that
\begin{equation}\label{eq:Lojasiewicz}
|{g}(\vect{y})-{g}(\vect{x})|^{\theta}\leq \varsigma\|\ProjGrad{g}{\vect{y}}\|.
\end{equation}
\end{definition}

\begin{lemma}[{\cite[Prop. 2.2]{SU15:pro}}]\label{lemma-SU15}
Let $\mm'\subseteq\RR^m$ be an analytic submanifold\footnote{See {\cite[Def. 2.7.1]{krantz2002primer}} or \cite[Def. 5.1]{LUC2018} for a definition of an analytic submanifold.} and $g: \mm' \to \RR$ be a real analytic function.
Then for any $\vect{x}\in \mm'$, $g$ satisfies a \L{}ojasiewicz gradient inequality \eqref{eq:Lojasiewicz} in the $\varepsilon$-neighborhood of $\vect{x}$, for
some\footnote{The values of $\varepsilon,\varsigma,\theta$ depend on the specific point in question.} $\varepsilon,\varsigma>0$ and $\theta\in [\frac{1}{2}, 1)$.
\end{lemma}

\begin{theorem}[{\cite[Thm.  2.3]{SU15:pro}}]\label{theorem-SU15}
Let $\mm'\subseteq\RR^m$ be an analytic submanifold
and
$\{\vect{x}_k\}_{k\geq 0}\subseteq\mm'$.
Suppose that $g: \mm' \to \RR$ is real analytic and, for large enough $k$,\\
(i) there exists $\phi>0$ such that
\begin{equation*}\label{eq:sufficient_descent}
{g}(\vect{x}_k)-{g}(\vect{x}_{k+1})\geq \phi\|\ProjGrad{g}{\vect{x}_k}\|\|\vect{x}_{k+1}-\vect{x}_{k}\|;
\end{equation*}
(ii) $\ProjGrad{g}{\vect{x}_k}=\vect{0}$ implies that $\vect{x}_{k+1}=\vect{x}_{k}$.\\
Then any accumulation point $\vect{x}^*$ of $\{\vect{x}_k\}_{k\geq 0}$ must be the only limit point. Furthermore, if\\
(iii) there exists \(\zeta > 0\) such that for large enough $ k $ it holds that $ \| \vect{x}_{k+1} - \vect{x}_k \| \geq \zeta \| \grad g(\vect{x}_k) \|$,\\
then the convergence speed can be estimated by
\begin{equation}\label{eq:KL_convergence_rate}
\left\|\vect{x}_k-\vect{x}^*\right\| \lesssim \begin{cases}e^{-c k} & \text { if } \theta=\frac{1}{2} \ (\text {for some }c>0);  \\ k^{-\frac{1-\theta}{2 \theta-1}} & \text { if } \frac{1}{2}<\theta<1 .\end{cases}
\end{equation}
\end{theorem}

\section{TGP algorithmic framework and a summary of the convergence results}\label{sec:TGP_alg_frame}
In this section, we first present the TGP algorithmic framework, and then recall several related algorithms from the literature, demonstrating how they can be regarded as special instances of our algorithmic framework. Following this, in \cref{table-example-3-0-9}, we summarize the convergence results established in this paper, whose detailed proofs will be presented in the subsequent sections.

\subsection{TGP algorithmic framework}\label{subsec:TGP_framework}

The general TGP algorithms proposed in \cref{subsec:contri_TGP} can be summarized as in \cref{alg:TGP}. 

\begin{algorithm}
\caption{TGP algorithm}\label{alg:TGP}
\begin{algorithmic}[1]
\STATE{{\bf Input:} starting point $\matr{X}_{0}$.}
\STATE{{\bf Output:} the iterates $\matr{X}_{k}$, $k\geq 1$.}
\FOR{$k=0,1,2,\cdots$, until a stopping criterion is satisfied}
\STATE Compute
\begin{equation}\label{eq:updat_schem_TGP}
\matr{Y}_k(\tau) = \matr{X}_{k}-\tau\matr{H}_{k},
\end{equation}
where $\matr{H}_{k}$ is selected as in \eqref{eq:Trans_RGP} (or equivalently, \eqref{eq:Trans_EGP}). 
\STATE Choose the stepsize $\tau_{k}>0$ along the curve defined by
\begin{equation}
\matr{Z}_k(\tau)=\mathcal{P}_{\mathcal{M}}(\matr{Y}_k(\tau)).
\end{equation}
\STATE Update
\begin{equation}\label{eq:upd_pi_ytau}
\matr{X}_{k+1}=\matr{Z}_k(\tau_{k}).
\end{equation}
\ENDFOR
\end{algorithmic}
\end{algorithm}

We now provide some explanation and motivation for the components of the search direction \( \matr{H}_k \) in \eqref{eq:Trans_RGP} (or equivalently, \eqref{eq:Trans_EGP}), including the scaling matrices $\matr{L}(\matr{X}_{k})$, $\matr{R}(\matr{X}_{k})$, and the additional normal vectors \( \matr{N}(\matr{X}_k) \) in \eqref{eq:Trans_RGP} and \( \matr{N}'(\matr{X}_k) \) in \eqref{eq:Trans_EGP}.

\begin{itemize}
\item Scaling matrices: In \eqref{eq:Trans_EGP}, the Euclidean gradient \(\nabla f(\matr{X}_k)\) is multiplied by \( \matr{L}(\matr{X}_k) \) and \( \matr{R}(\matr{X}_k) \) on both sides, which is similar to the use of the scaling matrices in the \emph{scaled gradient projection} method for optimization over a closed convex set \cite{bonettini2015new,bonettini2008scaled,crisci2022convergence}. Therefore, we refer to the matrices \( \matr{L}(\matr{X}_k) \) and \( \matr{R}(\matr{X}_k) \) as the \emph{scaling matrices} in this paper. 
Certainly, from the perspective of \eqref{eq:Trans_RGP}, the scaling matrices can also be viewed as a transformation applied to the Riemannian gradient \( \grad f(\matr{X}_k) \). 
In the scaled gradient projection method over a closed convex set \cite{bonettini2015new,bonettini2008scaled,crisci2022convergence}, a common assumption is that the scaling matrix is positive definite and possesses eigenvalues that are uniformly bounded, ensuring that the iterates move towards a proper search direction. Inspired by these considerations, we introduce the following assumptions regarding the scaling matrices in \cref{alg:TGP}.
\begin{assumption}\label{asp:H-tilde-grad-equiv}
\mbox{} 
\begin{itemize}[leftmargin=2\parindent]
\item[{\rm ({\bf A1})}] $\matr{L}(\matr{X}_k)\matr{V}\matr{R}(\matr{X}_k)\in\TangM{\matr{X}_k}$ for all $ \matr{V} \in \TangM{\matr{X}_k}$ and $k \in \NN$.
\item[{\rm ({\bf A2})}] there exist positive constants \(\upsilon, \varpi > 0\) such that for all $k \in \NN$, it holds that
\begin{align}
\upsilon  \| \grad f(\matr{X}_k) \| & \leq  \| \matr{L}(\matr{X}_k)\grad f(\matr{X}_k) \matr{R}(\matr{X}_k)\| \leq \varpi \| \grad f(\matr{X}_k) \|,\\
\upsilon \| \grad f(\matr{X}_k) \|^2 & \leq \langle \grad f(\matr{X}_k), \matr{L}(\matr{X}_k)\grad f(\matr{X}_k) \matr{R}(\matr{X}_k) \rangle \leq \varpi \| \grad f(\matr{X}_k) \|^2.
\end{align}
\end{itemize}
\end{assumption}

As will be shown in \cref{lem:proper_c_k}, these assumptions ensure that the expression in \eqref{eq:c_tilde_hat_transf} remains in the tangent space, and is scale-equivalent to and directionally aligned with the Riemannian gradient. 
We present a broad class of choices satisfying \cref{asp:H-tilde-grad-equiv} for \( \matr{L}(\matr{X}_k) \) and \( \matr{R}(\matr{X}_k) \) on the Stiefel and Grassmann manifolds, as detailed in \cref{exa:LR_St} and \cref{exa:LR_Gr}. 
These choices include several existing algorithms from the literature (as detailed in \cref{subsec:special_literature}) as special cases. Moreover, we will introduce a new algorithm, TGP-A-Eigen \eqref{eq:H-TGP-eigen}, by specifying the scaling matrices later in \cref{subsec:eigenvalue-problem}. 
\item Additional normal vector: The introduction of the additional normal vector is essential for enhancing the generality of the TGP framework, enabling it to encompass the RGD and EGP algorithms on the Stiefel manifold, as well as the Shifted PM method \eqref{eq:iter_power_method} on the unit sphere. 
Moreover, we will see that this additional normal vector can also lead to new algorithms, such as  TGP-$\ast$-E/R \eqref{eq:H_TGP_R_E}, which perform better in certain cases, as shown in \cref{sec:numer_exper}.
\end{itemize}

In \cref{alg:TGP}, with various specific choices of the aforementioned components, it includes several existing algorithms from the literature as special cases; see \cref{subsec:special_literature} for more details.
It can be seen that there exists a partial overlap between the TGP algorithms and the RetrLS algorithms \eqref{eq:iter_retr}.
If \( \matr{H}_k \) is chosen as a tangent vector to \( \mm \) at \( \matr{X}_k \), then the TGP algorithms reduce to a special class of RetrLS algorithms \eqref{eq:iter_retr} using the projection as a retraction.
We further observe that the classical EGP algorithm \eqref{eq:iter_prjec} and the Shifted PM method \eqref{eq:iter_power_method} are both special cases of the TGP algorithms, and TGP algorithms fall within the class of ProjLS algorithms \eqref{eq:iter_prjec-h}. 
There also exists an overlap between EGP and the Shifted PM. For example, by setting the shift $s=1$ in \eqref{eq:iter_power_method}, the scheme reduces to EGP on the unit sphere with a fixed stepsize of 1.
See \cref{fig:compare-retra-proj-alg} for more details.
\begin{figure}[htbp]
\centering
\begin{tikzpicture}

\node [draw,
ellipse,
minimum width =7cm,
minimum height =6cm,
text width = 3.5cm,
align = left
] at (-2.5,0) {};
\node[align=center] (lab) at (-4.15, 0.85) {RetrLS \\algorithms \eqref{eq:iter_retr}};

\node [draw,
ellipse,
minimum width = 9.5cm,
minimum height =5.6cm,
] at (2.25,0){};
\node[align=center] (lab) at (2.5, 2) {TGP algorithms \eqref{eq:Trans_RGP} or \eqref{eq:Trans_EGP}: \\  $ \matr{H}_k = \matr{L}(\matr{X}_k)\grad f(\matr{X}_k)\matr{R}(\matr{X}_k) $ \\ $ + \matr{N}(\matr{X}_k)$};

\node [draw,
ellipse,
minimum width = 12cm,
minimum height =7cm,
] at (3.5,0){};
\node[align=center] (lab) at (8.2, 0.1) {ProjLS\\ algorithms \eqref{eq:iter_prjec-h}};

\node [draw,
ellipse,
minimum width = 6cm,
minimum height =2.5cm,
] at (-2.5, -1){};
\node[align=center] (lab) at (-3.8, -1.0) {RGD algorithms \eqref{eq:iter_retr_RGD}: \\ $\matr{H}_k = \grad f(\matr{X}_k)$};

\node [draw,
ellipse,
minimum width = 3.3cm,
minimum height =2.5cm,
] at (2.35, -1.5){};
\node[align=center] (lab) at (2.3, -1.5) {EGP algorithm \eqref{eq:iter_prjec}: \\ $\matr{H}_k = \nabla f(\matr{X}_k)$};

\node [draw,
ellipse,
minimum width = 4.2cm,
minimum height =2.5cm,
] at (4.3, 0.2){};
\node[align=center] (lab) at (4.3, 0.2) { Shifted PM \eqref{eq:iter_power_method}: \\ $\matr{H}_k = \nabla f(\matr{X}_k) + (1-s)\matr{X}_k,$\\ $ \mm  = \mathbb{S}^{n-1} = \St(1, n)$};

\node[align=center] (lab) at (-0.73, 0.9) {$ \matr{H}_k  \in \TangM{\matr{X}_k},$ \\ $ \retr_{\matr{X}_{k}}(\cdot) = \mathcal{P}_{\mathcal{M}}(\matr{X}_{k} + \cdot)$};
\end{tikzpicture}
\caption{Relationships among TGP algorithms \eqref{eq:Trans_RGP} (or equivalently, \eqref{eq:Trans_EGP}), RetrLS algorithms \eqref{eq:iter_retr} and ProjLS algorithms \eqref{eq:iter_prjec-h} on $\mm$}
\label{fig:compare-retra-proj-alg}
\end{figure}

As will be discussed in \cref{rem:new-metric}, under \cref{asp:H-tilde-grad-equiv}, if the normal component is omitted (i.e., \( \matr{H}_k = \tilde{\matr{H}}_k \)), \( \matr{H}_k \) coincides with the Riemannian gradient under a new Riemannian metric. In this sense, TGP without the normal component can be viewed as a special instance of a \emph{Riemannian preconditioned method} \cite{gao2025optimization,dong2022new,gao2024riemannian} or, equivalently, of an \emph{inexact Riemannian gradient method} \cite{absil2009optimization}.
Moreover, our framework provides a unified way to construct a broader class of Riemannian metrics, extending those previously studied on the Stiefel and Grassmann manifolds; see \cref{exa:LR_St,exa:LR_Gr} for more details. 
More importantly, the TGP framework admits much more general search directions than purely tangential or simply scaled gradients. For instance, we may choose  
$\matr{H}_k = p_k(\grad f(\matr{X}_k)) + \matr{N}(\matr{X}_k)$, where \( p_k: \TangM{\matr{X}_k} \to \TangM{\matr{X}_k} \) is a linear or nonlinear mapping, such as $p_k(\cdot) = \operatorname{Hess}^{-1}(\matr{X}_k)(\cdot)$, where $\operatorname{Hess}(\cdot)$ represents the Riemannian Hessian \cite{absil2009optimization}. 
Assumption~({\bf A2}) still holds when \( \matr{X}_k \) 
lies in a region where \( f \) is locally geodesically strongly convex, and our convergence analysis can be correspondingly extended to such cases. 

At the end of this subsection, we would like to remark that, as mentioned earlier, several variants of the EGP algorithm \eqref{eq:iter_prjec} have also been extensively studied for the optimization problem $\min_{\vect{x}\in\mathcal{X}} f(\vect{x})$, where the feasible region $\mathcal{X}\subseteq\RR^{m}$ is a closed convex subset, and the convergence properties were established utilizing the properties of the projection onto $\mathcal{X}$  \cite{nesterov2018LecturesConvexOptimization,nocedal2006NumericalOptimization,beck2017FirstOrderMethodsOptimization}. For example, the \emph{scaled gradient projection} method over a closed convex set was developed in \cite{bonettini2015new,bonettini2008scaled,crisci2022convergence}, where the update scheme is given by $\vect{x}_{k+1} = \mathcal{P}_{\mathcal{X}}(\vect{x}_k - \tau_k \matr{D}_k \nabla f(\vect{x}_k))$ and $ \matr{D}_k $ is the \emph{scaling matrix}, which is usually assumed to be positive definite. If the feasible set $ \mm $ in problem \eqref{eq:objec_func_g} is non-convex, as is the case with the Stiefel manifold and Grassmann manifold we focus on in this paper, the existing convergence analysis for the closed convex constraint optimization cannot be directly applied to the ProjLS algorithms \eqref{eq:iter_prjec-h}.

\subsection{Constructing $\matr{L}(\matr{X}_k)$ and $\matr{R}(\matr{X_k})$ satisfying \cref{asp:H-tilde-grad-equiv}}\label{subsec:assumption-L-R}

The following lemma helps us derive the conditions under which our examples meet the assumption ({\bf A2}) based on the eigenvalues of the scaling matrices.

\begin{lemma}\label{lem:multiplied-by-symm}
Let \( \matr{L} \in \symmm{\RR^{n \times n}} \) and \( \matr{R} \in \symmm{\RR^{r \times r}} \) be positive semi-definite matrices. Then for all \( \matr{T} \in \RR^{n \times r} \), we have
\begin{align*}
\lambda_{\min}(\matr{L}) \lambda_{\min}(\matr{R}) \| \matr{T} \| &\leq \| \matr{L}\matr{T}\matr{R} \| \leq \lambda_{\max}(\matr{L})\lambda_{\max}(\matr{R}) \| \matr{T} \|,\\
\lambda_{\min}(\matr{L}) \lambda_{\min}(\matr{R}) \| \matr{T} \|^2 &\leq \langle \matr{T}, \matr{L}\matr{T}\matr{R} \rangle \leq \lambda_{\max}(\matr{L})\lambda_{\max}(\matr{R}) \| \matr{T} \|^2.
\end{align*}
\end{lemma}
\begin{proof}
It suffices to prove the left sides of the above two inequalities, as the proofs for the right sides can be demonstrated in a similar manner.
Let $ \matr{L} = \matr{Q}_{\matr{L}}^{\T}\matr{\Lambda}_{\matr{L}}\matr{Q}_{\matr{L}} $ and $ \matr{R} = \matr{Q}_{\matr{R}}^{\T}\matr{\Lambda}_{\matr{R}}\matr{Q}_{\matr{R}} $  be the spectral decompositions, with \( \matr{\Lambda}_{\matr{L}} = \Diag{\lambda_{1}, \lambda_{2} , \ldots , \lambda_{n}} \) and \( \matr{\Lambda}_{\matr{R}} = \Diag{\lambda'_{1}, \lambda'_{2} , \ldots , \lambda'_{r}} \).
Let \( \matr{T}' = \matr{Q}_{\matr{L}}\matr{T}\matr{Q}_{\matr{R}}^{\T} \). Then \( \| \matr{T}'\| = \|\matr{T} \| \) and
\( \| \matr{L} \matr{T}\matr{R} \| = \| \matr{Q}_{\matr{L}}^{\T} \matr{\Lambda}_{\matr{L}}\matr{T}'\matr{\Lambda}_{\matr{R}} \matr{Q}_{\matr{R}}\| = \| \matr{\Lambda}_{\matr{L}}\matr{T}'\matr{\Lambda}_{\matr{R}} \| \). Note that \(\lambda_{i}, \lambda'_{j} \geq 0\) for all \( 1 \leq i \leq n, 1 \leq j \leq r \). It follows that
\begin{equation*}
\| \matr{L}\matr{T}\matr{R} \| = \| \matr{\Lambda}_{\matr{L}}\matr{T}'\matr{\Lambda}_{\matr{R}} \| \geq \lambda_{\min}(\matr{L})\lambda_{\min}(\matr{R}) \| \matr{T}' \| = \lambda_{\min}(\matr{L})\lambda_{\min}(\matr{R}) \| \matr{T} \|.
\end{equation*}
Similarly, we have that
\begin{align*}
\langle \matr{T}, \matr{L} \matr{T}\matr{R} \rangle &= \langle \matr{T}', \matr{\Lambda}_{\matr{L}}\matr{T}'\matr{\Lambda}_{\matr{R}} \rangle
= \sum_{i,j} \lambda_{i} \lambda'_{j} (\matr{T}'_{ij})^2 \\
&\geq \lambda_{\min}(\matr{L})\lambda_{\min}(\matr{R}) \| \matr{T}' \|^2 =  \lambda_{\min}(\matr{L})\lambda_{\min}(\matr{R}) \| \matr{T} \|^2.
\end{align*}
The proof is complete.
\end{proof}

Now we construct two examples on $\St(r, n)$ and $\Gr(p, n)$, respectively.
It will be seen that the classes of scaling matrices \( \matr{L}(\matr{X}_k) \) and \( \matr{R}(\matr{X}_k) \) in these two examples satisfy the assumptions ({\bf A1}) and ({\bf A2}).

\begin{example}[A class of $\matr{L}(\matr{X})$ and $\matr{R}(\matr{X})$ on $\St(r, n)$]\label{exa:LR_St}
Let $\matr{E} \in \symmm{\RR^{r \times r}} $, $\matr{F} \in \symmm{\RR^{(n-r) \times (n-r)}}$, \(\mu \in \RR\). Define 
\begin{equation}\label{eq:LR_St}
\matr{L}(\matr{X}) = \matr{I}_n + \mu \matr{X}\matr{E}\matr{X}^{\T} + \matr{X}_{\perp}\matr{F}\matr{X}_{\perp}^{\T}, \quad \matr{R}(\matr{X}) = \matr{E}.
\end{equation}
Then, for $\matr{V} \in \TangSt{\matr{X}}$ in the form \eqref{def-St-tangent-space-decomp}, we always have
\[ \matr{L}(\matr{X})\matr{V}\matr{R}(\matr{X}) = \matr{X} (\matr{A} + \mu\matr{E}^{\T}\matr{A}\matr{E}) + \matr{X}_{\perp}(\matr{B} + \matr{F}\matr{B}\matr{E}).  \]
Since $\matr{E}^{\T}\matr{A}\matr{E} \in \skewww{\RR^{r \times r}}$ for $\matr{A} \in \skewww{\RR^{r \times r}}$, we have $\matr{L}(\matr{X})\matr{V}\matr{R}(\matr{X}) \in \TangSt{\matr{X}}$ for all $\matr{V} \in \TangSt{\matr{X}}$. Therefore, the assumption ({\bf A1}) is always satisfied.
Noting that  $[\matr{X}, \matr{X}_{\perp}] \in \ON{n}$ and $\matr{L}(\matr{X}) = \matr{I}_n + [\matr{X}, \matr{X}_{\perp}] \Diag{ \mu\matr{E}, \matr{F}} [\matr{X}, \matr{X}_{\perp}]^{\T}$, we have
$$\lambda_{\min}(\matr{L}(\matr{X})) = 1 + \min \{  \lambda_{\min}(\mu\matr{E}), \lambda_{\min}(\matr{F}) \},\ \ \lambda_{\max}(\matr{L}(\matr{X})) = 1 + \max \{  \lambda_{\max}(\mu\matr{E}), \lambda_{\max}(\matr{F}) \}.$$
It follows from \cref{lem:multiplied-by-symm} that the assumption ({\bf A2}) is satisfied if
\[ \left(1+\min\{ \lambda_{\min}(\mu\matr{E}), \lambda_{\min}(\matr{F})\}\right)\lambda_{\min}(\matr{E}) \geq \upsilon,\ \ \left(1+\max\{ \lambda_{\max}(\mu\matr{E}), \lambda_{\max}(\matr{F})\}\right)\lambda_{\max}(\matr{E}) \leq \varpi. \]
\end{example}

\begin{remark}\label{rem:LR_St}
Let $f$ be the cost function in \eqref{eq:objec_func_g} and $\matr{X}\in\St(r,n)$.
As a generalization of the Riemannian gradient $\grad f(\matr{X})$, a tangent vector \( \matr{D}_{\rho}(\matr{X}) \in \TangSt{\matr{X}}\) was introduced in \cite{jiang2015framework,liu2019QuadraticOptimizationOrthogonality} as follows:
\begin{align}
\matr{D}_{\rho}(\matr{X}) &\eqdef 2\rho \left(\nabla f(\matr{X}) - \matr{X} \nabla f(\matr{X})^{\T}\matr{X}\right) + (1-2\rho) \left(\nabla f(\matr{X}) - \matr{X} \matr{X}^{\T}\nabla f(\matr{X})\right)\notag\\
& = \nabla f(\matr{X}) - \matr{X}\left(2\rho \nabla f(\matr{X})^{\T}\matr{X}+(1-2\rho)\matr{X}^{\T}\nabla f(\matr{X})\right), \label{eq:d_rho_X}
\end{align}
for $\rho\in\RR$.
In this paper, we always assume that $\rho>0$.
It is clear that $\matr{D}_{1/4}(\matr{X})=\grad f(\matr{X})$.
By direct calculations, it follows that the above $\matr{D}_{\rho}(\matr{X})$, satisfies
\begin{align*}
\matr{D}_{\rho}(\matr{X}) & = \left(\matr{I}_n+ (4\rho - 1)\matr{X}\matr{X}^{\T}\right) \grad f(\matr{X}) \\
& = \left(\matr{I}_n+ (4\rho - 1)\matr{X}\matr{X}^{\T}\right) \nabla f(\matr{X}) - 4\rho\matr{X}\symmo{\matr{X}^{\T}\nabla f(\matr{X})}.
\end{align*}
Therefore, when \( \matr{E} =  \matr{I}_r \), \( \matr{F} = \matr{0}_{(n-r) \times (n-r)} \) and \(\mu = 4\rho-1\), \cref{exa:LR_St} includes $\matr{D}_{\rho}(\matr{X})$ as a special case.

In addition to the above $\matr{D}_{\rho}(\matr{X})$, \cref{exa:LR_St} further introduces a variety of new search directions due to the flexibility in choosing parameters \( \matr{E} \), \( \matr{F} \), and \(\mu\), including cases where \( \matr{F} \) is non-zero. In \cref{subsec:eigenvalue-problem,subsec:tgp_diffe_scale}, we will specify parameter choices that yield a new search direction, which outperforms classical search directions in certain cases.
\end{remark}

\begin{example}[A class of $\matr{L}(\matr{X})$ and $\matr{R}(\matr{X})$ on $\Gr(p, n)$]\label{exa:LR_Gr}
Let
\[ \matr{L}(\matr{X}) =  \matr{R}(\matr{X}) = \matr{Q} \Diag{\matr{G}_1, \matr{G}_2}\matr{Q}^{\T},   \]
where $\matr{G}_1 \in \symmm{\RR^{p \times p}}$, $\matr{G}_2 \in \symmm{\RR^{(n-p) \times (n-p)}}$ and the orthogonal matrix $ \matr{Q} \in \ON{n} $ satisfies that $\matr{X} = \matr{Q}\Diag{\matr{I}_{p}, \matr{0}_{(n-p) \times (n-p)}}\matr{Q}^{\T}$.
Then, for all $\matr{V} \in \TangGr{\matr{X}}$ in the form of \eqref{eq:def-Gr-tangent-decomp}, we have
\[ \matr{L}(\matr{X})\matr{V}\matr{R}(\matr{X}) = \matr{Q} \begin{bmatrix}\matr{0}_{p \times p}&\matr{G}_1\matr{J}^{\T}\matr{G}_2\\\matr{G}_2\matr{J}\matr{G}_1&\matr{0}_{(n-p) \times (n-p)}\end{bmatrix} \matr{Q}^{\T}\in \TangGr{\matr{X}}. \]
Therefore, the assumption ({\bf A1}) is satisfied. Note that the eigenvalues of \( \matr{L}(\matr{X}) \) are the union of the eigenvalues of \( \matr{G}_1 \) and \( \matr{G}_2 \).
It follows from \cref{lem:multiplied-by-symm} that the assumption ({\bf A2}) is satisfied if \(\min\{ \lambda_{\min}(\matr{G}_1), \lambda_{\min}(\matr{G}_2)\} \geq \sqrt{\upsilon} \) and \(\max\{ \lambda_{\max}(\matr{G}_1), \lambda_{\max}(\matr{G}_2)\} \leq \sqrt{\varpi} \).
\end{example}

\subsection{Orthogonal projection of $\matr{H}_k$}\label{subsec:orthgonal-projection-H}
In the TGP algorithmic framework, we denote
\begin{align*}
\matr{L}_k&\eqdef \matr{L}(\matr{X}_{k}),\ \  \matr{R}_{k}\eqdef\matr{R}(\matr{X}_{k}),\ \  \matr{N}_k\eqdef\matr{N}(\matr{X}_k), \\
\tilde{\matr{H}}_k&\eqdef\mathcal{P}_{\TangM{\matr{X}_k}}(\matr{H}_k),\ \  \hat{\matr{H}}_k \eqdef\mathcal{P}_{\NormalM{\matr{X}_k}}(\matr{H}_k)
\end{align*}
for simplicity.
We now show the following relationship between $ \tilde{\matr{H}}_k $ and $ \grad f(\matr{X}_{k}) $ under \cref{asp:H-tilde-grad-equiv}.

\begin{lemma}\label{lem:proper_c_k}
(i) If $\matr{L}_k$ and $\matr{R}_k$ satisfy the assumption ({\bf A1}), then
\begin{equation}\label{eq:H-projection}
\tilde{\matr{H}}_k = \matr{L}_k \grad f(\matr{X}_k)\matr{R}_k.
\end{equation}
(ii) If $\matr{L}_k$ and $\matr{R}_k$ satisfy \cref{asp:H-tilde-grad-equiv}, then
\begin{align}
\upsilon \| \grad f(\matr{X}_k) \| &\leq \| \tilde{\matr{H}}_k \| \leq \varpi \| \grad f(\matr{X}_k) \|, \label{eq:H-grad-norm-equiv}\\
\upsilon \| \grad f(\matr{X}_k) \|^2 &\leq \langle \matr{\tilde{H}}_k, \grad f(\matr{X}_k) \rangle = \langle \matr{H}_k, \grad f(\matr{X}_k) \rangle \leq \varpi \| \grad f(\matr{X}_k) \|^2. \label{eq:H-grad-product}
\end{align}
\end{lemma}

\begin{proof}
(i) For any normal vector $ \matr{W} \in \NormalM{\matr{X}_k} $ and tangent vector $ \matr{V} \in \TangM{\matr{X}_k}  $, we have that
\[ \langle \matr{L}_k\matr{W}\matr{R}_k, \matr{V} \rangle = \langle \matr{W}, \matr{L}_k^{\T}\matr{V}\matr{R}_k^{\T} \rangle =  \langle \matr{W}, \matr{L}_k\matr{V}\matr{R}_k \rangle = 0,\]
where the assumption ({\bf A1}) is used in the last equality. It follows that $ \matr{L}_k\matr{W}\matr{R}_k \in \NormalM{\matr{X}_k} $ for all $ \matr{W} \in \NormalM{\matr{X}_k} $. Combining this property with the fact that \( \nabla f(\matr{X}_k) - \grad f(\matr{X}_k) \in \NormalM{\matr{X}_k} \), we have \( \matr{L}_k(\nabla f(\matr{X}_k) - \grad f(\matr{X}_k))\matr{R}_k \in \NormalM{\matr{X}_k} \).
Then, it follows from the fact that the projection onto the linear subspace $ \TangM{\matr{X}_k} $ is a linear mapping that
\begin{align*}
\tilde{\matr{H}}_k & =  \mathcal{P}_{\TangM{\matr{X}_k}}(\matr{L}_k \nabla f(\matr{X}_k)\matr{R}_k + \matr{N}_k)\\
& = \mathcal{P}_{\TangM{\matr{X}_k}}(\matr{L}_k \grad f(\matr{X}_k)\matr{R}_k) +  \mathcal{P}_{\TangM{\matr{X}_k}}(\matr{L}_k (\nabla f(\matr{X}_k) - \grad f(\matr{X}_k)) \matr{R}_k + \matr{N}_k)\\
& \stackrel{(a)}{=} \mathcal{P}_{\TangM{\matr{X}_k}}(\matr{L}_k \grad f(\matr{X}_k)\matr{R}_k)
\stackrel{(b)}{=}\matr{L}_k \grad f(\matr{X}_k)\matr{R}_k,
\end{align*}
where \((a)\) is due to $ \matr{L}_k(\nabla f(\matr{X}_k) - \grad f(\matr{X}_k))\matr{R}_k$, $ \matr{N}_k \in \NormalM{\matr{X}_k} $ and $ (b) $ follows from the assumption ({\bf A1}).
(ii) can be easily obtained from the assumption ({\bf A2}) and \eqref{eq:H-projection}.
The proof is complete.
\end{proof}

\begin{remark}
Upon satisfying the assumption ({\bf A1}), based on  \cref{fig:compare-retra-proj-alg}, we now more explicitly show the connections between the TGP algorithms and the RetrLS algorithms \eqref{eq:iter_retr} (see \cref{fig:exte_rima_opt} as an example on $\St(1,2)$). In fact, it follows from \cref{lem:proper_c_k}(i) that
\begin{align}\label{eq:c_tilde_hat_sum}
\matr{H}_k = \tilde{\matr{H}}_k + \hat{\matr{H}}_k = \matr{L}_k \grad f(\matr{X}_k)\matr{R}_k + \hat{\matr{H}}_k,
\end{align}
where $\tilde{\matr{H}}_k\in\TangM{\matr{X}_k}$ and $\hat{\matr{H}}_k\in\NormalM{\matr{X}_k}$.
Note that $ \retr_{\matr{X}}(\matr{V}) \eqdef \mathcal{P}_{\mm}(\matr{X} + \matr{V}) $ for $ \matr{V} \in \TangM{\matr{X}} $ forms a retraction on $ \mm $ \cite{absil2012projection}.
If $\hat{\matr{H}}_k=\vect{0}$ (equivalently, $\matr{H}_k\in\TangM{\matr{X}_k}$), the update scheme \eqref{eq:updat_schem_TGP} in \cref{alg:TGP} reduces to the RetrLS algorithms in \cite{absil2009optimization}.
In this sense, TGP algorithms can also be viewed as an extension of the RetrLS algorithms on $ \mm $ (employing the projection as a form of retraction).
\end{remark}

\begin{figure}
\centering
\includegraphics[width=0.7\linewidth]{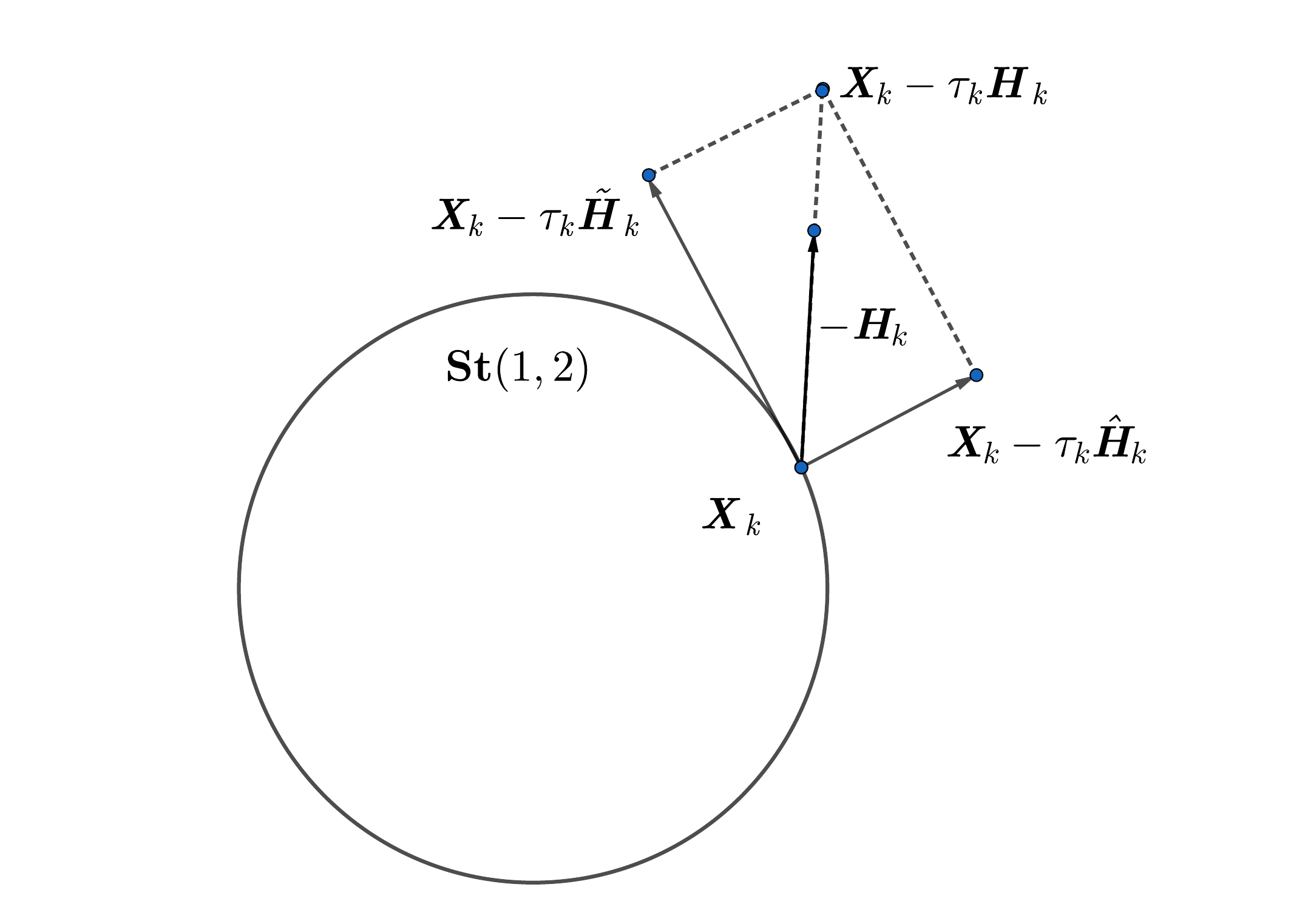}
\caption{An illustration of the update scheme \eqref{eq:updat_schem_TGP} in \cref{alg:TGP}}\label{fig:exte_rima_opt}
\end{figure}

\begin{remark}\label{rem:new-metric}
It was demonstrated in \cite[Sec. 4.1]{wen2013feasible} that the tangent vector \( \matr{D}_{1/2}(\matr{X}) \in \TangSt{\matr{X}}\), as defined in \eqref{eq:d_rho_X}, can be interpreted as the Riemannian gradient of $ f(\matr{X}) $ when considering a distinct Riemannian metric, namely the \emph{canonical} metric. In \cref{alg:TGP}, if the scaling matrices $ \matr{L}(\matr{X}) $ and \( \matr{R}(\matr{X}) \) satisfy \cref{asp:H-tilde-grad-equiv} and vary smoothly on $ \mm $, we can define a new Riemannian metric by
\[ \langle \matr{V}, \matr{V}' \rangle_{\matr{X}} \eqdef \langle \matr{V}, \matr{L}(\matr{X})^{-1}\matr{V}' \matr{R}(\matr{X})^{-1} \rangle \text{ for all } \matr{V}, \matr{V}' \in \TangM{\matr{X}}.  \]
Then the Riemannian gradient under this new metric is \(  \matr{L}(\matr{X}) \grad f(\matr{X}) \matr{R}(\matr{X}) \) by definition.
In this scenario, according to \cref{lem:proper_c_k}, the tangent component \(\tilde{\matr{H}}_k\) can also be interpreted as a Riemannian gradient of \( f(\matr{X}) \) at \( \matr{X}_k \) under the above new Riemannian metric.
\end{remark}

\subsection{Related algorithms in the literature}\label{subsec:special_literature}
As stated in \cref{subsec:TGP_framework}, \cref{alg:TGP} comprises several existing algorithms from the literature as special cases.
To begin, let us first recall these algorithms within the context of the Stiefel manifold, an extensively examined scenario.
The quantity \( \matr{D}_{\rho}(\matr{X})\) defined in \eqref{eq:d_rho_X} possesses the following properties.
\begin{lemma}[{\cite[Prop. 2, Eq. (28)]{liu2019QuadraticOptimizationOrthogonality}}]\label{lem-D-rho-prop}
(i) The tangent vector \( \matr{D}_{\rho}(\matr{X})\) is equivalent to the Riemannian gradient $ \grad f(\matr{X})$:
\[ \frac{1}{2} \min \left\{1, \frac{1}{2\rho}\right\} \cdot\left\|\matr{D}_\rho(\matr{X})\right\| \leq\|\operatorname{grad} f(\matr{X})\| \leq \max \left\{1, \frac{1}{2\rho}\right\} \cdot\left\|\matr{D}_\rho(\matr{X})\right\| .  \]\\
(ii) The tangent vector \( \matr{D}_{\rho}(\matr{X})\) satisfies that
\[ \left\langle\nabla g\left(\matr{X}\right), \matr{D}_\rho\left(\matr{X}\right)\right\rangle = \left\langle\grad g\left(\matr{X}\right), \matr{D}_\rho\left(\matr{X}\right)\right\rangle  \geq \min\left\{ \rho, \frac{1}{4\rho}, \frac{1}{4\rho^{2}} \right\} \cdot\left\|\matr{D}_\rho\left(\matr{X}\right)\right\|^2.  \]
\end{lemma}

Note that \( \matr{D}_{\rho}(\matr{X}) = \left(\matr{I}_n+ (4\rho - 1)\matr{X}\matr{X}^{\T}\right) \nabla f(\matr{X}) - 4\rho\matr{X}\symmo{\matr{X}^{\T}\nabla f(\matr{X})} \).  When the scaling matrices \( \matr{L}(\matr{X}) = \matr{I}_n + (4 \rho - 1) \matr{X} \matr{X}^\top \), \( \matr{R}(\matr{X}) = \matr{I}_r \), and the normal vector \( \matr{N}(\matr{X}) =  - 4\rho\matr{X}\symmo{\matr{X}^{\T}\nabla f(\matr{X})} \) in \eqref{eq:Trans_EGP}, the search direction \( \matr{H}_k \) simplifies to \( \matr{D}_{\rho}(\matr{X}_k) \).
Thus, the projection-based algorithms using \( \matr{D}_{\rho}(\matr{X}) \) as the search direction can also be subsumed under the TGP algorithmic framework on the Stiefel manifold.
We now provide a summary of the algorithms specifically designed for the Stiefel manifold from the literature, along with their convergence analysis results.
\begin{itemize}
\item In \cite{liu2019QuadraticOptimizationOrthogonality}, when $\mm=\St(r, n)$ and $\matr{H}_{k} = \matr{D}_{\rho}(\matr{X}_{k})$, the weak convergence and global convergence of \cref{alg:TGP} were established\footnote{The algorithm in \cite{liu2019QuadraticOptimizationOrthogonality} is based on retraction, unlike \cref{alg:TGP}, which only considers projection.} using the Armijo-type stepsize.
If $\rho=\frac{1}{4}$, we have $\matr{H}_{k}=\matr{D}_{1/4}(\matr{X}_{k}) = \grad f(\matr{X}_{k})$, and it is then reduced to the RGD algorithm\footnote{It was originally called Riemannian gradient ascent in \cite{sheng2022riemannian}, and was used to solve a maximization problem.} studied in \cite{sheng2022riemannian}.
\item In \cite{oviedo2019non}, when $\mm=\St(r, n)$, the weak convergence of \cref{alg:TGP} was established with
\begin{align}
\matr{H}_{k}&=(\alpha+\beta) \matr{D}_{\alpha/2(\alpha+\beta)}(\matr{X}_{k})\notag \\&=
\alpha \left(\nabla f(\matr{X}) - \matr{X} \nabla f(\matr{X})^{\T}\matr{X}\right) + \beta \left(\nabla f(\matr{X}) - \matr{X} \matr{X}^{\T}\nabla f(\matr{X})\right)
\end{align}
where $\alpha > 0, \beta \geq 0$. If $\alpha=1$ and $\beta=0$, we have\footnote{In the general case, \cref{alg:TGP} does not encompass the algorithm in \cite{oviedo2019scaled,oviedo2021two} as a special case since $\matr{D}(\mu,\tau)$ in \cite[Eq. (13)] {oviedo2019scaled} and \cite[Eq. (8)]{oviedo2021two} depends on the stepsize $\tau$. However, if $\mu=0$, then $\matr{D}(\mu,\tau)=\matr{I}_{n}$, and this algorithm falls within the framework of \cref{alg:TGP}.} $\matr{H}_{k}=\matr{D}_{1/2}(\matr{X}_{k})$ \cite{oviedo2019scaled,oviedo2021two}.
As in \cite{wen2013feasible}, the Armijo-type stepsize and nonmonotone search with Barzilai–Borwein stepsize were both used in \cite{oviedo2019scaled,oviedo2021two,oviedo2019non}.
\item In \cite{chen2009tensor,gao2018new,yang2019epsilon,hu2019LROAT,li2019polar}, when $\mm=\St(r, n)$ and $\matr{H}_{k}=\nabla f(\matr{X}_{k})$, \cref{alg:TGP} was applied to the tensor approximations and electronic structure calculations, and the convergence properties were established as well.
This is the simplest form of the EGP algorithm with $\nabla f(\matr{X}_{k})$ as the search direction (update scheme \eqref{eq:iter_prjec}).
A fixed stepsize was used in \cite{gao2018new,yang2019epsilon,li2019polar}, while an adaptive one was used in \cite{hu2019LROAT}.
\item The \emph{shifted power method} (Shifted PM) is a classical algorithm for computing the eigenvalue of largest modulus and its associated eigenvector of a matrix \cite[Sec. 4.1.2]{saad2011numerical}. 
Tensor extensions developed in \cite{kolda2011shifted,kolda2014adaptive,li2019polar} are proved to exhibit weak convergence, while global convergence can be established under suitable conditions. 
For our problem \eqref{eq:objec_func_g} with $\mm = \mathbb{S}^{n-1} = \St(1, n) $, the update scheme of the Shifted PM can be written as: 
\begin{equation}\label{eq:iter_power_method}
\matr{X}_{k+1} = \mathcal{P}_{\St(1, n)}(- \nabla f(\matr{X}_k) + s \matr{X}_k) = \mathcal{P}_{\St(1, n)}(\matr{X}_k - (\nabla f(\matr{X}_k) + (1-s)\matr{X}_k)),
\end{equation}
where \( s \in \RR \) represents the \emph{shift}. 
From \eqref{eq:iter_power_method}, we can see that the Shifted PM method can be viewed as a special case of the TGP algorithm, where
$\mm = \mathbb{S}^{n-1} = \St(1, n) $, and the search direction is given by \( \matr{H}_k = \nabla f(\matr{X}_k) + (1-s)\matr{X}_k\), with a fixed stepsize \(\tau=1\).
\end{itemize}

For problem \eqref{eq:objec_func_g} on $\Gr(p,n)$, there are not as many algorithms as in the above $\St(r, n)$ scenario.
Several existing methods treat it as a quotient manifold \cite{boumal2015low,usevich2014optimization}. To our knowledge, for the Grassmann manifold in the form of \(\Gr(p,n) \eqdef \{\matr{X}\in\RR^{n\times n}: \matr{X}^{\T}=\matr{X}, \matr{X}^2=\matr{X}, \rank{\matr{X}}=p\}\), there is currently no projection-based algorithm in the literature. Instead, the RetrLS algorithm, relying on QR decomposition \cite{sato2014optimization}, has been introduced, demonstrating the weak convergence of the RGD algorithm.

For problem \eqref{eq:objec_func_g} on a general compact matrix manifold $ \mm \subseteq \RR^{n\times r}$,
it is worth mentioning that the weak convergence of RetrLS algorithms \eqref{eq:iter_retr} has been established, specifically when employing the Armijo stepsize and a gradient-related direction \cite[Thm 4.3.1]{absil2009optimization}.
The work presented in \cite{boumal2019GlobalRatesConvergence} derived the iteration complexity of the RGD algorithm \eqref{eq:iter_retr_RGD} with both fixed stepsize and Armijo stepsize under the \emph{Lipschitz-type regularity assumption}\footnote{See \cref{rem:Lipschitz-type-regualrity-asusmption} for more details about it.}. Furthermore, under a similar assumption, the iteration complexity of the RGD algorithm \eqref{eq:iter_retr_RGD} with Zhang-Hager type nonmonotone Armijo stepsize was also obtained in \cite{oviedo2023worst}.

\subsection{A summary of the convergence results}

In this paper, for the stepsize $\tau_k$ in \cref{alg:TGP}, we will choose three different types: the \emph{Armijo} stepsize presented in \cref{sec:TGP-A}, the \emph{Zhang-Hager type nonmonotone Armijo} stepsize presented in \cref{sec:nonmonotone_TGP}, and the \emph{fixed} stepsize discussed in \cref{sec:fixed_TGP}.
Using these three different types of stepsizes,
we mainly establish the weak convergence, iteration complexity and global convergence of \cref{alg:TGP} in the general sense, and our convergence results are summarized in \cref{table-example-3-0-9}. It will be seen that these convergence results subsume the results found in the literature designed for those special cases listed in \cref{subsec:special_literature}.
{\renewcommand{\arraystretch}{1.2}

\begin{table}[h!]
\centering
\caption{A summary of the convergence results for TGP algorithms (\cref{alg:TGP})
\tablefootnote{In the last column of \cref{table-example-3-0-9}, we use ``WeakC'', ``IterC'' and ``GlobC'' to represent weak convergence, iteration complexity and global convergence for simplicity. 
We use \( (\St) \), \( (\mm) \) and \( (L \mm) \) to denote the feasible regions for which convergence has been established in the literature.  
Here, \( (L \mm) \) represents a general submanifold with additional Lipschitz-type regularity assumption; See \cref{rem:Lipschitz-type-regualrity-asusmption}.
By \emph{gradient-equivalent}, we mean a tangent direction with similar properties as \( \matr{D}_{\rho}(\matr{X})\) in \cref{lem-D-rho-prop}.
See \cite[Def. 4.2.1]{absil2009optimization} for the definition of \emph{gradient-related} sequence.}}
\label{table-example-3-0-9}
\scalebox{0.8}{
\begin{tabular}{| P{2.3cm} | P{2.3cm} | P{2.3cm} | P{2.3cm} | P{1.8cm} | P{3.0cm} | P{2.5cm} |}
\toprule
\multirow{2}{*}{\textrm{\makecell[c]{Stepsizes}}} & \multirow{2}{*}{\textrm{\makecell[c]{Weak \\ convergence}}}& \multirow{2}{*}{\textrm{\makecell[c]{\multicolumn{1}{c}{Iteration} \\ complexity}}}  &
\multirow{2}{*}{\textrm{\makecell[c]{Global\\ convergence}}} & \multicolumn{3}{c|}{Special cases}    \\
\cline{5-7}
&&&& \multicolumn{1}{c|}{References} & \multicolumn{1}{c|}{Form of $\matr{H}_{k}$} & \multicolumn{1}{c|}{Convergence} \\
\hline
\multirow{6}{*}{\textrm{\makecell[c]{Armijo stepsize}}} & \multirow{6}{*}{\textrm{\makecell[c]{\cref{coro:weak_converg}}}} & \multirow{6}{*}{\textrm{\makecell[c]{\cref{thm:Armijo-stepsize-lower-bound-convergence-rate}}}(iii)}  & \multirow{6}{*}{\textrm{\makecell[c]{\cref{thm:Armijo-global-convergence}}}} &  \textrm{\makecell[c]{\cite{oviedo2019scaled,oviedo2021two}}} & \textrm{\makecell[c]{$\matr{D}_{1/2}(\matr{X}_{k})$ }}& \textrm{\makecell[c]{WeakC \( (\St) \)}}\\
\cline{5-7}&
& & &\textrm{\makecell[c]{\cite{oviedo2019non}}} & \makecell[c]{$\matr{D}_{\alpha/2(\alpha+\beta)}(\matr{X}_{k})$} & \textrm{\makecell[c]{WeakC \( (\St) \)}} \\
\cline{5-7}&
& & & \textrm{\makecell[c]{\cite{liu2019QuadraticOptimizationOrthogonality}}} & \textrm{\makecell[c]{$\matr{D}_{\rho}(\matr{X}_{k})$}} & \textrm{\makecell[c]{GlobC \( (\St) \)}}\\
\cline{5-7}&
& & & \textrm{\makecell[c]{\cite{sheng2022riemannian}}} & \textrm{\makecell[c]{$\matr{D}_{1/4}(\matr{X}_{k})$}} & \textrm{\makecell[c]{GlobC \( (\St) \)}}\\
\cline{5-7} & & & & \textrm{\makecell[c]{\cite{absil2009optimization}}} & \textrm{\makecell[c]{\small gradient-related}} & \textrm{\makecell[c]{WeakC \( (\mm) \)}}\\
\cline{5-7} & & & & \textrm{\makecell[c]{\cite{boumal2019GlobalRatesConvergence}}} & \textrm{\makecell[c]{\small gradient-equivalent}} & \textrm{\makecell[c]{IterC \( (L\mm) \)}}\\
\cline{1-7}
\multirow{3}{*}{\textrm{\makecell[c]{Zhang-Hager type \\ Nonmonotone \\ Armijo stepsize}}} &
\multirow{3}{*}{\textrm{\makecell[c]{\cref{thm:weak_converg-na}}}}  &
\multirow{3}{*}{\textrm{\makecell[c]{\cref{thm:nonmonotone-Armijo-stepsize-lower-bound-convergence-rate}(iii)}}} &
\multirow{3}{*}{\textrm{\makecell[c]{\cref{thm:nonmonotone-global-convergence}}}} &
\textrm{\makecell[c]{\cite{oviedo2019non}}}&
\textrm{\makecell[c]{$\matr{D}_{\alpha/2(\alpha+\beta)}(\matr{X}_{k})$}}&
\textrm{\makecell[c]{WeakC \( (\St) \)}} \\
\cline{5-7} & & & & \textrm{\makecell[c]{\cite{oviedo2022global}}} & \textrm{\makecell[c]{\small gradient-equivalent}} & \textrm{\makecell[c]{WeakC \( (L\mm) \)}}\\
\cline{5-7} & & & & \textrm{\makecell[c]{\cite{oviedo2023worst}}} & \textrm{\makecell[c]{\small gradient-equivalent}} & \textrm{\makecell[c]{IterC \( (L\mm) \)}}\\
\cline{1-7}
\multirow{5}{*}{\textrm{\makecell[c]{Fixed stepsize\tablefootnote{We would like to emphasize that although fixed stepsizes are used in both the listed literature and our paper, our derivation is different from the listed literature, and thus the conditions required for convergence are also different; see more details in \cref{remark:tgp-f-conv}.}\\ \mbox{}}}} &  \multirow{5}{*}{\textrm{\makecell[c]{\cref{thm:TGP-F-weak-convergence-convergence-rate}(i)\\ \mbox{}}}}  & \multirow{5}{*}{\textrm{\makecell[c]{\cref{thm:TGP-F-weak-convergence-convergence-rate}(ii)\\ \mbox{}}}} &\multirow{5}{*}{\textrm{\makecell[c]{\cref{thm:TGP-F-global-convergnce}\\ \mbox{} }}}&   \textrm{\makecell[c]{\cite{gao2018new,yang2019epsilon}\\ \cite{hu2019LROAT,li2019polar} \\}}&\textrm{\makecell[c]{$\nabla f(\matr{X}_k)$}}& \textrm{\makecell[c]{WeakC \&\\ GlobC \( (\St) \)}} \\
\cline{5-7} & & & & \textrm{\makecell[c]{\cite{kolda2011shifted,kolda2014adaptive}\\ \cite{saad2011numerical,li2019polar}}} & \textrm{\makecell[c]{$\nabla f(\matr{X}_k) + (1-s)\matr{X}_k$}} & \textrm{\makecell[c]{  WeakC \&\\ GlobC \((\St(1, n))\)}}\\
\cline{5-7} & & & & \textrm{\makecell[c]{\cite{boumal2019GlobalRatesConvergence}}} & \textrm{\makecell[c]{$\grad f(\matr{X}_k)$}} & \textrm{\makecell[c]{  IterC \( (L\mm) \)}}\\
\bottomrule
\end{tabular}
}
\end{table}


\section{Properties of the projection onto a general compact manifold}\label{sec:Prop-proj}

In this section, we explore the geometric properties of the projection onto a general compact manifold, which will play a crucial role in the convergence analysis of our TGP algorithmic framework in the later sections.

\subsection{Uniqueness and smoothness of the projection}
In the context of RetrLS algorithms \eqref{eq:iter_retr}, a retraction is required to be smooth, at least in a local sense.
In this subsection, we discuss the \emph{smoothness} of the projection onto a general compact submanifold, as well as its \emph{uniqueness}.
We now begin with the following classical result, which demonstrates the well-defined nature of the projection onto a general compact manifold and derives its differential, indicating that this projection forms a retraction.

\begin{lemma}[{\cite[Lem. 4]{absil2012projection}}]\label{lem:differential-projection}
Let $ \mm' \subseteq \RR^{m}$ be a submanifold of class $C^k$ with \( k \geq 2 \) and $\mathcal{P}_{\mm'}$ be the projection onto $\mm'$. Given any $\bar{\vect{x}} \in \mm'$, there exists \(\varrho_{\bar{\vect{x}}} > 0\) such that $\mathcal{P}_{\mm'}(\vect{y})$ uniquely exists for all $ \vect{y} \in \openball{\bar{\vect{x}}}{\varrho_{\bar{\vect{x}}}}$. Moreover, $ \mathcal{P}_{\mm'}(\vect{y}) $ is of class \(C^{k-1}\) for $ \vect{y} \in \openball{\bar{\vect{x}}}{\varrho_{\bar{\vect{x}}}} $, and $\DD\mathcal{P}_{\mm'}(\bar{\vect{x}})=\mathcal{P}_{\TangM{\bar{\vect{x}}}}$.
\end{lemma}

Based on the above \cref{lem:differential-projection}, we now present a new result that reveals a stability property of the projection when the normal vector is relatively small. The proof of this result highly depends on that of \cite[Lem. 4]{absil2012projection}.

\begin{lemma}\label{lem:differential-projection-2}
Let $ \mm' \subseteq \RR^{m}$ be a submanifold of class $C^k$ with \( k \geq 2 \) and $\mathcal{P}_{\mm'}$ be the projection onto $\mm'$.
Let $\bar{\vect{x}} \in \mm'$ and \(\varrho_{\bar{\vect{x}}} > 0\) be as given in \cref{lem:differential-projection}.
Then, for all $\vect{x} \in \mm' \cap \openball{\bar{\vect{x}}}{\varrho_{\bar{\vect{x}}}} $ and $ \vect{w} \in \NormalMM{\vect{x}} $ satisfying $\vect{x} + \vect{w} \in \openball{\bar{\vect{x}}}{\varrho_{\bar{\vect{x}}}} $, we have \(\mathcal{P}_{\mm'}(\vect{x}+ \vect{w}) = \vect{x}\).
\end{lemma}

\begin{proof}
Consider the mapping $ F: \NormalBB \to \RR^m, (\vect{x}, \vect{w}) \mapsto \vect{x} + \vect{w} $, where \( \NormalBB\eqdef \{ (\vect{x}, \vect{w}) \in \RR^m \times \RR^m: \vect{x} \in \mm', \vect{w} \in \NormalMM{\vect{x}} \} \) denotes the normal bundle of \( \mm' \). It follows from the construction of \(\varrho_{\bar{\vect{x}}}\) in the proof of \cite[Lem. 4]{absil2012projection} that $ F $ is injective  in $\mathcal{I} \eqdef \{ (\vect{x}, \vect{w}) : \vect{x} \in \mm' \cap \openball{\bar{\vect{x}}}{2\varrho_{\bar{\vect{x}}}}, \vect{w} \in \NormalMM{\vect{x}},\|\vect{w}\| < 3\varrho_{\bar{\vect{x}}}\}$. For all $ \vect{x} \in \mm' \cap\openball{\bar{\vect{x}}}{\varrho_{\bar{\vect{x}}}} $ and $ \vect{w} \in \NormalMM{\vect{x}} $ satisfying $ \vect{x} + \vect{w} \in  \openball{\bar{\vect{x}}}{\varrho_{\bar{\vect{x}}}}$, we know from \cref{lem:differential-projection} that $ \mathcal{P}_{\mm'}(\vect{x} + \vect{w}) $ uniquely exists and is in $ \mm' \cap \openball{\bar{\vect{x}}}{2\varrho_{\bar{\vect{x}}}} $. By the definition of projection, \(\mathcal{P}_{\mm'}(\vect{x} + \vect{w})\) is the solution of the optimization problem $ \min_{\vect{z} \in \mm'} \| \vect{x}+ \vect{w} - \vect{z} \|^2 $. Thus, \(\mathcal{P}_{\mm'}(\vect{x}+\vect{w})\) satisfies the first-order necessary condition: $ \vect{x} + \vect{w} - \mathcal{P}_{\mm'}(\vect{x} + \vect{w}) \in \NormalMM{\mathcal{P}_{\mm'}(\vect{x}+ \vect{w})}  $.
Therefore, there exists $ \vect{w}' \in \NormalMM{\mathcal{P}_{\mm'}(\vect{x}+ \vect{w})}$ such that $ \vect{x}+ \vect{w} = \mathcal{P}_{\mm'}(\vect{x} + \vect{w}) + \vect{w}' $. It follows from $\vect{x}, \vect{x} + \vect{w} \in \openball{\bar{\vect{x}}}{\varrho_{\bar{\vect{x}}}}$ that  $ \| \vect{w} \| < 2\varrho_{\bar{\vect{x}}} $. Since $ \mathcal{P}_{\mm'}(\vect{x}+ \vect{w}) \in \openball{\bar{\vect{x}}}{2\varrho_{\bar{\vect{x}}}}$ and $\vect{x} \in \openball{\bar{\vect{x}}}{\varrho_{\bar{\vect{x}}}}$, we have $ \| \vect{w}' \| < 3 \varrho_{\bar{\vect{x}}} $. Thus, $ (\vect{x}, \vect{w}) $ and $( \mathcal{P}_{\mm'}(\vect{x}+ \vect{w}), \vect{w}') $ belong to the region $\mathcal{I}$. It follows from $F(\vect{x}, \vect{w}) = \vect{x} + \vect{w} = \mathcal{P}_{\mm'}(\vect{x} + \vect{w}) + \vect{w}' = F(\mathcal{P}_{\mm'}(\vect{x} + \vect{w}), \vect{w}') $ that $ \vect{x} = \mathcal{P}_{\mm'}(\vect{x} + \vect{w}) $.
The proof is complete.
\end{proof}

Note that the above \cref{lem:differential-projection,lem:differential-projection-2} are both local results, \emph{i.e.}, the radius \(\varrho_{\bar{\vect{x}}}\) depends on \( \bar{\vect{x}} \). In this paper, by the compactness of \( \mm' \), we can achieve the following stronger result where the radius is independent of the specific choice of \( \bar{\vect{x}} \), denoted by \(\varrho\).

\begin{lemma}\label{cor:general-manifold-uniform-projection-radius}
Let $ \mm' \subseteq \RR^{m}$ be a compact submanifold of class $C^k$ with \( k \geq 2 \).  Then there exists a positive constant \(\varrho>0\) such that, for all $ \vect{y}\in \mathcal{B}(\mm';\varrho), \mathcal{P}_{\mm'}(\vect{y})$ uniquely exists and is of class \(C^{k-1}\). Moreover, for all $ \vect{x}\in\mm'$ and $ \vect{w} \in \NormalMM{\vect{x}} $ satisfying $ \| \vect{w} \| < \varrho $, we have
\begin{equation}\label{eq:first-order-boundedness-normal-general}
\mathcal{P}_{\mm'}(\vect{x} + \vect{w}) = \vect{x}.
\end{equation}
\end{lemma}

\begin{proof}
It follows from \cref{lem:differential-projection} that for all $ \bar{\vect{x}} \in \mm' $, there exists \(\varrho_{\bar{\vect{x}}} > 0\) such that \(\mathcal{P}_{\mm'}(\vect{y})\) is of class \(C^{k-1}\),  $\mathcal{P}_{\mm'}(\vect{y})$ is unique for $ \vect{y}\in \openball{\bar{\vect{x}}}{\varrho_{\bar{\vect{x}}}} $ and $ \mathcal{P}_{\mm'}(\vect{x}+\vect{w}) = \vect{x} $ for $ \vect{x} \in \mm' \cap\openball{\bar{\vect{x}}}{\varrho_{\bar{\vect{x}}}} $ and $ \vect{w} \in \NormalMM{\vect{x}} $ satisfying $ \vect{x} + \vect{w} \in \openball{\bar{\vect{x}}}{\varrho_{\bar{\vect{x}}}} $.
We now first prove that there exists $ \varrho_1 > 0 $ such that $ \closedball{\mm'}{\varrho_{1}} \subseteq \cup_{\bar{\vect{x}} \in \mm'} \openball{\bar{\vect{x}}}{\varrho_{\bar{\vect{x}}}} $ by contradiction.
If not, then for all $ k \geq 1 $, there exists $ \vect{y}_k \in \closedball{\mm'}{1/k} $, such that $ \vect{y}_k \not\in \cup_{\bar{\vect{x}}\in \mm'}\openball{\bar{\vect{x}} }{\varrho_{\bar{\vect{x}}}} $. Since $ \mm' $ is bounded, the sequence $\{ \vect{y}_k \}_{k \geq 1}$ is contained in a compact set, and thus it has an accumulation point, namely, $ \vect{y}_{*} $. Noting that $ \dist(\vect{y}_k, \mm') < \frac{1}{k} $, we have $ \dist(\vect{y}_{*}, \mm') = 0 $, which implies $ \vect{y}_{*} \in \mm' $ due to the compactness of $ \mm' $. On the other hand, since $\cup_{\bar{\vect{x}} \in \mm'}\openball{\bar{\vect{x}} }{\varrho_{\bar{\vect{x}}}}$ is an open set and $ \vect{y}_k \not\in \cup_{\bar{\vect{x}}\in \mm'}\openball{\bar{\vect{x}}}{\varrho_{\bar{\vect{x}}}} $ for all $ k \geq 1 $, the accumulation point $ \vect{y}_{*} \not\in \cup_{\bar{\vect{x}} \in \mm'}\openball{\bar{\vect{x}}}{\varrho_{\bar{\vect{x}}}}$, which contradicts the fact that $ \vect{y}_{*} \in \mm' $ and $ \mm' \subseteq \cup_{\bar{\vect{x}} \in \mm'}\openball{\bar{\vect{x}} }{\varrho_{\bar{\vect{x}}}} $. Therefore, there exists $\varrho_1 > 0$ such that $ \closedball{\mm'}{\varrho_{1}} \subseteq \cup_{\bar{\vect{x}} \in \mm'} \openball{\bar{\vect{x}}}{\varrho_{\bar{\vect{x}}}} $.

Then, by the Lebesgue number lemma \cite{munkres1974topology}, there exists $ \varrho_2 > 0 $, such that for each subset of $ \closedball{\mm'}{\varrho_{1}} $ having diameter less than $ \varrho_2 $, there exists an element of $ \{ \openball{\bar{\vect{x}}}{\varrho_{\bar{\vect{x}}}}: \bar{\vect{x}} \in \mm' \} $ containing it.
Denote $ \varrho \eqdef \min \{ \varrho_1, \varrho_2/2\} $. Then we have $\openball{\mm'}{\varrho} \subseteq \openball{\mm'}{\varrho_1} \subseteq \cup_{\bar{\vect{x}} \in \mm'}\openball{\bar{\vect{x}}}{\varrho_{\bar{\vect{x}}}}$.  It follows from the definition of \( \varrho_{\bar{\vect{x}}} \) that $ \mathcal{P}_{\mm'}(\vect{y}) $ uniquely exists and is of class \(C^{k-1}\) for $ \vect{y}\in\openball{\mm'}{\varrho} $. Moreover, for all $ \vect{x} \in \mm' $ and $ \vect{w} \in \NormalMM{\vect{x}} $ with $ \|\vect{w}\| < \varrho $, the diameter of $ \openball{\vect{x}}{\varrho} $ is $2 \varrho$, which satisfies $2 \varrho \leq \varrho_2$. By the definition of $ \varrho_2 $, there exists $ \bar{\vect{x}} \in \mm' $ such that $ \openball{\vect{x}}{\varrho} \subseteq \openball{\bar{\vect{x}}}{\varrho_{\bar{\vect{x}}}} $. In particular, we have $ \vect{x}, \vect{x} + \vect{w} \in \openball{\bar{\vect{x}}}{\varrho_{\bar{\vect{x}}}} $ for $ \vect{w} \in \NormalMM{\vect{x}} $ satisfying $ \| \vect{w} \| < \varrho$. By the property of \( \varrho_{\bar{\vect{x}}} \) stated in \cref{lem:differential-projection-2}, we have $ \mathcal{P}_{\mm'}( \vect{x} + \vect{w}) = \vect{x} $.
The proof is complete.
\end{proof}

In this paper, we denote by \(\varrho_{*}\) the maximum value of the above positive constant \(\varrho\) in \cref{cor:general-manifold-uniform-projection-radius}.

\begin{remark}
The $\emph{reach}$ of a subset $\mathcal{S} \subseteq \RR^m$ is the largest \(\epsilon \in [0, \infty]\) such that the projection of any $\vect{y} \in \openball{\mathcal{S}}{\epsilon}$ onto $\mathcal{S}$ uniquely exists \cite[Def. 4.1]{federer1959curvature}. In our paper, the above value $\varrho_{*}$ serves as the reach of the submanifold $\mm'$. \cref{cor:general-manifold-uniform-projection-radius} implies that any compact \(C^2\) submanifold of Euclidean space has a positive reach. In fact, this result has already been established in \cite{foote1984regularity}. In \cref{cor:general-manifold-uniform-projection-radius}, we provide a different proof based on the local result \cref{lem:differential-projection}.
\end{remark}

Based on \cref{cor:general-manifold-uniform-projection-radius}, we now derive the following relationship among the tangent component, normal component and the trajectory of iterates, which will play a crucial role in the iteration complexity analysis of \cref{alg:TGP}.

\begin{lemma}\label{lem:first-second-order-boundedness}
Let $ \mm' \subseteq \RR^{m}$ be a compact submanifold of class \(C^3\) and \(\varrho_{*}>0\) be the positive constant defined after \cref{cor:general-manifold-uniform-projection-radius}. Then for any \(\delta \in (0, \varrho_{*}]\), there exist positive constants \( L^{(\delta)}_0, L^{(\delta)}_1, L^{(\delta)}_2 > 0 \) such that for all $ \vect{x} \in \mm' $, $ \vect{v} \in \TangMM{\vect{x}}$ and $\vect{w} \in \NormalMM{\vect{x}} $ satisfying $ \| \vect{w} \| \leq \varrho_{*} - \delta$, we have
\begin{align}
\| \mathcal{P}_{\mm'}(\vect{x}+\vect{v}+\vect{w}) - \vect{x} \|& \leq L^{(\delta)}_0 \| \vect{v} \|,\label{eq:first-order-boundedness}\\
\| \mathcal{P}_{\mm'}(\vect{x}+\vect{v}+\vect{w}) - \vect{x} - \vect{v} \| & \leq L^{(\delta)}_1 \| \vect{v} \|^2 + L^{(\delta)}_2 \| \vect{v} \| \| \vect{w} \|.\label{eq:second-order-boundedness-general}
\end{align}
\end{lemma}

\begin{proof}
Denote $ \mathcal{U} \eqdef \closedball{\mm'}{\varrho_{*} - \delta / 2} $ for simplicity.
By \cref{cor:general-manifold-uniform-projection-radius}, the projection mapping \(\mathcal{P}_{\mm'}\) is of class \(C^2\) on $ \openball{\mm'}{\varrho^{*}} $, and $ \mathcal{P}_{\mm'}(\vect{x} + \vect{w}) = \vect{x} $ for $\vect{w} \in \NormalMM{\vect{x}}$ satisfying $ \| \vect{w} \| \leq \varrho_{*} - \delta / 2$. Since $\mathcal{U} \subseteq \openball{\mm'}{\varrho^{*}}$ is a compact set, both $ \mathcal{P}_{\mm'} $ and its differential $ \DD \mathcal{P}_{\mm'} $ are Lipschitz continuous on $ \mathcal{B} $. Denote the Lipschitz constants of them by \(L_{\mathcal{P}_{\mm'}}\) and \(L_{\DD \mathcal{P}_{\mm'}}\), respectively. For any given $ \vect{x} \in \mm' $, $ \vect{v} \in \TangMM{\vect{x}} $ and $ \vect{w} \in \NormalMM{\vect{x}} $ satisfying $ \| \vect{w} \| \leq \varrho_{*} - \delta $, we consider the following two cases.

\emph{Case I:} $\| \vect{v} \| \leq \delta / 2$. In this case, both $ \vect{x} + \vect{w} $ and $ \vect{x} + \vect{v} + \vect{w}  $ are in $ \mathcal{U} $ since $ \| \vect{v} + \vect{w} \| \leq \| \vect{v} \| + \| \vect{w} \| \leq \varrho_{*} - \delta / 2$. It follows from \eqref{eq:first-order-boundedness-normal-general} and the Lipschitz continuity of $ \mathcal{P}_{\mm'} $ on $ \mathcal{U} $ that
\begin{equation}\label{eq:proof-first-order-boundedness-general-1}
\|\mathcal{P}_{\mm'}(\vect{x}+\vect{v}+\vect{w}) - \vect{x}\| = \| \mathcal{P}_{\mm'}(\vect{x}+\vect{v}+\vect{w}) - \mathcal{P}_{\mm'}(\vect{x}+\vect{w}) \| \leq L_{\mathcal{P}_{\mm'}} \| \vect{v} \|.
\end{equation}
Moreover, by using \eqref{eq:first-order-boundedness-normal-general} and applying the descent lemma to the mapping \(\mathcal{P}\) at $ \vect{x} + \vect{w} $, we have
\begin{align}
&\| \mathcal{P}_{\mm'}(\vect{x} + \vect{v} + \vect{w} ) - \vect{x} - \DD \mathcal{P}_{\mm'}(\vect{x}+\vect{w})[\vect{v}] \|\notag\\
&= \| \mathcal{P}_{\mm'}(\vect{x} + \vect{v} + \vect{w} ) - \mathcal{P}_{\mm'}(\vect{x}+\vect{w}) - \DD \mathcal{P}_{\mm'}(\vect{x}+\vect{w})[\vect{v}] \|\leq \frac{L_{\DD \mathcal{P}_{\mm'}}}{2} \| \vect{v} \|^2.\label{eq:proof-second-order-boundedness-general-1}
\end{align}
Combining $\vect{v} = \mathcal{P}_{\TangMM{\vect{x}}}(\vect{v}) = \DD \mathcal{P}_{\mm'}(\vect{x})[\vect{v}]$ from \cref{lem:differential-projection} and the Lipschitz continuity of $ \DD \mathcal{P}_{\mm'} $, we obtain that
\begin{align}
\| \DD \mathcal{P}_{\mm'}(\vect{x} + \vect{w})[\vect{v}] - \vect{v} \| & =   \| \DD \mathcal{P}_{\mm'}(\vect{x} + \vect{w})[\vect{v}] - \DD \mathcal{P}_{\mm'}(\vect{x})[\vect{v}]  \| \notag \\
& = \| (\DD \mathcal{P}_{\mm'}(\vect{x} + \vect{w}) - \DD \mathcal{P}_{\mm'}(\vect{x}))[\vect{v}] \|  \leq L_{\DD \mathcal{P}_{\mm'}}  \| \vect{v} \| \| \vect{w} \|. \label{eq:proof-second-order-boundedness-general-2}
\end{align}
It follows from \eqref{eq:proof-second-order-boundedness-general-1} and \eqref{eq:proof-second-order-boundedness-general-2} that
\begin{equation}\label{eq:proof-second-order-boundedness-general-3}
\| \mathcal{P}_{\mm'}(\vect{x} + \vect{v} + \vect{w} ) - \vect{x} - \vect{v} \| \leq  \frac{L_{\DD \mathcal{P}_{\mm'}}}{2} \| \vect{v} \|^2 + L_{\DD \mathcal{P}_{\mm'}} \| \vect{v} \| \| \vect{w} \|.
\end{equation}

\emph{Case II:} $ \| \vect{v} \| > \delta / 2$. Let $ d_{\mm'} $ be the diameter of $ \mm' $, \emph{i.e.}, $ d_{\mm'}\eqdef\max\{ \| \vect{x} - \vect{x}' \|: \vect{x}, \vect{x}' \in \mm' \} $.
Note that $ \mathcal{P}_{\mm'}(\vect{x} + \vect{v} + \vect{w}), \vect{x} \in \mm' $ and $ \|\vect{v}\| > \delta / 2 $. We have
\begin{align}
\| \mathcal{P}_{\mm'}(\vect{x} + \vect{v} + \vect{w} ) - \vect{x} \| & \leq d_{\mm'} \leq \frac{2 d_{\mm'}}{\delta} \| \vect{v} \|, \label{eq:proof-first-order-boundedness-general-2} \\
\| \mathcal{P}_{\mm'}(\vect{x} + \vect{v} + \vect{w} ) - \vect{x} - \vect{v} \| & \leq \| \mathcal{P}_{\mm'}(\vect{x} + \vect{v} + \vect{w} ) - \vect{x} \| + \| \vect{v} \| \leq \left(\frac{4 d_{\mm'}}{\delta^2} +  \frac{2}{\delta}\right) \| \vect{v} \|^2. \label{eq:proof-second-order-boundedness-general-4}
\end{align}
It follows from \eqref{eq:proof-first-order-boundedness-general-1},  \eqref{eq:proof-second-order-boundedness-general-3}, \eqref{eq:proof-first-order-boundedness-general-2} and \eqref{eq:proof-second-order-boundedness-general-4} that, the proof is complete if we set $ L^{(\delta)}_0 = \max\left\{ L_{\mathcal{P}_{\mm'}}, \frac{2d_{\mm'}}{\delta} \right\} $, $ L^{(\delta)}_1 = \max\left\{ \frac{L_{\DD \mathcal{P}_{\mm'}}}{2}, \frac{4d_{\mm'}}{\delta^{2}}+\frac{2}{\delta} \right\} $ and $ L^{(\delta)}_2 = L_{\DD \mathcal{P}_{\mm'}}$.
\end{proof}

\begin{remark}
In \cref{cor:general-manifold-uniform-projection-radius}, we proved that, for all $ \vect{y} $  satisfying $ \dist(\vect{y}, \mm') < \varrho_{*}$, the projection \(\mathcal{P}_{\mm'}(\vect{y})\) uniquely exists.
It is worth noting that, in \cref{lem:first-second-order-boundedness}, the vector $\vect{x}+\vect{v}+\vect{w}$ may not satisfy this condition, and thus its projection is not necessarily unique.
For example, when $ \| \vect{v} \| > \delta / 2$, it is possible that $ \dist(\vect{x}+\vect{v}+\vect{w}, \mm') \geq \varrho_{*}$.
However, even in this case, the inequalities \eqref{eq:first-order-boundedness} and \eqref{eq:second-order-boundedness-general} still hold, considering $\mathcal{P}_{\mm'}(\vect{x}+\vect{v}+\vect{w})$ as an arbitrary projection of it.
\end{remark}

\begin{remark}
Let $ \mm' \subseteq \RR^{m}$ be a compact submanifold of class \(C^3\)  and \( \retr \) be a retraction on it. It was shown in \cite[Eq. (B.3), (B.4)]{boumal2019GlobalRatesConvergence} that there exist positive constants \(L_0',L_1'>0\), such that for all \( \vect{x} \in \mm' \) and \( \vect{v} \in \TangMM{\vect{x}} \),
\begin{equation}\label{eq:boumal-first-second-boudnedness}
\|\retr_{\vect{x}}(\vect{v}) - \vect{x}\| \leq L_0' \| \vect{v} \|, \
\| \retr_{\vect{x}}(\vect{v}) - \vect{x} - \vect{v} \| \leq L_1' \| \vect{v} \|^2,
\end{equation}
where the second inequality is referred to as \emph{second-order boundedness}.
In \cite{liu2019QuadraticOptimizationOrthogonality}, a value of \(L_1'\) satisfying \eqref{eq:boumal-first-second-boudnedness} is obtained for multiple kinds of retractions on \( \St(r,n ) \).
Note that, if we set \(\delta =\varrho_{*}\) in \cref{lem:first-second-order-boundedness}, then
$ \vect{w} = \vect{0} $ and the inequalities \eqref{eq:first-order-boundedness} and \eqref{eq:second-order-boundedness-general} reduce to
\begin{equation*}
\|\mathcal{P}_{\mm'}(\vect{x} + \vect{v}) - \vect{x}\| \leq L_0^{(\varrho_{*})} \| \vect{v} \|, \
\| \mathcal{P}_{\mm'}(\vect{x}+\vect{v}) - \vect{x} - \vect{v} \| \leq L_1^{(\varrho_{*})} \| \vect{v} \|^2,
\end{equation*}
for all $ \vect{x} \in \mm' $ and $ \vect{v} \in \TangMM{\vect{x}} $.
Hence, constructing a retraction by the projection, \cref{lem:first-second-order-boundedness} can be regarded as an extension of \eqref{eq:boumal-first-second-boudnedness} allowing the appearance of a normal vector $\vect{w} \in \NormalMM{\vect{x}} $.
\end{remark}

\begin{remark}
We would like to emphasize that it is not possible to improve \cref{lem:first-second-order-boundedness} in the following two parts through the examination of specific examples within \( \St(r, n) \).\\
(i) The condition that \( \| \vect{w}  \| \leq \varrho_{*} - \delta \) can not be weaker. Note that \(\varrho_{*}=1\) on \( \St(r, n) \) by \cref{exa:max-rho-St}. We just need to show that the inequality \eqref{eq:first-order-boundedness} may fail when \( \| \vect{w} \| = 1\) on \( \St(r, n) \).
Let \( \vect{x} = (1,0)^{\T} \in \St(1,2) \), \( \vect{w} = (-1, 0)^{\T}\) and \( \vect{v} = (0, \epsilon)^{\T} \) for some \(\epsilon > 0\). Then
\[
\| \mathcal{P}_{\St(1,2)}(\vect{x} + \vect{v} + \vect{w}) -\vect{x} \| = \| \mathcal{P}_{\St(1,2)}((0, \epsilon)^{\T}) - (1, 0)^{\T} \| = \| (0, 1)^{\T} - (1,0)^{\T} \| = \sqrt{2}.
\]
It is clear that there does not exist a positive constant \(L_{0}>0\) such that the inequality \eqref{eq:first-order-boundedness} always holds. \\
(ii) The last term \( \| \vect{v} \| \| \vect{w} \|\) in the inequality \eqref{eq:second-order-boundedness-general} can not be removed unless \(\delta = \varrho_{*}\). For any \(\delta \in (0, 1)\) , let \( \vect{x} = (1,0)^{\T} \in \St(1,2) \), \( \vect{v} = (0, \epsilon)^{\T}\) and \( \vect{w} = (\delta / 2 - 1, 0)^{\T} \) for some \(\epsilon > 0\). Then
\begin{align*}
\left\| \mathcal{P}_{\St(1,2)}(\vect{x} + \vect{v} + \vect{w}) -\vect{x} - \vect{v} \right\|
& = \left\| \mathcal{P}_{\St(1,2)}((\delta / 2, \epsilon)^{\T}) - (1, 0)^{\T} - (0, \epsilon)^{\T} \right\|\\
& = \left\| \frac{1}{\sqrt{\delta^{2} / 4 + \epsilon^2}} (\delta / 2 - \sqrt{\delta^2 / 4 + \epsilon^2}, \epsilon (1 - \sqrt{\delta^2 / 4 + \epsilon^2}))^{\T} \right\| \\
& \geq \frac{\epsilon (1- \sqrt{\delta^{2} / 4 + \epsilon^2})}{\sqrt{\delta^{2} / 4 + \epsilon^2}}.
\end{align*}
Since \( \| \vect{v} \| = \epsilon\), we have that \( \liminf_{\epsilon \to 0} \| \mathcal{P}_{\St(1,2)}(\vect{x} + \vect{v} + \vect{w}) -\vect{x} - \vect{v} \| / \| \vect{v} \| \geq 2 / \delta - 1\), implying that \( \| \mathcal{P}_{\St(1,2)}(\vect{x} + \vect{v} + \vect{w}) -\vect{x} - \vect{v} \| / \| \vect{v} \|^2 \to \infty  \) as \( \epsilon \to 0 \). Therefore, there does not exist \(L_{1}>0\) such that \( \| \mathcal{P}_{\St(1,2)}(\vect{x} + \vect{v} + \vect{w}) -\vect{x} - \vect{v} \| \leq L_{1} \| \vect{v} \|^2 \) for all \( \vect{v} \in \TangMM{\vect{x}} \).
\end{remark}

While \cref{lem:first-second-order-boundedness} established an upper bound of the distance between  $\mathcal{P}_{\mm'}(\vect{x}+\vect{v}+\vect{w})$ and $\vect{x}$ in terms of the tangent component $\vect{v}$ and the normal component $\vect{w}$, we now present an inequality that builds its lower bound.

\begin{lemma}\label{thm:anti-first-order-boudnedness}
Let \( \mm' \subseteq \RR^m \) be a compact submanifold of class \(C^2\). Then there exists a positive constant \( L_3 > 0 \), such that for all \( \vect{x} \in \mm', \vect{v} \in \TangMM{\vect{x}} \) and \( \vect{w} \in \NormalMM{\vect{x}} \), we have
\begin{equation}\label{eq:anti-first-order-boundedness}
\| \mathcal{P}_{\mm'}(\vect{x}+\vect{v}+\vect{w}) - \vect{x} \| \geq \frac{\| \vect{v} \|}{1+L_3\| \vect{v} + \vect{w} \|}.
\end{equation}
\end{lemma}

\begin{proof}
For any \( \bar{\vect{x}} \in \mm' \), there exist a positive constant \(\nu_{\bar{\vect{x}}} > 0\) and a \(C^2\) local defining function \(\Phi_{\bar{\vect{x}}}: \openball{\bar{\vect{x}}}{\nu_{\bar{\vect{x}}}} \to \RR^{m-d} \) such that \( \Phi_{\bar{\vect{x}}}(\vect{x}) = \vect{0} \) and \( \matr{J}_{\Phi_{\bar{\vect{x}}}}(\vect{x})\) is of full rank for all \( \vect{x} \in \mm' \cap \openball{\bar{\vect{x}}}{\nu_{\bar{\vect{x}}}}  \), where \( d \) is the dimension of \( \mm' \) and  \( \matr{J}_{\Phi_{\bar{\vect{x}}}}(\vect{x}) \in \RR^{(m-d) \times m}\) is the Jacobian matrix of \( \Phi_{\bar{\vect{x}}} \) at \( \vect{x} \). Note that \( \TangMM{\vect{x}} = \operatorname{ker}(\DD \Phi_{\bar{\vect{x}}}(\vect{x})) \) by \cite[Eq. 3.19]{absil2009optimization} and \( \DD \Phi_{\bar{\vect{x}}}(\vect{x})[\vect{u}] = \matr{J}_{\Phi_{\bar{\vect{x}}}}(\vect{x})\vect{u} \).
It follows that, for all \( \vect{x} \in \mm' \cap \openball{\bar{\vect{x}}}{\nu_{\bar{\vect{x}}}} \), we have
\begin{equation}\label{eq:project-tangent-space-by-level-set}
\mathcal{P}_{\TangMM{\vect{x}}}(\vect{u}) = \left(\matr{I}_m - \matr{J}_{\Phi_{\bar{\vect{x}}}}(\vect{x})^{\T}(\matr{J}_{\Phi_{\bar{\vect{x}}}}(\vect{x}) \matr{J}_{\Phi_{\bar{\vect{x}}}}(\vect{x})^{\T})^{-1} \matr{J}_{\Phi_{\bar{\vect{x}}}}(\vect{x})\right)\vect{u},\  \forall \vect{u} \in \RR^m.
\end{equation}
Since \( \mm' \subseteq \cup_{\bar{\vect{x}} \in \mm'} \openball{\bar{\vect{x}}}{\nu_{\bar{\vect{x}}} / 2} \) and \( \mm' \) is compact, there exists a finite subset \( \mathcal{S} \subseteq \mm' \) such that \( \mm' \subseteq \cup_{\bar{\vect{x}} \in \mathcal{S}}\openball{\bar{\vect{x}}}{\nu_{\bar{\vect{x}}} / 2} \).
Since \( \matr{I}_m - \matr{J}_{\Phi_{\bar{\vect{x}}}}(\vect{x})^{\T}(\matr{J}_{\Phi_{\bar{\vect{x}}}}(\vect{x}) \matr{J}_{\Phi_{\bar{\vect{x}}}}(\vect{x})^{\T})^{-1} \matr{J}_{\Phi_{\bar{\vect{x}}}}(\vect{x}) \) is continuously differentiable for all \( \vect{x} \in \closedball{\bar{\vect{x}}}{\nu_{\bar{\vect{x}}} / 2 } \), it is Lipschitz continuous on \(\closedball{\bar{\vect{x}}}{\nu_{\bar{\vect{x}}} / 2 }\).  Let \( L_{\mathcal{P}_{\TangMM{\bar{\vect{x}}}}} \) be its Lipschitz constant and \(L_{\mathcal{P}_{\TangBundle{\mm'}}} \eqdef \max_{\bar{\vect{x}} \in \mathcal{S}} L_{\mathcal{P}_{\TangMM{\bar{\vect{x}}}}}\).
By a similar argument as in the proof of \cref{cor:general-manifold-uniform-projection-radius} and using the Lebesgue number lemma, we know that there exists \(\nu > 0\) such that, for all \( \vect{x}, \vect{z} \in \mm' \) satisfying \( \| \vect{x} - \vect{z} \| < \nu\), there exists \( \bar{\vect{x}} \in \mathcal{S} \) such that \( \vect{x}, \vect{z} \in \openball{\bar{\vect{x}}}{\nu_{\bar{\vect{x}}} / 2} \).
Then for such \( \vect{x} \) and \( \vect{z} \), using \eqref{eq:project-tangent-space-by-level-set} and the definition of \( L_{\mathcal{P}_{\TangMM{\bar{\vect{x}}}}} \), we have
\begin{equation}\label{eq:Lipschitz-projection-tangent}
\|\mathcal{P}_{\TangMM{\vect{x}}}(\vect{u}) - \mathcal{P}_{\TangMM{\vect{z}}}(\vect{u})\| \leq  L_{\mathcal{P}_{\TangMM{\bar{\vect{x}}}}}  \| \vect{u} \| \leq L_{\mathcal{P}_{\TangBundle{\mm'}}}\| \vect{u} \|,\ \forall \vect{u} \in \RR^m.
\end{equation}

Denote \( \vect{z}\eqdef\mathcal{P}_{\mm'}(\vect{x}+ \vect{v} + \vect{w}) \) and \(\vect{w}'\eqdef \vect{z} - (\vect{x} + \vect{v} + \vect{w}) \) for simplicity. Since \( \vect{z} \) is a projection of \( \vect{x} + \vect{v} + \vect{w}  \), we have \(\vect{w}' \in \NormalM{\vect{z}} \).  We consider the following two cases.

\emph{Case I:} \( \| \vect{z} - \vect{x} \| < \nu \).
By \eqref{eq:Lipschitz-projection-tangent}, we have that
\begin{align*}
\| \vect{z} - \vect{x} \| &= \| \vect{v} + \vect{w} + \vect{w}' \| \geq \| \mathcal{P}_{\TangM{\vect{z}}}(\vect{v} + \vect{w} + \vect{w}') \| =  \| \mathcal{P}_{\TangM{\vect{z}}}(\vect{v} + \vect{w}) \|\\
&\geq \| \mathcal{P}_{\TangMM{\vect{x}}}(\vect{v} + \vect{w}) \| - \| \mathcal{P}_{\TangM{\vect{z}}}(\vect{v}+\vect{w}) - \mathcal{P}_{\TangMM{\vect{x}}}(\vect{v}+\vect{w}) \| \\
&\geq \| \vect{v} \| - L_{\mathcal{P}_{\TangBundle{\mm'}}}\| \vect{z} - \vect{x} \| \| \vect{v} + \vect{w} \|.
\end{align*}
It follows that \( \| \vect{z} - \vect{x} \| \geq \| \vect{v} \| / (1 + L_{\mathcal{P}_{\TangBundle{\mm'}}} \| \vect{v} + \vect{w} \|). \)

\emph{Case II:} \( \| \vect{z} - \vect{x} \| \geq \nu\). Noting that \( \| \vect{v}+\vect{w} \| \geq \| \vect{v} \|\), we have
\begin{equation}
\| \vect{z} - \vect{x} \| \geq \nu \geq \nu\| \vect{v} \| / \| \vect{v} + \vect{w} \| \geq \| \vect{v} \| / (1 + \| \vect{v} + \vect{w} \| / \nu).
\end{equation}
Then, the proof is complete by setting \( L_3 = \max\{ L_{\mathcal{P}_{\TangBundle{\mm'}}}, 1 / \nu \} \).
\end{proof}

We conclude this subsection with an inequality providing a bound on $\mathcal{P}_{\NormalMM{\vect{x}}}(\vect{x} - \vect{y} )$, which will be used later.
\begin{lemma}\label{lem:project-x-y-normal}
Let \( \mm' \subseteq \RR^m \) be a compact submanifold of class \(C^2\). Then there exists a constant \( L_4 > 0 \), such that, for all \( \vect{x}, \vect{y} \in \mm'\),
\begin{equation}\label{eq:project-x-y-normal}
 \| \mathcal{P}_{\NormalMM{\vect{x}}}(\vect{x} - \vect{y} ) \| \leq L_4 \| \vect{x} - \vect{y} \|^2.
\end{equation}
\end{lemma}
\begin{proof}
Let \( \nu \), \( \nu_{\bar{\vect{x}}} \) and \( \mathcal{S} \) be as defined in the proof of \cref{thm:anti-first-order-boudnedness}.
When \( \| \vect{x} - \vect{y} \| < \nu \), there exists \( \bar{\vect{x}}\in \mm' \) such that \( \vect{x}, \vect{y} \in \openball{\bar{\vect{x}}}{\nu_{\bar{\vect{x}}} / 2} \).
Let \( L_{\DD \Phi_{\bar{\vect{x}}}} \) denote the Lipschitz constant of \( \DD \Phi_{\bar{\vect{x}}} \) on the compact set \( \closedball{\bar{\vect{x}}}{\nu_{\bar{\vect{x}} }/ 2} \), and define \( L_{\DD \Phi} \eqdef \max_{\bar{\vect{x}} \in \mathcal{S}}L_{\DD \Phi_{\bar{\vect{x}}}} \).
By the descent lemma \cite{beck2023nonlinear},
\begin{equation*}
\|\matr{J}_{\Phi_{\bar{\vect{x}}}} (\vect{x}) (\vect{y} - \vect{x})  \| = \| \Phi_{\bar{\vect{x}}}(\vect{y}) - \Phi_{\bar{\vect{x}}}(\vect{x}) - \matr{J}_{\Phi_{\bar{\vect{x}}}} (\vect{x}) (\vect{y} - \vect{x})  \| \leq \frac{L_{\DD \Phi}}{2} \| \vect{y} - \vect{x} \|^2,
\end{equation*}
where we used \( \Phi_{\bar{\vect{x}}}(\vect{x}) = \Phi_{\bar{\vect{x}}}(\vect{y}) = \vect{0} \) for \( \vect{x}, \vect{y} \in \mm'\cap \closedball{\bar{\vect{x}}}{\nu_{\bar{\vect{x}}} / 2} \).
Let \[ \Delta_{\matr{J}_{\bar{\vect{x}}}} \eqdef \max_{\vect{x} \in \closedball{\bar{\vect{x}}}{\nu_{\bar{\vect{x}}}/ 2} } \| \matr{J}_{\Phi_{\bar{\vect{x}}}}(\vect{x})^{\T}(\matr{J}_{\Phi_{\bar{\vect{x}}}}(\vect{x}) \matr{J}_{\Phi_{\bar{\vect{x}}}}(\vect{x})^{\T})^{-1} \|, \  \Delta_{\matr{J}} \eqdef \max_{\bar{\vect{x}} \in \mathcal{S}}\Delta_{\matr{J}_{\bar{\vect{x}}}}. \]
By \eqref{eq:project-tangent-space-by-level-set}, we have that 
\[   \| \mathcal{P}_{\NormalMM{\vect{x}}}(\vect{x} - \vect{y} ) \| = \left\|\left(\matr{J}_{\Phi_{\bar{\vect{x}}}}(\vect{x})^{\T}(\matr{J}_{\Phi_{\bar{\vect{x}}}}(\vect{x}) \matr{J}_{\Phi_{\bar{\vect{x}}}}(\vect{x})^{\T})^{-1}\right)\left(\matr{J}_{\Phi_{\bar{\vect{x}}}} (\vect{x}) (\vect{y} - \vect{x})\right) \right\| \leq \frac{\Delta_{\matr{J}}L_{\DD \Phi}}{2} \| \vect{x} - \vect{y} \|^2. \]
In the other case when \( \|\vect{x} - \vect{y}\| \geq \nu \), we have \[ \| \mathcal{P}_{\NormalMM{\vect{x}}}(\vect{x} - \vect{y} ) \| \leq \| \vect{x} - \vect{y} \| \leq \frac{1}{\nu} \| \vect{x} - \vect{y} \|^2.\] Thus, by setting \( L_4 \eqdef \max \left\{\frac{1}{\nu}, \frac{\Delta_{\matr{J}}L_{\DD \Phi}}{2}\right\} \),
we obtain the desired inequality for all \( \vect{x}, \vect{y} \in \mm' \). The proof is complete.
\end{proof}

\begin{remark}
A set \( \mathcal{S} \subseteq \RR^m \) is said to be \emph{proximally smooth} with radius \( \vartheta > 0 \) if the distance function \( \dist(\vect{y}, \mathcal{S}) \) is continuously differentiable for all \( \vect{y} \in \openball{\mathcal{S}}{\vartheta} \) \cite{clarke1995proximal}, which is equivalent to the uniqueness of the projection \( \mathcal{P}_{\mathcal{S}}(\vect{x}) \) when \( \mathcal{S} \) is weakly closed \cite[Thm.~4.11]{clarke1995proximal}.
By \cref{cor:general-manifold-uniform-projection-radius}, any compact \(C^2\) submanifold is proximally smooth, which is already known in the literature \cite{clarke1995proximal}.
Our results further reveal several enhanced properties of smooth submanifolds that do not hold for general proximally smooth sets, as summarized in \cref{tab:comparison_proximally_smooth}:
\begin{itemize}
\item
While proximal smoothness guarantees only local uniqueness of the projection, \cref{cor:general-manifold-uniform-projection-radius} shows that the projection onto a \( C^k \) submanifold is additionally \( C^{k-1} \) smooth.
\item The additional smoothness of the projection map onto a smooth manifold further allows us to establish the bounds for projection onto smooth submanifolds in \cref{lem:first-second-order-boundedness}. Also, our results \cref{thm:anti-first-order-boudnedness,lem:project-x-y-normal} rely on the manifold structure, which do not hold for general proximally smooth sets.
\item For a proximally smooth set \( \mathcal{S} \) with radius \( \vartheta \), the \emph{uniform normal inequality} \cite{clarke1995proximal,davis2025stochastic}:
\(\langle \vect{n}, \vect{y} - \vect{x} \rangle \leq \frac{1}{2 \vartheta} \| \vect{n} \| \| \vect{y} - \vect{x} \|^2\)
holds for any \( \vect{x} \in \mathcal{S} \) and \( \vect{n} \in \mathbf{N}_{\vect{x}} \mathcal{S} \).
In contrast, when \( \mathcal{S} \) is additionally a \( C^2 \) submanifold \( \mm \), our result \cref{lem:project-x-y-normal} implies
\[| \langle \vect{n}, \vect{y} - \vect{x} \rangle | = | \langle \vect{n}, \mathcal{P}_{\NormalMM{\vect{x}}}(\vect{y} - \vect{x}) \rangle | \leq L_4 \| \vect{n} \| \| \vect{y} - \vect{x} \|^2,
\]
where the first equality follows from the fact that $\vect{n} \in \NormalMM{\vect{x}}$.
From this perspective, our inequality \eqref{eq:project-x-y-normal} is stronger for \(C^2\) submanifold, since it implies the uniform normal inequality, even though the constant \( L_4 \) is implicit.
\end{itemize}

\begin{table}[htbp]
  \centering
  \footnotesize
  \resizebox{\textwidth}{!}{%
  \begin{tabular}{@{}l@{\quad}l@{\quad}l@{}}
    \toprule
    & Proximally smooth set & \( C^k \) submanifold \\
    \midrule
    Projection in neighborhood & Unique & Unique and \( C^{k-1} \)-smooth (\cref{cor:general-manifold-uniform-projection-radius}) \\
    Upper bound on \( \| \mathcal{P}_{\mm'}(\vect{x} + \vect{v} + \vect{w}) - \vect{x} \| \) & --- & \( L^{(\delta)}_0 \| \vect{v} \| \) (\cref{lem:first-second-order-boundedness}) \\
    Upper bound on \( \| \mathcal{P}_{\mm'}(\vect{x} + \vect{v} + \vect{w}) - \vect{x} - \vect{v} \| \) & --- & \( L^{(\delta)}_1 \| \vect{v} \|^2 + L^{(\delta)}_2 \| \vect{v} \| \| \vect{w} \| \) (\cref{lem:first-second-order-boundedness}) \\
    Lower bound on \( \| \mathcal{P}_{\mm'}(\vect{x} + \vect{v} + \vect{w}) - \vect{x} \| \) & --- & \( \| \vect{v} \| / (1 + L_3 \| \vect{v} + \vect{w} \|) \) (\cref{thm:anti-first-order-boudnedness}) \\
    Uniform normal inequality & \( \langle \vect{n}, \vect{y} - \vect{x} \rangle \leq \frac{1}{2 \vartheta} \| \vect{n} \| \| \vect{y} - \vect{x} \|^2 \) & \( \| \mathcal{P}_{\NormalMM{\vect{x}}}(\vect{x} - \vect{y}) \| \leq L_4 \| \vect{n} \| \| \vect{x} - \vect{y} \|^2 \) (\cref{lem:project-x-y-normal}) \\
    \bottomrule
  \end{tabular}
  }
  \caption{Comparison of geometric properties of proximally smooth sets with radius \( \vartheta \) and \( C^k \) submanifolds with \( k \geq 3 \)}\label{tab:comparison_proximally_smooth}
\end{table}
\end{remark}

\subsection{Explicit parameters for $\St(r, n)$ and $\Gr(p, n)$}
Our convergence analysis will show that the iteration complexity bounds for both the Armijo stepsize and the Zhang–Hager–type nonmonotone Armijo stepsize do not require explicit knowledge of the exact parameter values. However, establishing global convergence does require an upper bound on $\varrho_{*}$. In contrast, the fixed stepsize rule depends on the explicit values of these parameters. 
In this section, for the cases of \( \St(r, n) \) and \( \Gr(p, n) \), we estimate \(\varrho_{*}\) and \( L_0^{(\delta)} \).
Before that, we first present two lemmas concerning the projection onto \( \St(r, n) \).

\begin{lemma}[{\cite[Cor. 7.3.5]{horn2012matrix}}]\label{lem:Hoffman-Wielandt-thm}
Let $\matr{Y}, \tilde{\matr{Y}}\in\RR^{n\times r}$ be two matrices with $\sigma_1(\matr{Y}) \geq \cdots \geq \sigma_r(\matr{Y})$ and $\sigma_1(\tilde{\matr{Y}}) \geq \cdots \geq \sigma_r(\tilde{\matr{Y}})$ as the non-increasingly ordered singular values, respectively. Then \\
(i) $\left|\sigma_i(\matr{Y})-\sigma_i(\tilde{\matr{Y}})\right| \leq\|\matr{Y}-\tilde{\matr{Y}}\|_{\infty}$ for $1\leq i\leq r$;\\
(ii) $\sum_{i=1}^r\left(\sigma_i(\matr{Y})-\sigma_i(\tilde{\matr{Y}})\right)^2 \leq\|\matr{Y}-\tilde{\matr{Y}}\|^2$.
\end{lemma}

\begin{lemma}[{\cite[Thm.  9.4.1]{GoluV96:jhu}, \cite[Thm.  8.1]{higham2008functions}, \cite[Thm.  7.3.1]{horn2012matrix}}]\label{lemma-polar-decom}
Let $\matr{Y}\in\RR^{n\times r}$ with $1\leq r\leq n$.
There exist $\matr{X}\in\St(r,n)$ and a unique positive semi-definite matrix $\matr{P}\in \RR^{r\times r}$ such that $\matr{Y}$ has the \emph{polar decomposition} $\matr{Y}=\matr{X}\matr{P}$.
We say that $\matr{X}$ is the \emph{orthogonal polar factor} and $\matr{P}$ is the \emph{positive semi-definite polar factor}.
Moreover,\\
(i) for any $\matr{X}'\in\St(r,n)$, we have \cite[pp. 217]{higham2008functions}
\begin{equation*}
\langle\matr{X}, \matr{Y}\rangle\geq\langle\matr{X}', \matr{Y}\rangle;
\end{equation*}
(ii) $\matr{X}$ is the best orthogonal approximation \cite[Thm.  8.4]{higham2008functions} to $\matr{Y}$,
that is, for any $\matr{X}'\in\St(r,n)$, we have
\begin{equation*}
\|\matr{Y}-\matr{X}\| \leq \|\matr{Y}-\matr{X}'\|;
\end{equation*}
(iii) if $\rank{\matr{Y}}=r$, then $\matr{P}$ is positive definite and $\matr{X}$ is unique \cite[Thm.  8.1]{higham2008functions}. Moreover, we have $ \matr{X} = \matr{Y}(\matr{Y}^{\T}\matr{Y})^{-1 / 2} $ in this case.
\end{lemma}

\begin{example}[Calculation of $\varrho_{*}$ on \( \St(r, n) \)]\label{exa:max-rho-St}
The positive constant \(\varrho_{*} = 1\) on \( \St(r, n) \).
We first demonstrate that \(\varrho_{*} \geq 1\).
For any \( \matr{Y} \in\RR^{n\times r}\) satisfying $\dist(\matr{Y},\St(r, n))<1$, let \( \matr{X} \in \St(r, n) \) be a projection of \( \matr{Y} \). Then \( \| \matr{Y}-\matr{X} \| < 1\). It follows that
\[ \sigma_{\min}(\matr{Y}) = \sigma_{\min}(\matr{X}+(\matr{Y}-\matr{X})) \geq \sigma_{\min}(\matr{X}) - \sigma_{\max}(\matr{Y}-\matr{X}) \geq 1 - \| \matr{Y}-\matr{X} \| > 0, \]
where the first inequality follows from \cref{lem:Hoffman-Wielandt-thm}.
Therefore, \(\matr{Y}\) is nonsingular, implying that \( \mathcal{P}_{\St(r, n)}(\matr{Y}) \) is unique and \(\mathcal{P}_{\St(r, n)}(\matr{Y}) = \matr{Y}(\matr{Y}^{\T}\matr{Y})^{-1 / 2}\) by \cref{lemma-polar-decom}. It follows from the explicit expression of \(\mathcal{P}_{\St(r, n)}(\matr{Y})\) that it is of class \(C^{\infty}\) for all \( \matr{Y} \in\RR^{n\times r} \in \openball{\St(r, n)}{1}\). For any \( \matr{W} \in \NormalSt{\matr{X}} \), by the representation of normal space \eqref{def-St-normal-space}, there exists \( \matr{S} \in \symmm{\RR^{r \times r}} \) such that \( \matr{W} = \matr{X}\matr{S} \). If \( \| \matr{W} \| < 1\), we have \( \lambda_{\min}(\matr{S}) > -1 \). Thus \( \matr{X} + \matr{W} = \matr{X}(\matr{I}_r + \matr{S}) \) is the polar decomposition of \( \matr{X}+ \matr{W}\) and \( \mathcal{P}_{\St(r, n)}(\matr{X} + \matr{W}) = \matr{X} \).

On the other hand, we show that \(\varrho_{*}\leq 1\). For any \( \matr{X} \in \St(r, n) \), let \( \matr{Y} = [\matr{X}_{1:r-1}, \matr{0}_{n}] \in \RR^{n \times r} \), where \( \matr{X}_{1:r-1} \) denotes the first \( r-1 \) columns of \( \matr{X} \). Then \(\dist(\matr{Y}, \St(r, n)) = 1\) and the projection of \( \matr{Y} \) is of the form \(  [\matr{X}_{1:r-1}, \vect{y}]\), where \( \vect{y} \) is any unit vector orthogonal to \( \matr{X}_{1:r-1} \).
Therefore, the projection \( \mathcal{P}_{\St(r, n)}(\matr{Y}) \) is not unique.
\end{example}

\begin{lemma}[{\cite[Thm. 1, Thm. 2]{li1995new}}]\label{lema:pi_lipschitz}
Let $\matr{Y}, \bar{\matr{Y}}\in\RR^{n\times r}$ be two matrices of
full column rank, having the polar decompositions $\matr{Y}=\matr{X}\matr{P}$ and $\bar{\matr{Y}}=\bar{\matr{X}}\bar{\matr{P}}$, respectively.
Then we have
\begin{equation}\label{eq:pi_lips}
\|\matr{X}-\bar{\matr{X}}\| \leq \left(\frac{2}{\sigma_{\rm min}(\matr{Y})+\sigma_{\rm min}(\bar{\matr{Y}})}+\frac{1}{\max (\sigma_{\rm min}(\matr{Y}),\sigma_{\rm min}(\bar{\matr{Y}}))}\right)\|\matr{Y}-\bar{\matr{Y}}\|.
\end{equation}
\end{lemma}

\begin{lemma}\label{lem:singular-value-dominated-by-normal-vector}
Let \( \matr{X} \in \St(r,n)\) and $\matr{S} \in \symmm{\RR^{r \times r}}$. If \( \lambda_{\min}(\matr{S}) \geq \delta - 1 \) for some \(\delta > 0\), then for all $ \matr{V} \in \TangSt{\matr{X}} $, we have $ \sigma_{\min}(\matr{X} + \matr{V} + \matr{X}\matr{S}) \geq \delta$.
\end{lemma}

\begin{proof}
Denote \( \matr{S}_{\delta} \eqdef (1-\delta) \matr{I}_{r}+\matr{S} \). It is clear that \( \matr{S}_\delta \succeq \matr{0} \).
By equation \eqref{def-St-tangent-space-decomp}, there exist \(\matr{A}_{\matr{V}}\in\skewww{\RR^{r \times r}}\) and \(\matr{B}_{\matr{V}}\in \RR^{(n-r) \times r}\) such that \( \matr{V} = \matr{X}\matr{A}_{\matr{V}} + \matr{X}_{\perp}\matr{B}_{\matr{V}} \). Then we have
\small
\begin{align*}
(\matr{X}+\matr{V}+\matr{X}\matr{S})^{\T}(\matr{X}+\matr{V}+\matr{X}\matr{S}) & =  (\matr{X}+\matr{X}\matr{A}_{\matr{V}}+\matr{X}_{\perp}\matr{B}_{\matr{V}}+\matr{X}\matr{S})^{\T}(\matr{X}+\matr{X}\matr{A}_{\matr{V}}+\matr{X}_{\perp}\matr{B}_{\matr{V}}+\matr{X}\matr{S}) \\
& = (\matr{X}(\matr{I}+\matr{S}+\matr{A}_{\matr{V}}) + \matr{X}_{\perp}\matr{B}_{\matr{V}})^{\T}(\matr{X}(\matr{I}+\matr{S}+\matr{A}_{\matr{V}}) + \matr{X}_{\perp}\matr{B}_{\matr{V}})\\
& \stackrel{(a)}{=}(\matr{I}+\matr{S}+\matr{A}_{\matr{V}})^{\T}(\matr{I}+\matr{S}+\matr{A}_{\matr{V}}) + \matr{B}_{\matr{V}}^{\T}\matr{X}_{\perp}^{\T}\matr{X}_{\perp}\matr{B}_{\matr{V}}\\
& \succeq (\matr{I}+\matr{S}+\matr{A}_{\matr{V}})^{\T}(\matr{I}+\matr{S}+\matr{A}_{\matr{V}})
 = (\delta \matr{I} + (\matr{A}_{\matr{V}} + \matr{S}_{\delta}))^{\T}(\delta \matr{I} + (\matr{A}_{\matr{V}} + \matr{S}_{\delta}))\\
& = \delta^2\matr{I}  + \delta (\matr{A}_{\matr{V}}^{\T} + \matr{S}_{\delta}^{\T} + \matr{A}_{\matr{V}} + \matr{S}_{\delta}) + (\matr{A}_{\matr{V}}+\matr{S}_{\delta})^{\T}(\matr{A}_{\matr{V}}+\matr{S}_{\delta})
\stackrel{(b)}{\succeq} \delta^2\matr{I},
\end{align*}
\normalsize
where the equality \( (a) \) follows from that \( \matr{X}^{\T}\matr{X}_{\perp} = \matr{0} \), and the inequality \( (b) \) follows from that \( \matr{A}_{\matr{V}}^{\T} + \matr{A}_{\matr{V}} = \matr{0}\) and \(\matr{S}_{\delta}^{\T} = \matr{S}_{\delta} \succeq \matr{0} \).
Therefore, we have \(\sigma_{\min}(\matr{X}+\matr{V}+\matr{X}\matr{S}) \geq \delta\).
The proof is complete.
\end{proof}

\begin{example}[\( L_0^{(\delta)} \) on \( \St(r, n) \)]\label{exa:L0-St}
Let \( \matr{X} \in \St(r, n) \), \( \matr{V} \in \TangSt{\matr{X}} \) and \( \matr{W} \in \NormalSt{\matr{X}} \). It follows from equation \eqref{def-St-normal-space} that there exists \( \matr{S} \in \symmm{\RR^{r \times r}} \) such that \( \matr{W} = \matr{X}\matr{S} \). When \(\lambda_{\min}(\matr{S}) \geq \delta - 1\) from some \(\delta > 0\), we have \( \sigma_{\min}(\matr{X} + \matr{W} + \matr{V}) \geq \delta \) by \cref{lem:singular-value-dominated-by-normal-vector}. Since \( \matr{X} + \matr{W} = \matr{X}(\matr{I}_r + \matr{S}) \), we have \( \sigma_{\min}(\matr{X} + \matr{W}) \geq \delta \) and \( \mathcal{P}_{\St(r, n)}(\matr{X} + \matr{W}) = \matr{X} \). It follows from \cref{lema:pi_lipschitz} that
\[ \| \mathcal{P}_{\St(r, n)}(\matr{X} + \matr{V} + \matr{W}) - \matr{X} \| = \| \mathcal{P}_{\St(r, n)}(\matr{X} + \matr{V} + \matr{W}) - \mathcal{P}_{\St(r, n)}(\matr{X} + \matr{W}) \| \leq \frac{2}{\delta} \| \matr{V} \|.\]
In particular, for \( \matr{W} = \matr{X}\matr{S} \) satisfying \( \| \matr{W} \| \leq 1 - \delta\), we have \(\lambda_{\min}(\matr{S}) \geq \delta - 1\), implying that the above inequality holds. Thus, $2/\delta$ is a choice of \(L_0^{(\delta)}\) in \eqref{eq:first-order-boundedness}.
\end{example}

Within the context of the Stiefel manifold, we present the following proposition, which offers a more specific result for \cref{thm:anti-first-order-boudnedness}.
\begin{lemma}\label{lem:first-order-boundness-lower}
For all \( \matr{X} \in \St(r,n), \matr{V} \in \TangSt{\matr{X}}\) and \(\matr{W} \in \NormalSt{\matr{X}} \) satisfying $ \|\matr{X} + \matr{V} + \matr{W}\| \neq 0 $, we have that
\[
\| \mathcal{P}_{\St(r, n)} (\matr{X} + \matr{V} + \matr{W}) - \matr{X} \| \geq \frac{\| \matr{V} \|}{(r+1) \| \matr{X} + \matr{V} + \matr{W} \|}.
\]
\end{lemma}
\begin{proof}
Denote \(\matr{Z} \eqdef \mathcal{P}_{\St(r, n)}(\matr{X} + \matr{V} + \matr{W}) \) for simplicity. By \cref{lemma-polar-decom}, there exists a positive semi-definite matrix \( \matr{P} \in \symmm{\RR^{r \times r}} \) such that \( \matr{X} + \matr{V} + \matr{W} = \matr{Z}\matr{P} \). Since $ \matr{V} \in \TangSt{\matr{X}} $, we have \( \matr{X}^{\T}\matr{V}  + \matr{V}^{\T}\matr{X} = \matr{0} \) by \eqref{def-St-tangent-space-product}. For \( \matr{W} \in \NormalSt{\matr{X}} \), there exists \( \matr{S} \in \symmm{\RR^{r \times r}} \) such that \( \matr{W} = \matr{X}\matr{S} \) by \eqref{def-St-normal-space}, implying that \( \matr{X}\matr{X}^{\T}\matr{W} = \matr{W} = \matr{X}\matr{W}^{\T}\matr{X}\). Then we have
\begin{align*}
\|\matr{V}\| & =\frac{1}{2} \left\| \left(\matr{X} + \matr{V}+\matr{W} - \matr{X}\matr{X}^{\T}(\matr{X} + \matr{V}+\matr{W})\right) + \left(\matr{X} + \matr{V}+\matr{W} - \matr{X}(\matr{X}+\matr{V}+\matr{W})^{\T}\matr{X}\right) \right\| \\
& = \frac{1}{2} \| (\matr{Z}\matr{P} - \matr{X}\matr{X}^{\T}\matr{Z}\matr{P}) +(\matr{Z}\matr{P} - \matr{X}\matr{P}\matr{Z} ^{\T}\matr{X})\| \\
& =\frac{1}{2} \left\| \left((\matr{Z}-\matr{X})\matr{P} - \matr{X}\matr{X}^{\T}(\matr{Z}-\matr{X})\matr{P}\right) + \left((\matr{Z}-\matr{X})\matr{P} -\matr{X}\matr{P}(\matr{Z}-\matr{X})^{\T}\matr{X}\right)\right\| \\
& \leq \frac{1}{2} \left(\| (\matr{Z}-\matr{X})\matr{P}\| + \| \matr{X}\matr{X}^{\T}\| \|(\matr{Z}-\matr{X})\matr{P}\| + \|(\matr{Z}-\matr{X})\matr{P}\| +  \|  \matr{X}\| \|\matr{P}(\matr{Z}-\matr{X})^{\T}\| \|\matr{X}\|\right) \\
& \leq (r+1)\|\matr{Z}-\matr{X}\| \|\matr{X} + \matr{V} + \matr{W}\|,
\end{align*}
where the last inequality follows from $ \| \matr{P} \|  = \| \matr{X} + \matr{V} + \matr{W} \|$ and $ \| \matr{X} \| = \sqrt{r}$. The proof is complete.
\end{proof}

In the following lemma, we derive the properties of the projection onto $\Gr(p, n)$ and demonstrate that the computation of $\mathcal{P}_{\Gr(p, n)}$ can be achieved through eigenvalue decomposition.

\begin{lemma}\label{lem:proj-Gr}
Let $ \matr{Y} \in \RR^{n \times n} $ and $\symmo{\matr{Y}} = \matr{Q}\matr{\Lambda} \matr{Q}^{\T}$ be the eigenvalue decomposition, where $ \matr{Q} \in \ON{n} $ and $\matr{\Lambda} = \Diag{\lambda_1 , \lambda_2, \ldots , \lambda_n}$ with \(\lambda_1 \geq \lambda_2 \geq  \ldots \geq  \lambda_n\). Then we have that\\
(i) $\mathcal{P}_{\Gr(p, n)}(\matr{Y}) = \mathcal{P}_{\Gr(p, n)}(\symmo{\matr{Y}}) = \mathcal{P}_{\Gr(p, n)}(\matr{Y}^{\T})$;\\
(ii) \( \dist(\matr{Y}, \Gr(p, n)) \geq \dist(\symmo{\matr{Y}}, \Gr(p, n)) \);\\
(iii) $\matr{Q} \Diag{\matr{I}_p, \matr{0}}\matr{Q}^{\T} $ is a projection of $\matr{Y}$. This projection is unique if and only if \(\lambda_p > \lambda_{p+1}\).
\end {lemma}

\begin{proof}
(i) For any matrix $\matr{Y} \in \RR^{n \times n}$ and $\matr{X} \in \Gr(p, n)$, we have \( \| \matr{Y} - \matr{X} \|^2 = \| \matr{Y} \|^2 + \| \matr{X} \|^2 - 2 \langle \matr{Y}, \matr{X} \rangle = \| \matr{Y} \|^2 + p - 2 \langle \matr{Y}, \matr{X} \rangle\), where we use the fact that $ \| \matr{X} \|^2 = p $ for all $\matr{X} \in \Gr(p, n)$. Thus, $\mathcal{P}_{\Gr(p, n)}(\matr{Y}) = \argmax_{\matr{X} \in \Gr(p, n)} \langle \matr{Y}, \matr{X} \rangle$.
Note that $\matr{X} \in \Gr(p, n)$ is symmetric.  We have $ \langle \matr{Y}, \matr{X} \rangle = \langle \matr{Y}^{\T}, \matr{X} \rangle = \langle \symmo{\matr{Y}}, \matr{X} \rangle$. It follows that $\mathcal{P}_{\Gr(p, n)}(\matr{Y}) = \argmax_{\matr{X} \in \Gr(p, n)} \langle \matr{Y}, \matr{X} \rangle = \argmax_{\matr{X} \in \Gr(p, n)} \langle \matr{Y}^{\T}, \matr{X} \rangle = \mathcal{P}_{\Gr(p, n)}(\matr{Y}^{\T})$. Similarly, we also have $\mathcal{P}_{\Gr(p, n)}(\matr{Y}) = \mathcal{P}_{\Gr(p, n)}(\symmo{\matr{Y}})$. \\
(ii) Let \( \matr{X} \) be an arbitrary projection of \( \matr{Y} \). It follows from (i) that \( \matr{X} \) is also the projection of \( \symmo{\matr{Y}} \). Then we have
\begin{align*}
d^2(\matr{Y}, \Gr(p, n)) & = \| \matr{Y} - \matr{X} \|^2 = \| \matr{Y} \|^2 - 2 \langle \matr{Y}, \matr{X} \rangle + \| \matr{X} \|^2 = \| \matr{Y} \|^2 - 2 \langle \symmo{\matr{Y}}, \matr{X} \rangle + \| \matr{X} \|^2\\
& \geq \| \symmo{\matr{Y}} \|^2 - 2 \langle \symmo{\matr{Y}}, \matr{X} \rangle + \| \matr{X} \|^2 = d^2(\symmo{\matr{Y}}, \Gr(p, n)).
\end{align*}

(iii) It follows from the proof of (i) that
\[\mathcal{P}_{\Gr(p, n)}(\matr{Y}) = \mathcal{P}_{\Gr(p, n)}(\symmo{\matr{Y}}) =  \argmax_{\matr{X} \in \Gr(p, n)} \langle \symmo{\matr{Y}}, \matr{X} \rangle.\]
By the proof of \cite[Prop. 3.4]{sato2014optimization}, we have that \(\matr{Q}^{\T} \Diag{\matr{I}_p, \matr{0}}\matr{Q} \in \argmax_{\matr{X} \in \Gr(p, n)} \langle \symmo{\matr{Y}}, \matr{X} \rangle \). So \( \matr{Q}^{\T} \Diag{\matr{I}_p, \matr{0}}\matr{Q} \) is a projection of \( \matr{Y} \). Moreover, based on the remark below \cite[Prop. 3.4]{sato2014optimization}, the projection is unique if and only if \(\lambda_p > \lambda_{p+1}\).
The proof is complete.

\end{proof}

We now estimate $\varrho_{*}$ and \(L_0\) for $\Gr(p, n)$.

\begin{example}[Calculation of $\varrho_{*}$ on \( \Gr(p, n) \)]\label{exa:rstart-Gr}
Note that \( \Gr(n, n ) = \{\matr{I}_n\} \) and \( \Gr(0, n) = \{ \matr{0}_{n\times n} \} \), and the projection always results in a single point in these two trivial cases. We assume that \( 1 \leq p \leq n-1 \). We now show that \( \varrho_{*} = 1 / \sqrt{2} \). For any \( \matr{X} \in \Gr(p, n) \), there exists \( \matr{Q} \in \ON{n} \) such that \( \matr{X} = \matr{Q}\Diag{\matr{I}_{p}, \matr{0}}\matr{Q}^{\T} \) by definition. Let \( \matr{Y} = \matr{Q}\Diag{\matr{I}_{p-1}, 1 / 2, 1 / 2, \matr{0}}\matr{Q}^{\T}\). Then \( \| \matr{X} - \matr{Y} \| = 1 / \sqrt{2}\) and, according to \cref{lem:proj-Gr} (iii), the projection of \( \matr{Y} \) is not unique, since its \( p \)th eigenvalue is equal to the \( (p+1) \)th one. Therefore, $\varrho_{*} \leq 1 / \sqrt{2}$. Now we deduce that $\varrho_{*} = 1 / \sqrt{2}$. For any $\matr{Y} \in \openball{\Gr(p, n)}{1 / \sqrt{2}}$, let \( \matr{X}\) be a projection of \( \matr{Y}\). Then $\| \symmo{\matr{Y}} - \matr{X} \| \leq \| \matr{Y} - \matr{X} \| < 1/ \sqrt{2}$ by \cref{lem:proj-Gr}. Let \(\lambda_p\) and \(\lambda_{p+1}\) be the $p$th and $(p+1)$th largest eigenvalue of $\symmo{\matr{Y}}$, respectively. Note that the $p$th and $(p+1)$th largest eigenvalue of $\matr{X}$ is 1 and 0, respectively. It follows from \cite[Cor. 6.3.8]{horn2012matrix} that \((\lambda_p - 1)^2 + \lambda_{p+1}^2 \leq \| \symmo{\matr{Y}} - \matr{X} \|^2 < 1 / 2\). If \(\lambda_p = \lambda_{p+1}\), then $2\lambda_p^2 - 2\lambda_p + 1 / 2 < 0$, which is a contradiction. Therefore, \(\lambda_p > \lambda_{p+1}\) and the projection of $\matr{Y} \in \openball{\Gr(p, n)}{1 / \sqrt{2}}$ is unique by \cref{lem:proj-Gr}(iii). Using the fact that $\Gr(p, n)$ is a \(C^{\infty}\) manifold \cite{lee2012smooth}, we know that \(1 / \sqrt{2}\) satisfies the requirement of \(\varrho\) in \cref{cor:general-manifold-uniform-projection-radius} by \cite[Thm. 4.1]{dudek1994nonlinear} and \cite[Thm. 3.13]{dudek1994nonlinear}. In summary, $\varrho_{*} = 1 / \sqrt{2}$.
\end{example}


\begin{example}[\( L_0^{(\delta)} \) on \( \Gr(p, n) \)]
  Note that \( \Gr(p, n) \) is proximally smooth with radius \( 1 / \sqrt{2} \)
  as shown in \cite{balashov2020gradient}. We see that \(\mathcal{P}_{\Gr(p, n)}\) is Lipschitz continuous with constant \( L_{\mathcal{P}_{\Gr(p, n)}} = \sqrt{2} / \delta \) on \( \openball{\Gr(p, n)}{\varrho_{*} - \delta / 2} \) by the property of proximally smooth sets \cite[Thm. 4.8]{clarke1995proximal}. It follows from the construction of \( L_0^{(\delta)} \) in the proof of \cref{lem:first-second-order-boundedness} that \( \max \left\{ \sqrt{2} / \delta, 4 \sqrt{p} / \delta \right\} = 4 \sqrt{p} / \delta \) is a choice of \( L_0^{(\delta)} \), where we used the fact that the diameter of \( \Gr(p, n) \) is less than \( 2 \sqrt{p} \).
\end{example}

\subsection{Decrease in function value after projection}
We conclude this section with an inequality similar to the descent lemma, which helps estimate the decrease in function value in each iteration.
Before that, let us first recall the following \emph{Riemannian subgradient inequality} for weakly convex functions over compact manifolds.
\begin{lemma}[{\cite[Cor. 1]{li2021weakly}}]\label{lem:riemannian-subgradient-inequality}
Let \( \mm'\subseteq\RR^m \) be a compact submanifold given by \(\mm' = \{ \vect{x} \in \RR^m: F(\vect{x}) = \vect{0} \}\), where \( F:\RR^m \to \RR^{m'} \) is a smooth mapping whose derivative \( \DD F(\vect{x}) \) at \( \vect{x} \) has full row rank for all \( \vect{x} \in \mm' \). Then, for any weakly convex function \( g: \RR^m \to \RR \), there exists a constant \( c > 0 \) such that
\begin{equation*}
g(\vect{y}) - g(\vect{x})-\left\langle \tilde{\nabla}_R g(\vect{x}), \vect{y}-\vect{x}\right\rangle \geq - c\|\vect{y}-\vect{x}\|^2
\end{equation*}
for all $\vect{x}, \vect{y} \in \mm'$ and \( \tilde{\nabla}_R g(\vect{x}) \in \partial_R g(\vect{x}) \), where \( \partial_R g(\vect{x}) \) denotes the projection of the subdifferential \( \partial g(\vect{x}) \) onto the tangent space \( \TangMM{\vect{x}} \).
\end{lemma}

By checking the proof of \cite[Cor. 1]{li2021weakly}, it can be verified that, using the inequality in \cref{lem:project-x-y-normal}, we can extend the above lemma to the following result directly.

\begin{lemma}\label{lem:riemannian-subgradient-inequality-new}
Let \( \mm'\subseteq\RR^m \) be a compact submanifold of class \(C^2\). Then, for any function \( g: \RR^m \to \RR \) which has \(L_g\)-Lipschitz continuous gradient in the convex hull of $\mm'$, there exists a constant \(\Gamma_0 > 0\)  such that, for any $\vect{x}, \vect{y} \in \mm'$,
\begin{equation}\label{eq:general-Riemannian-descent-lemma-point-smooth}
\left|g(\vect{y}) - g(\vect{x})-\left\langle \grad g(\vect{x}), \vect{y}-\vect{x}\right\rangle \right| \leq \Gamma_0\|\vect{y}-\vect{x}\|^2.
\end{equation}
\end{lemma}

Based on \cref{lem:riemannian-subgradient-inequality-new} and our previous result \cref{lem:first-second-order-boundedness},  
we are now able to prove the following result. It is worth noticing that the following lemma holds for all projections of \( \vect{x} + \vect{v} + \vect{w}  \) even if the projection is not unique.
\begin{lemma}\label{thm:Riemannian-descent-lemma-proj}
  Let \( \mm' \subseteq \RR^m \) be a compact submanifold of class \(C^3\), \( g:\RR^m \to \RR \)
have Lipschitz continuous gradient in the convex hull of $\mm'$, and \(\delta \in (0, \varrho_{*}]\) be fixed, where \(\varrho_{*}\) is the constant related to \( \mm' \) defined after \cref{cor:general-manifold-uniform-projection-radius}. Then there exist constants \(\Gamma_{1}^{(\delta)}, \Gamma_{2}^{(\delta)} > 0\), such that for all $ \vect{x} \in \mm' $, $ \vect{v} \in \TangMM{\vect{x}}$ and $\vect{w} \in \NormalMM{\vect{x}} $ satisfying $ \| \vect{w} \| \leq \varrho_{*} - \delta$, we have
\begin{equation}\label{eq:Riemannian-gradient-descent-general-smooth}
| g\left(\mathcal{P}_{\mm'}(\vect{x} + \vect{v} + \vect{w} )\right) - g(\vect{x}) - \langle \grad g(\vect{x}), \vect{v} \rangle | \leq \Gamma_{1}^{(\delta)} \| \vect{v} \|^2 + \Gamma_{2}^{(\delta)} \| \grad g(\vect{x}) \| \| \vect{v} \|\| \vect{w} \|.
\end{equation}
\end{lemma}

\begin{proof}
It follows from \eqref{eq:general-Riemannian-descent-lemma-point-smooth} and \eqref{eq:first-order-boundedness} that
\begin{align*}
| g(\mathcal{P}_{\mm'}(\vect{x} + \vect{v} + \vect{w} )) - g(\vect{x}) - \langle \grad g(\vect{x}), \mathcal{P}_{\mm'}(\vect{x} + \vect{v} + \vect{w} ) - \vect{x} \rangle | & \leq \Gamma_0 \| \mathcal{P}_{\mm'}(\vect{x} + \vect{v} + \vect{w} ) - \vect{x} \|^2\\
& \leq \Gamma_0 (L_0^{(\delta)})^2 \| \vect{v} \|^{2}.
\end{align*}
Moreover, utilizing \eqref{eq:second-order-boundedness-general}, we have
\begin{align*}
|\langle \grad g(\vect{x}), \mathcal{P}_{\mm'}(\vect{x} + \vect{v} + \vect{w} ) - \vect{x} - \vect{v} \rangle|  & \leq \| \grad g(\vect{x}) \|  \| \mathcal{P}_{\mm'}(\vect{x} + \vect{v} + \vect{w} ) - \vect{x} - \vect{v} \| \\
& \leq \| \grad g(\vect{x}) \| (L_1^{(\delta)}\| \vect{v} \|^2 + L_2^{(\delta)}\| \vect{v} \| \| \vect{w} \|).
\end{align*}
Note that $ \| \grad g(\vect{x}) \| \leq \| \nabla g(\vect{x}) \| \leq \Delta_1 $. The proof is complete by combing the above two inequalities and setting \(\Gamma_{1}^{(\delta)} \eqdef \Gamma_0  (L_0^{(\delta)})^2 + \Delta_1 L_1^{(\delta)}\) and \(\Gamma_{2}^{(\delta)} \eqdef L_2^{(\delta)}\).
\end{proof}
\begin{remark}\label{rem:Lipschitz-type-regualrity-asusmption}
In \cite[Lem. 2.7]{boumal2019GlobalRatesConvergence}, the \emph{Lipschitz-type regularity assumption} is studied, showing that for any compact submanifold \( \mm' \subseteq \RR^m  \) and function \( g:\RR^m \to \RR \) with Lipschitz continuous gradient in the convex hull of \( \mm' \), there exists \( L_g > 0 \) such that for all \( \vect{x} \in \mm' \) and \( \vect{v} \in \TangMM{\vect{x}} \),
\begin{equation}\label{eq:Lipschit-type-regularity-assump}
|g(\retr_{\vect{x}}(\vect{v})) - g(\vect{x}) - \langle \grad g(\vect{x}), \vect{v} \rangle | \leq \frac{L_g}{2} \| \vect{v} \|^2.
\end{equation}
If we set \(\delta = \varrho_{*}\) in \cref{thm:Riemannian-descent-lemma-proj}, then \eqref{eq:Riemannian-gradient-descent-general-smooth} reduces to
\[
| g(\mathcal{P}_{\mm'}(\vect{x} + \vect{v})) - g(\vect{x}) - \langle \grad g(\vect{x}), \vect{v} \rangle | \leq \Gamma_{1}^{(\varrho_{*})} \| \vect{v} \|^2.
\]
Thus, \cref{thm:Riemannian-descent-lemma-proj} can be considered an extension of the Lipschitz-type regularity assumption to the case where the projection is used as the retraction in \eqref{eq:Lipschit-type-regularity-assump}. 
\end{remark}

\section{TGP algorithms using the Armijo stepsize}\label{sec:TGP-A}

\subsection{TGP-A algorithm}
We refer to \cref{alg:TGP} with the \emph{Armijo stepsize} \cite{armijo1966minimization,nocedal2006NumericalOptimization} as the \emph{Transformed Gradient Projection with the Armijo stepsize} (TGP-A) algorithm.
In this algorithm, the stepsize \(\tau_k\) is determined by employing a backtracking procedure to satisfy the Armijo condition:
\begin{equation}\label{eq:Armijo-condition}
f(\matr{X}_{k+1}) - f(\matr{X}_{k}) \leq \gamma \tau_k \langle \nabla f(\matr{X}_k), \matr{Z}'_k(0) \rangle= - \gamma \tau_k  \langle \grad f(\matr{X}_k), \matr{H}_k \rangle,
\end{equation}
where \(\gamma  \in (0, 1)\) is fixed.
Here, we make the assumption that the backtracking procedure is conducted using a parameter $\beta \in (0, 1)$ and a trial stepsize \(\hat{\tau}_k>0\) for each iteration \( k \). Then \(\tau_k\) can be expressed as follows:
\begin{equation}\label{eq:Armijo-stepsize}
\tau_k = \max\{\hat{\tau}_k \beta^{i}: f(\matr{Z}_k(\hat{\tau}_k \beta^{i})) - f(\matr{X}_{k}) \leq \gamma \hat{\tau}_k \beta^{i} \langle \nabla f(\matr{X}_k), \matr{Z}_k'(0) \rangle, i \in \NN \}, \text{ for all } k \in \NN.
\end{equation}
In the classical Armijo stepsize method, the trial stepsize \(\hat{\tau}_k\) is usually assumed to be fixed. However, in our context, we consistently assume that \(\hat{\tau}_k\) is adaptive, maintaining a uniform lower bound \(\hat{\tau}^{(l)} > 0\) and upper bound \(\hat{\tau}^{(u)} > 0\) throughout the paper.
As shown in \cref{cor:general-manifold-uniform-projection-radius}, $ \matr{Z}_k(\tau) $ is continuously differentiable for $\tau \in [0, \varrho_{*}/  \|\matr{H}_k\|) $.
It follows from \cref{lem:differential-projection} that
\begin{equation}
\matr{Z}_k'(0) = \mathcal{P}_{\TangM{\matr{X}_k}} (\matr{Y}_{k}'(0)) = \mathcal{P}_{\TangM{\matr{X}_k}} (- \matr{H}_k) = - \tilde{\matr{H}}_k,
\end{equation}
which implies that
\begin{equation}\label{eq:differential-Z}
\langle \nabla f(\matr{X}_k), \matr{Z}_{k}'(0) \rangle = \langle \grad f(\matr{X}_k), \matr{Z}_{k}'(0) \rangle = \langle \grad f(\matr{X}_k), \matr{Y}_k'(0) \rangle = - \langle \grad f(\matr{X}_k), \matr{H}_{k}  \rangle.
\end{equation}
Therefore, if \cref{asp:H-tilde-grad-equiv} is satisfied, then \( \langle \nabla f(\matr{X}_k), \matr{Z}_k'(0) \rangle < 0\), and thus the Armijo stepsize $ \tau_k $ always exists.

For the convenience of subsequent convergence analysis of \cref{alg:TGP},
we now introduce the following mild assumption about the boundedness of \( \matr{H}_k \).

\begin{assumption}[Boundedness of $\matr{H}_k$]\label{asp:H-boundedness}
$ \{ \| \matr{H}_k \| \}_{k\geq 0} $ is uniformly bounded, \emph{i.e.}, $\Delta_{\matr{H}} < +\infty $, where $ \Delta_{\matr{H}} \eqdef \sup_{k \in \NN}\| \matr{H}_k \|$.
\end{assumption}

Denote \( \Delta_{\hat{\matr{H}}} \eqdef \sup_{k \in \NN} \|\hat{\matr{H}}_k\| \). It easy to see that \( \Delta_{\hat{\matr{H}}} \leq \Delta_{\matr{H}} < +\infty \) under \cref{asp:H-boundedness}.

\begin{remark}
Given the compactness of $\mm$, we observe that \cref{asp:H-boundedness} is satisfied if $\matr{H}_k$ is chosen in a manner that continuously depends on $ \matr{X}_k $. For example, the classical EGP algorithm \eqref{eq:iter_prjec} with  $ \matr{H}_k = \nabla f(\matr{X}_k) $ and the RGD algorithm \eqref{eq:iter_retr_RGD} \cite{liu2019QuadraticOptimizationOrthogonality,sheng2022riemannian} with $\matr{H}_k = \grad f(\matr{X}_k)$ both satisfy this condition.
\end{remark}

\subsection{Weak convergence}\label{subsec:weak_TGP-A}
In this subsection, our objective is to establish the weak convergence of the TGP-A algorithm.
To begin, we first prove the following result, which can be seen as an extension of \cite[Thm 4.3.1]{absil2009optimization} when the retraction is constructed using the projection.

\begin{lemma}\label{lem:Armijo-stepsize-general-weak-convergence}
Let \( \mm \subseteq \RR^{n \times r} \) be a submanifold of class \(C^2\) and the cost function $f$ in \eqref{eq:objec_func_g} be continuously differentiable over $\RR^{n \times r}$.
In TGP-A algorithm, if there exists a fixed \(\upsilon > 0\) such that
\begin{equation}\label{eq:assumA-half}
\langle \grad f(\matr{X}_{k}), \matr{H}_{k} \rangle \geq  \upsilon \| \grad f(\matr{X}_k) \|^2\ \textrm{ for all} \ k \in \NN,
\end{equation}
and \( \matr{H}_k \) satisfies \cref{asp:H-boundedness},
then every accumulation point of the sequence \( \{ \matr{X}_k \}_{k \geq 0} \) is a stationary point. Moreover, we have that \( \lim_{k\to \infty}  \|\grad f(\matr{X}_k)\|  = 0 \).
\end{lemma}
\begin{proof}
We prove this lemma by contradiction. Assume that \( \matr{X}^{*} \) is an accumulation point which is a non-stationary point of \( f \) and  \( \{ \matr{X}_{k_j} \}_{j \geq 0} \) is the corresponding subsequence converging to \(\matr{X}^{*}\). Noting that the Armijo condition \eqref{eq:Armijo-condition} holds for all $ k \in \NN $ and the sequence $\{ f(\matr{X}_k) \}_{k \geq 0}$ is monotonically decreasing, we have that, for all $ j \in \NN $,
\[
f(\matr{X}_{k_{j+1}}) - f(\matr{X}_{k_j}) \leq f(\matr{X}_{k_j+1}) - f(\matr{X}_{k_j}) \leq - \gamma \tau_{k_j} \langle \grad f(\matr{X}_{k_j}), \matr{H}_{k_j} \rangle \leq - \gamma \tau_{k_j} \upsilon \| \grad f(\matr{X}_{k_j}) \|^2.
\]
In the above inequality, we have \( \lim_{j \to \infty}\tau_{k_j} = 0 \) by letting \( j \to \infty \) and using \( \| \grad f(\matr{X}^{*}) \|^2  > 0 \).
Thus, there exists \(j_0 \in \NN\) such that \(\tau_{k_j} / \beta < \min \{ \hat{\tau}^{(l)}, \varrho_{*} / \Delta_{\matr{H}} \}\) for all \( j \geq j_0 \).

Now we consider all \( j \) satisfying \( j \geq j_0 \). It follows from the rule of backtracking that the stepsize \( \tau_{k_j} / \beta \) does not satisfy the Armijo condition, that is,
\[
f(\matr{Z}_{k_j}(\tau_{k_j} / \beta)) - f(\matr{X}_{k_j}) > -\gamma \tau_{k_j}  \langle \grad f(\matr{X}_{k_j}), \matr{H}_{k_j} \rangle/ \beta.
\]
Dividing both sides by \(\tau_{k_j} / \beta\), we obtain that
\[
\frac{f(\matr{Z}_{k_j}(\tau_{k_j} / \beta)) - f(\matr{Z}_{k_j}(0))}{\tau_{k_j} / \beta}  > -\gamma   \langle \grad f(\matr{X}_{k_j}), \matr{H}_{k_j} \rangle = \gamma \langle \nabla f(\matr{X}_{k_j}), \matr{Z}_{k_j}'(0) \rangle.
\]
Moreover, it follows from \(\tau_{k_j}/ \beta < \varrho_{*} / \Delta_{\matr{H}}  \) and \cref{cor:general-manifold-uniform-projection-radius} that \(\matr{Z}_{k_j}(\tau) \) is continuously differentiable for \(\tau \in [0, \tau_{k_j}/\beta]\). By applying the mean value theorem on $ \matr{Z}_{k_j}(\tau) $, we know that there exists \(\bar{\tau}_{k_j} \in [0, \tau_{k_j} / \beta]\) such that
\begin{equation}
\label{eq:weak-convergence-proof-Armijo-condition-not-satisfied}
\langle \nabla f(\matr{X}_{k_j}), \matr{Z}_{k_j}'(\bar{\tau}_{k_j}) \rangle
=  \frac{f(\matr{Z}_{k_j}(\tau_{k_j} / \beta)) - f(\matr{Z}_{k_j}(0))}{\tau_{k_j} / \beta} > \gamma \langle \nabla f(\matr{X}_{k_j}), \matr{Z}_{k_j}'(0) \rangle.
\end{equation}
Since the sequence \(\{\matr{H}_{k_j}\}_{j \geq 0}\) is bounded, it has a convergent subsequence. Without loss of generality, we assume that the whole subsequence \( \{ \matr{H}_{k_j} \}_{j \geq 0}\) is convergent and denote its limit point by $ \matr{H}^{*} $. Let $ j \to\infty $ in \eqref{eq:weak-convergence-proof-Armijo-condition-not-satisfied}. Noting that \( \lim_{j \to \infty} \tau_{k_j} = 0 \) and \( \matr{Z}_{k_j}'(0) = -\mathcal{P}_{\TangM{\matr{X}_{k_j}}}(\matr{H}_{k_j}) \), we have
\begin{equation}\label{eq:Armijo-not-satisfy}
\langle \nabla f(\matr{X}^{*}), \mathcal{P}_{\TangM{\matr{X}^{*}}}(\matr{H}^{*}) \rangle \leq \gamma \langle \nabla f(\matr{X}^*), \mathcal{P}_{\TangM{\matr{X}^{*}}}(\matr{H}^{*}) \rangle.
\end{equation}
On the other hand, using \( \langle \nabla f(\matr{X}_{k_j}), \tilde{\matr{H}}_{k_j} \rangle =  \langle \grad f(\matr{X}_{k_j}), \allowbreak\matr{H}_{k_j} \rangle \) and \( \liminf_{j \to \infty} \langle \grad f(\matr{X}_{k_j}), \matr{H}_{k_j} \rangle > 0 \) from \eqref{eq:assumA-half} and the assumption that $\matr{X}^{*}$ is not a stationary point, we have
\(  \langle \nabla f(\matr{X}^{*}), \mathcal{P}_{\TangM{\matr{X}^{*}}}(\matr{H}^{*}) \rangle > 0\), which contradicts \eqref{eq:Armijo-not-satisfy} since \(\gamma \in (0,1)\). As a result, such accumulation point \( \matr{X}^{*} \) does not exist and every accumulation point of \( \{ \matr{X}_{k} \}_{k \geq 0} \) is a stationary point.

Now we further prove \( \lim_{k \to \infty} \| \grad f(\matr{X}_{k}) \| = 0\) by contradiction.
Assume there exist \(\epsilon > 0\) and a subsequence \( \{ \grad f(\matr{X}_{k_j}) \}_{j \geq 0} \) such that \( \| \grad f(\matr{X}_{k_j}) \| > \epsilon\) for all $j \geq 0$. Since \(\mm\) is compact, \( \{ \matr{X}_{k_j} \}_{j \geq 0} \) has an accumulation point, denoted by \(\bar{\matr{X}}\). Then \( \|\grad f(\bar{\matr{X}})\| \geq \epsilon \), which contradicts the fact that \( \bar{\matr{X}} \) is a stationary point we have proved. The proof is complete.
\end{proof}

It can be deduced from  \cref{lem:proper_c_k} (ii) that, if $ \matr{L}_k $ and $\matr{R}_k$ satisfy \cref{asp:H-tilde-grad-equiv}, then the inequality \eqref{eq:assumA-half} holds. Thus, we have the following result by \cref{lem:Armijo-stepsize-general-weak-convergence}.

\begin{corollary}\label{coro:weak_converg}
Let \( \mm \subseteq \RR^{n \times r} \) be a submanifold of class \(C^2\) and the cost function $f$ in \eqref{eq:objec_func_g} be continuously differentiable over $\RR^{n \times r}$.
In TGP-A algorithm, if $\matr{L}_k$ and $\matr{R}_k$ satisfy \cref{asp:H-tilde-grad-equiv} and $\matr{H}_{k}$ satisfies \cref{asp:H-boundedness}
, then every accumulation point of \( \{ \matr{X}_k \}_{k \geq 0} \) is a stationary point of $f$. Moreover, we have that $ \lim_{k \to \infty} \| \grad f(\matr{X}_k) \| = 0 $.
\end{corollary}

\begin{remark}\label{rem:literature_algorithms_weak_convergence}
(i) It is well-known that the weak convergence of general RetrLS algorithms with the Armijo stepsize has already been established in \cite[Thm. 4.3.1]{absil2009optimization}, covering a special case of \cref{coro:weak_converg} where \( \matr{H}_k \in \TangM{\matr{X}_k} \). In comparison, our result \cref{coro:weak_converg} extends beyond this by allowing for the presence of a normal component in $\matr{H}_k$, enabling applications to more general ProjLS algorithms \eqref{eq:iter_prjec-h} such as the classical EGP algorithm \eqref{eq:iter_prjec} (\( \matr{H}_k = \nabla f(\matr{X}_k) \)). \\
(ii) When $\mm=\St(r, n)$, the weak convergence of some special cases of TGP-A algorithm has been established with $\matr{H}_k =(\alpha+\beta)\matr{D}_{\alpha/2(\alpha+\beta)}(\matr{X}_{k})$ in \cite{oviedo2019non} and  $\matr{H}_k = \matr{D}_{1/2}(\matr{X}_{k})$ in \cite{oviedo2019scaled,oviedo2021two}.
In \cref{coro:weak_converg}, we have proved a more general weak convergence result.
Furthermore, we will establish their iteration complexity in \cref{subsec:conver_rate_tgp_a} and global convergence in \cref{subsec:glo-tgp-a}.
\end{remark}

\subsection{Iteration complexity}\label{subsec:conver_rate_tgp_a}

In this subsection, we mainly prove the following result about the iteration complexity of the TGP-A algorithm.

\begin{theorem}\label{thm:Armijo-stepsize-lower-bound-convergence-rate}
  Let \( \mm \subseteq \RR^{n \times r} \) be a compact submanifold of class \(C^3\) and the cost function $f$ in \eqref{eq:objec_func_g}
have Lipschitz continuous gradient in the convex hull of $\mm$. 
In TGP-A algorithm, if $\matr{L}_k$ and $\matr{R}_k$ satisfy \cref{asp:H-tilde-grad-equiv}, $\matr{H}_{k}$ satisfies \cref{asp:H-boundedness}, then \\
(i) there exists \(\tilde{\tau} > 0\) such that \(\tau_k \geq  \tilde{\tau}\) for all $ k \in \NN$;\\
(ii) we have that \( f(\matr{X}_k) - f(\matr{X}_{k+1}) \geq \gamma\tilde{\tau}\upsilon \| \grad f(\matr{X}_k) \|^2 \) for all \( k \in \NN \); in particular, it holds that \( \lim_{k \to \infty} \| \grad f(\matr{X}_k) \| = 0\);
\\
(iii) for any $ K \in \NN $, we have that
\[ \min_{0 \leq k \leq K}\| \grad f(\matr{X}_k) \| \leq \sqrt{\frac{f(\matr{X}_{0}) - f^{*}}{\gamma\tilde{\tau} \upsilon(K+1)}}. \]
\end{theorem}

\begin{proof}
  (i) Let \(\delta_{*} = \rho_{*} / 2\). 
For any \( k \in \NN \),
we consider two cases: (a) \(\tau_k > \rho_{*} / (2 \Delta_{\hat{\matr{H}}})\), and (b) \( \tau_k \leq \rho_{*} / (2 \Delta_{\hat{\matr{H}}}) \). In the first case, the stepsize already admits a uniform lower bound independent of \( k \), and thus we only need to consider the second case. 
Let \(\tau\) be any constant satisfying \( \tau \leq \rho_{*} / (2 \Delta_{\hat{\matr{H}}}) \).
Then
\(\tau \| \hat{\matr{H}}_k \| \leq \rho_{*} - \delta_{*}\). 
It follows from \cref{thm:Riemannian-descent-lemma-proj} that
\begin{align*}
f(\matr{Z}_k(\tau)) - f(\matr{X}_k) &\leq -\tau\langle \grad f(\matr{X}_k), \tilde{\matr{H}}_k \rangle + \tau^2\left(\Gamma_1^{(\delta_{*})} \| \tilde{\matr{H}}_k \|^2 + \Gamma_2^{(\delta_{*})} \| \grad f(\matr{X}_k) \| \| \tilde{\matr{H}}_k \| \| \hat{\matr{H}}_k \|\right)\\
&\leq -\tau\langle \grad f(\matr{X}_k), \tilde{\matr{H}}_k \rangle + \tau^2\left(\Gamma_1^{(\delta_{*})} \| \tilde{\matr{H}}_k \|^2 + \Gamma_2^{(\delta_{*})} \Delta_{\hat{\matr{H}}} \| \grad f(\matr{X}_k) \| \| \tilde{\matr{H}}_k \|\right).
\end{align*}
Thus, the Armijo condition \eqref{eq:Armijo-condition} holds for \(\tau\) satisfying
\[ \tau \langle \grad f(\matr{X}_k), \matr{H}_k \rangle - \tau^2 (\Gamma_1^{(\delta_{*})}  \| \tilde{\matr{H}}_k \| ^2 - \Gamma_2^{(\delta_{*})} \Delta_{\hat{\matr{H}}}\| \grad f(\matr{X}_k) \| \| \tilde{\matr{H}}_k \|) \geq \gamma \tau \langle \grad f(\matr{X}_k), \matr{H}_k \rangle,  \]
which is equivalent to that
\[\tau \leq \frac{(1 - \gamma)\langle \grad f(\matr{X}_{k}), \matr{H}_k \rangle}{\Gamma_1^{(\delta_{*})} \| \tilde{\matr{H}}_k \|^2 + \Gamma_2^{(\delta_{*})} \Delta_{\hat{\matr{H}}} \| \grad f(\matr{X}_k) \| \| \tilde{\matr{H}}_k \|}. \]
By \eqref{eq:H-grad-norm-equiv} and \eqref{eq:H-grad-product}, we see that the right-hand side of the above inequality has a uniform lower bound as follows:
\begin{align}
\frac{(1-\gamma) \langle \grad f(\matr{X}_k), \matr{H}_k \rangle }{\Gamma_1^{(\delta_{*})} \| \tilde{\matr{H}}_k \|^2 + \Gamma_2^{(\delta_{*})} \Delta_{\hat{\matr{H}}} \| \grad f(\matr{X}_k) \| \| \tilde{\matr{H}}_k \|} & \geq \frac{(1-\gamma)\upsilon \| \grad f(\matr{X}_k) \|^2}{\Gamma_1^{(\delta_{*})} \| \tilde{\matr{H}}_k \|^2 + \Gamma_2^{(\delta_{*})} \Delta_{\hat{\matr{H}}} \| \grad f(\matr{X}_k) \| \| \tilde{\matr{H}}_k \|}\notag\\
& \geq \frac{(1-\gamma)\upsilon}{\Gamma_1^{(\delta_{*})} \varpi^2 + \Gamma_2^{(\delta_{*})} \Delta_{\hat{\matr{H}}} \varpi}.\label{eq:proof-convergence-rate-2}
\end{align}
It follows that the Armijo condition holds for all \(\tau \leq \min \left( \frac{\rho_{*}}{2 \Delta_{\hat{\matr{H}}}}, \frac{(1-\gamma)\upsilon}{\Gamma_1^{(\delta_{*})} \varpi^2 + \Gamma_2^{(\delta_{*})} \Delta_{\hat{\matr{H}}} \varpi}\right)\). By the rule of backtracking \eqref{eq:Armijo-stepsize}, we have $ \tau_k \geq \min \left(\hat{\tau}_k,\beta \min\left(\frac{\rho_{*}}{2 \Delta_{\hat{\matr{H}}}}, \frac{(1-\gamma)\upsilon}{\Gamma_1^{(\delta_{*})} \varpi^2 + \Gamma_2^{(\delta_{*})} \Delta_{\hat{\matr{H}}} \varpi}\right)\right)$. Noting that \(\hat{\tau}_k\) has a lower bound \(\hat{\tau}^{(l)}\), we have that $\tau_{k}\geq \tilde{\tau} \eqdef \min \left(\hat{\tau}^{(l)}, \frac{\beta\varrho_{*}}{2 \Delta_{\hat{\matr{H}}}}, \frac{\beta(1-\gamma) \upsilon }{\Gamma_1^{(\delta_{*})} \varpi^2 + \Gamma_2^{(\delta_{*})}\Delta_{\hat{\matr{H}}} \varpi}\right)$ for all $k\in\NN$. \\
(ii)\&(iii) Combining \eqref{eq:Armijo-condition}, \eqref{eq:H-grad-product} and the result of (i), we obtain that for all $ k \in \NN $,
\[ f(\matr{X}_k)-f(\matr{X}_{k+1}) \geq \gamma \tau_k \langle \grad f(\matr{X}_k), \matr{H}_k \rangle \geq \gamma\tau_k \upsilon  \| \grad f(\matr{X}_k) \|^2 \geq \gamma\tilde{\tau} \upsilon  \| \grad f(\matr{X}_k) \|^2.  \]
For any $ K \in \NN $, summing the above inequality for $ k $ from 0 to $K$, we have
\begin{equation}\label{eq:proof-convergence-rate-1}
f(\matr{X}_0) - f^{*} \geq f(\matr{X}_0) - f(\matr{X}_{K+1}) \geq \gamma\tilde{\tau} \upsilon \sum_{k=0}^K  \| \grad f(\matr{X}_k) \|^2.
\end{equation}
It follows that \( \sum_{k=0}^{\infty} \| \grad f(\matr{X}_k) \|^2 \leq f(\matr{X}_0) - f^{*} \), implying that \( \lim_{k \to \infty}\| \grad f(\matr{X}_k) \| = 0\).
Moreover, it also follows from \eqref{eq:proof-convergence-rate-1}
\[ \min_{0 \leq k \leq K}\| \grad f(\matr{X}_k) \| \leq \sqrt{\frac{f(\matr{X}_{0}) - f^{*}}{\tilde{\tau} \upsilon \gamma (K+1)}}. \]
The proof is complete.
\end{proof}

\begin{remark}
In this paper, the weak convergence of the TGP-A algorithm is established both in \cref{lem:Armijo-stepsize-general-weak-convergence} and \cref{thm:Armijo-stepsize-lower-bound-convergence-rate}(ii) under different conditions.
We would like to remark that, although  \cref{thm:Armijo-stepsize-lower-bound-convergence-rate} assumes $f$ has Lipschitz continuous gradient in the convex hull of $\mm$, which is a stronger condition than that in \cref{lem:Armijo-stepsize-general-weak-convergence}, we establish an important inequality in \cref{thm:Armijo-stepsize-lower-bound-convergence-rate}(ii), which is crucial for the iteration complexity analysis in \cref{thm:Armijo-stepsize-lower-bound-convergence-rate}(iii).
In contrast, we don't obtain this inequality in the proof of \cref{lem:Armijo-stepsize-general-weak-convergence}.
\end{remark}

\begin{remark}\label{rem:literature_algorithms_convergence_rate}
For a general compact smooth Riemannian manifold $\mm$ and a sufficiently smooth cost function $f:\mm\rightarrow\RR$, if $f$ has Lipschitz continuous gradient in a convex compact set containing $ \mm $, the iteration complexity of the RetrLS algorithm with $\matr{H}_k = \grad f(\matr{X}_k)$ was proved in \cite[Thm. 2.11]{boumal2019GlobalRatesConvergence}.
In this paper, the iteration complexity result for the TGP-A algorithm we obtain in \cref{thm:Armijo-stepsize-lower-bound-convergence-rate}(iii) arrives at the same level as that in \cite[Thm. 2.11]{boumal2019GlobalRatesConvergence}, and our result allows for more choices of the tangent vectors (see \cref{exa:LR_St}), as well as the presence of a normal component in \( \matr{H}_k\),  which makes \( \matr{H}_k \) not necessarily tangent to $\mm$ at \( \matr{X}_k \).
\end{remark}

\subsection{Global convergence}\label{subsec:glo-tgp-a}
In this subsection, we mainly prove the following result about the global convergence of the TGP-A algorithm.

\begin{theorem}\label{thm:Armijo-global-convergence}
Let $ \mm \subseteq \RR^{n\times r}$ be an analytic compact matrix submanifold and $f$ in \eqref{eq:objec_func_g} be an analytic function. In TGP-A algorithm, if for sufficiently large \( k \),  $\matr{L}_k$ and \( \matr{R}_k \) satisfy \cref{asp:H-tilde-grad-equiv}, $\matr{H}_{k}$ satisfies \cref{asp:H-boundedness}, and there exists \(\delta \in (0, \varrho_{*}]\) such that \( \tau_k \| \hat{\matr{H}}_{k}\| \leq \varrho_{*} - \delta \), then the sequence \( \{ \matr{X}_k \}_{k \geq 0} \) converges to a stationary point \( \matr{X}^{*} \) and the estimation of the convergence speed \eqref{eq:KL_convergence_rate} holds.
\end{theorem}

\begin{proof}
For all \( k \in \NN \), we have that
\begin{equation}\label{eq:prf-global-convergence}
\begin{aligned}
f(\matr{X}_k)-f(\matr{X}_{k+1})  & \stackrel{(a)}{\geq} \gamma \tau_k \langle \grad f(\matr{X}_k), \matr{H}_k \rangle
= \gamma \tau_k \langle \grad f(\matr{X}_{k}), \tilde{\matr{H}}_k \rangle \stackrel{(b)}{\geq} \gamma \tau_k\upsilon \| \grad f(\matr{X}_k) \|^2\\
& \stackrel{(c)}{\geq} \frac{\gamma\upsilon }{\varpi} \| \grad f(\matr{X}_k) \| \| \tau_k\tilde{\matr{H}}_{k} \|
\stackrel{(d)}{\geq} \frac{\gamma\upsilon }{\varpi L_0^{(\delta)}}\| \grad f(\matr{X}_k) \| \| \matr{X}_{k+1} - \matr{X}_k \|,
\end{aligned}
\end{equation}
where \( (a) \) follows from the Armijo condition \eqref{eq:Armijo-condition}, \( (b) \) is by \eqref{eq:H-grad-product}, \( (c) \) is by \eqref{eq:H-grad-norm-equiv} and \( (d) \) follows from \( \| \mathcal{P}_{\mm}(\matr{X}_k - \tau_k \tilde{\matr{H}}_k - \tau_k \hat{\matr{H}}_k) - \matr{X}_k \| \leq L_0^{(\delta)} \tau_k \| \tilde{\matr{H}}_k \| \) by \eqref{eq:first-order-boundedness}. Therefore, the condition (i) of \cref{theorem-SU15} is satisfied.
When \( \grad f(\matr{X}_k) = \matr{0} \), we have \( \tilde{\matr{H}}_k = \matr{0} \) by \eqref{eq:H-grad-norm-equiv}, implying that \( \matr{H}_k = \hat{\matr{H}}_k \). Then it follows from \eqref{eq:first-order-boundedness-normal-general} and \( \tau_k \| \hat{\matr{H}}_{k} \| \leq \varrho_{*} - \delta \) that the condition (ii) of \cref{theorem-SU15} is also satisfied.
Since $\mm$ is compact, the sequence \( \{ \matr{X}_k \}_{k \geq 0} \) has an accumulation point  \( \matr{X}^{*} \). By \cref{thm:Armijo-stepsize-lower-bound-convergence-rate}(ii), \( \matr{X}^{*} \) is a stationary point of \( f \). Then, by \cref{theorem-SU15}, the sequence $ \{ \matr{X}_k \}_{k \geq 0} $ converges to this stationary point \( \matr{X}^{*} \).

For sufficiently large \( k \), since the assumptions of \cref{thm:Armijo-stepsize-lower-bound-convergence-rate} hold, we know there exists \(\tilde{\tau} > 0\) such that \(\tau_k \geq \tilde{\tau}\). It follows from \(\tau_k \| \matr{H}_k \| \leq \hat{\tau}^{(u)} \Delta_{\matr{H}}\), \cref{thm:anti-first-order-boudnedness} and \eqref{eq:H-grad-norm-equiv} that
\[\| \matr{X}_{k+1} - \matr{X}_k \| \geq \frac{\tau_{k} \|  \tilde{\matr{H}}_k \|}{1+\tau_k \| \matr{H}_k \|} \geq \frac{\tilde{\tau}\upsilon \| \grad f(\matr{X}_k) \|}{ 1+ \hat{\tau}^{(u)} \Delta_{\matr{H}}}. \]
Therefore, the condition (iii) of \cref{theorem-SU15} is satisfied, implying that \eqref{eq:KL_convergence_rate} holds.
The proof is complete.
\end{proof}

\begin{remark}\label{rem:literature_algorithms_global_convergence}
For an analytic function $f:\St(r, n)\rightarrow\RR$, the global convergence of RetrLS algorithm with \( \matr{H}_k = \matr{D}_{\rho}(\matr{X}_k) \) was established in \cite[Thm. 3]{liu2019QuadraticOptimizationOrthogonality}. When  \( \matr{H}_k = \grad f(\matr{X}_k) \) and the retraction is constructed using the projection, the global convergence was similarly established when it is applied to the orthogonal approximation problems of symmetric tensors \cite[Thm. 4.3]{sheng2022riemannian}.
In this paper, utilizing the geometric properties of the projection, we establish the global convergence of the TGP-A algorithm in \cref{thm:Armijo-global-convergence}, which is based on a general compact manifold $\mm$ and allows for more choices of the tangent vectors (see \cref{exa:LR_St}). Moreover, our result holds for a more general \(\matr{H}_k\), which is not necessarily a tangent vector at \( \matr{X}_k \).
\end{remark}

\begin{remark}
(i) In comparison to \cref{thm:Armijo-stepsize-lower-bound-convergence-rate}, \cref{thm:Armijo-global-convergence} introduces an additional assumption on \( \matr{H}_k \): \( \tau_k \| \hat{\matr{H}}_{k} \| \leq \varrho_{*} - \delta \) for sufficiently large \( k \). This condition is necessary to ensure that inequality (d) in \eqref{eq:prf-global-convergence} holds. Essentially, it guarantees that as the iterates approach a stationary point, \emph{i.e.}, \( \tilde{\matr{H}}_k \to \matr{0} \), the movement of the iterates also decreases. Therefore, this assumption cannot be removed. A similar requirement also appears in \cref{thm:nonmonotone-global-convergence}.\\ 
(ii)
In \cref{thm:Armijo-global-convergence}, we assume that there exists \(\delta \in (0, \varrho_{*}]\) such that \( \tau_k \| \hat{\matr{H}}_{k}\| \leq \varrho_{*} - \delta \). 
In contrast, \cref{thm:Armijo-stepsize-lower-bound-convergence-rate} does not rely on this assumption: even without it, our proof shows that the stepsize still admits a natural lower bound of \( \rho_{*} / (2 \Delta_{\hat{\matr{H}}}) \). 
This relaxation is crucial, as it allows for large movements in the normal direction in the TGP algorithms, thus generating iterates that can differ significantly from those produced by retraction-based algorithms; see, for instance, the illustration in \cref{fig:exa_QPInhomo}.
\end{remark}

\section{TGP algorithms using the Zhang-Hager type nonmonotone Armijo stepsize}\label{sec:nonmonotone_TGP}

\subsection{TGP-NA algorithm}
In addition to the Armijo stepsize which belongs to the monotone approach, the nonmonotone rules are also widely used, because, according to the authors of \cite{zhang2004nonmonotone,dai2002nonmonotone,toint1996assessment}, ``nonmonotone schemes can improve the likelihood of finding a global optimum; also, they can improve convergence speed in cases where a monotone scheme is forced to creep along the bottom of a narrow curved valley".
In this section, we study \cref{alg:TGP} utilizing the Zhang-Hager type nonmonotone Armijo stepsize \cite{zhang2004nonmonotone}, and call it the \emph{Transformed Gradient Projection with Zhang-Hager Type Nonmonotone Armijo stepsize} (TGP-NA) algorithm. In each iteration, the stepsize \(\tau_k\) is selected by backtracking with parameters $ \gamma, \beta \in (0, 1) $ and initial guess \(\hat{\tau}_k > 0\) as follows:
\begin{equation}\label{eq:nonmonotone-step-size}
\tau_k = \max\{ \hat{\tau}_k \beta^{i}: f(\matr{Z}_k(\hat{\tau}_k\beta^{i})) - c_k \leq \gamma \hat{\tau}_k \beta^{i} \langle \nabla f(\matr{X}_k), \matr{Z}'_k(0)\rangle, i \in \NN \}.
\end{equation}
Here \(c_{k+1} = ( \eta_k q_{k}c_{k} + f(\matr{X}_{k+1})) / q_{k+1}\) is the reference value of line-search, \(q_{k+1}= \eta_k q_{k} + 1\), \(\eta_k \in [0, 1)\), \(q_0 = 1\) and \(c_0 = f(\matr{X}_0)\).
In other words, \(\tau_k\) is the largest one among $ \{ \hat{\tau}_k, \hat{\tau}_k\beta^1, \hat{\tau}_k\beta^2 , \ldots , \hat{\tau}_k\beta^{i}, \ldots  \} $ satisfying the following nonmonotone Armijo-type condition:
\begin{equation}
\label{eq:nonmonotone-Armijo-condition}
f(\matr{Z}_k(\tau_k)) - c_k \leq \gamma \tau_k \langle \nabla f(\matr{X}_k), \matr{Z}_k'(0)  \rangle = - \gamma \tau_k \langle \grad f(\matr{X}_k), \matr{H}_k \rangle.
\end{equation}
It is easy to see that \(c_k\) is a convex combination of \( f(\matr{X}_0), f(\matr{X}_1) , \ldots , f(\matr{X}_k) \) and \( \{ f(\matr{X}_k) \}_{k \geq 0} \) is not necessarily decreasing. Moreover, if \( \eta_k = 0 \) for all \( k \in \NN \), then TGP-NA reduces to TGP-A in \cref{sec:TGP-A}. Similar to TGP-A, we assume that there exist \(\hat{\tau}^{(u)}, \hat{\tau}^{(l)} > 0\) such that \(\hat{\tau}^{(l)} \leq \hat{\tau}_k \leq \hat{\tau}^{(u)}\) for all \( k \in \NN \).
The following lemma can be obtained immediately by mimicking the proof of \cite[Lem. 1.1]{zhang2004nonmonotone}.
\begin{lemma}\label{lem:nonmonotone-g-bound-by-C}
In TGP-NA algorithm, if $ \matr{L}_k $ and $\matr{R}_k$ satisfy \cref{asp:H-tilde-grad-equiv}, then  $ f(\matr{X}_k) \leq c_k $ for all $ k \in \NN $ and the sequence $ \{ c_k \}_{k \geq 0} $ is non-increasing.
\end{lemma}

\begin{proof}
By \eqref{eq:H-grad-product} and \eqref{eq:nonmonotone-Armijo-condition}, we see that $ f(\matr{X}_{k+1}) \leq c_{k} $  for all $k \in \NN$.
Define \(\pi_k(t) \eqdef \frac{tc_{k-1} + f(\matr{X}_k)}{t+1}\) for $t \geq 0$. Note that $ \pi_k'(t) = \frac{c_{k-1} - f(\matr{X}_k)}{(t+1)^2}\geq 0$.
The function \(\pi_k(t)\) is non-decreasing and $ f(\matr{X}_k) = \pi_k(0) \leq \pi_k(\eta_{k-1} q_{k-1}) = c_k $.
Then, it follows from the definitions of \(c_k\) and \(q_k\) that
\[ c_{k} = (\eta_{k-1} q_{k-1}c_{k-1} + f(\matr{X}_k)) / q_k \leq (\eta_{k-1} q_{k-1}c_{k-1} + c_{k-1})/q_k = c_{k-1}.  \]
The proof is complete.
\end{proof}

\subsection{Weak convergence and iteration complexity}

By refining the proof of \cref{lem:Armijo-stepsize-general-weak-convergence}, the following result can be similarly derived for the TGP-NA algorithm.
Here, we omit the detailed proof for simplicity.

\begin{theorem}\label{thm:weak_converg-na}
Let \( \mm \subseteq \RR^{n \times r} \) be a submanifold of class \(C^2\) and the cost function $f$ in \eqref{eq:objec_func_g} be continuously differentiable over $\RR^{n \times r}$.
In TGP-NA algorithm, if $\matr{L}_k$ and $\matr{R}_k$ satisfy \cref{asp:H-tilde-grad-equiv}, $\matr{H}_{k}$ satisfies \cref{asp:H-boundedness} and there exists a constant \(\bar{\eta} \in [0, 1)\) such that \(\eta_k \leq \bar{\eta}\) for all \( k \in \NN \), then every accumulation point of \( \{ \matr{X}_k \}_{k \geq 0} \) is a stationary point of $f$. Moreover, we have that $ \lim_{k \to \infty} \| \grad f(\matr{X}_k) \| = 0 $.
\end{theorem}

Similar to the TGP-A algorithm discussed in \cref{subsec:conver_rate_tgp_a}, we now demonstrate that the stepsize \(\tau_k\) in the TGP-NA algorithm has a positive lower bound, and then derive the complexity result based on this lower bound.

\begin{theorem}\label{thm:nonmonotone-Armijo-stepsize-lower-bound-convergence-rate}
  Let \( \mm \subseteq \RR^{n \times r} \) be a compact submanifold of class \(C^3\) and the cost function $f$ in \eqref{eq:objec_func_g} have a Lipschitz continuous gradient in the convex hull of $\mm$.
In TGP-NA algorithm, if \(\matr{L}_{k}\) and \( \matr{R}_k \) satisfy \cref{asp:H-tilde-grad-equiv}, \(\matr{H}_{k}\) satisfies \cref{asp:H-boundedness}, and there exists a constant \( \bar{\eta} \in [0, 1) \) such that \( \eta_k \leq \bar{\eta} \) for all \( k \in \NN \), then \\
(i) there exists \(\tilde{\tau} > 0\) such that \(\tau_k \geq  \tilde{\tau}\) for all $ k \in \NN$;\\
(ii) $ \lim_{k \to \infty} \|\grad f(\matr{X}_k)\| = 0 $, implying that every accumulation point of \( \{ \matr{X}_k \}_{k \geq 0} \) is a stationary point;\\
(iii) for any $ K \in \NN $, we have that
\[ \min_{0 \leq k \leq K}\| \grad f(\matr{X}_k) \| \leq \sqrt{\frac{f(\matr{X}_{0}) - f^{*}}{\gamma \tilde{\tau} \upsilon (1-\bar{\eta})(K+1)}}. \]
\end{theorem}
\begin{proof}
Note that $ f(\matr{X}_k) \leq c_k $ by \cref{lem:nonmonotone-g-bound-by-C}.
It follows from \cref{thm:Riemannian-descent-lemma-proj} that the nonmonotone Armijo-type condition \eqref{eq:nonmonotone-Armijo-condition} is satisfied if \(\tau\) satisfies \(\tau \leq \rho_{*} / (2\Delta_{\hat{\matr{H}}}) \) and
\[ \tau \langle \grad f(\matr{X}_k), \matr{H}_k \rangle - \tau^2 \left(\Gamma_1^{(\delta_{*})}  \| \tilde{\matr{H}}_k \| ^2 - \Gamma_2^{(\delta_{*})} \| \grad f(\matr{X}_k) \| \| \tilde{\matr{H}}_k \| \| \hat{\matr{H}}_k \|\right) \geq \gamma \tau \langle \grad f(\matr{X}_k), \matr{H}_k \rangle,  \]
where \(\delta_{*} = \rho_{*} / 2 \).
The rest of (i) can be proved in a manner similar to \cref{thm:Armijo-stepsize-lower-bound-convergence-rate}.

By the nonmonotone Armijo-type condition \eqref{eq:nonmonotone-Armijo-condition} and the definition of \(c_{k+1}\), we have
\begin{equation}\label{eq:proof-nonmonotone-convergence-rate-1}
\frac{\gamma \tau_k\left\langle \grad f\left(\matr{X}_k\right), \matr{H}_k\right\rangle}{q_{k+1}} \leq \frac{c_k-f\left(\matr{X}_{k+1}\right)}{q_{k+1}} =c_k-c_{k+1}.
\end{equation}
It follows from the definition of \(q_k\) that \(q_{k+1}\leq  \sum_{i=0}^k \bar{\eta}^i \leq 1 / (1- \bar{\eta})\). Combining the inequality \eqref{eq:proof-nonmonotone-convergence-rate-1}, \(\tau_k \geq \tilde{\tau}\) and \eqref{eq:H-grad-product}, we have
\begin{equation}\label{eq:proof-nonmonotone-convergence-rate-Ck-descent}
c_k - c_{k+1} \geq (1-\bar{\eta}) \gamma \tau_k \langle \grad f(\matr{X}_k), \matr{H}_k \rangle \geq (1-\bar{\eta})  \gamma \tilde{\tau} \upsilon\| \grad f(\matr{X}_k) \|^2.
\end{equation}
For any $ K \in \NN $, summing \eqref{eq:proof-nonmonotone-convergence-rate-Ck-descent} for $0\leq k \leq K $ and using the fact that $ f^{*} \leq f(\matr{X}_{K+1}) \leq c_{K+1} $ and $ f(\matr{X}_0) = c_0 $, we have
\begin{equation}\label{eq:nonmonotone-gx0-gstar}
f(\matr{X}_0) - f^{*} \geq  c_0 - c_{K+1} \geq (1-\bar{\eta})  \gamma \tilde{\tau} \upsilon \sum_{k=0}^{K} \| \grad f(\matr{X}_k) \|^2,
\end{equation}
which implies that  $\lim_{k \to \infty} \|\grad f(\matr{X}_k)\| = 0$ and
\[ \min_{0 \leq k \leq K}\| \grad f(\matr{X}_k) \| \leq \sqrt{\frac{f(\matr{X}_{0}) - f^{*}}{\gamma \tilde{\tau} \upsilon (1-\bar{\eta})(K+1)}}. \]
The proof is complete.
\end{proof}

\begin{remark}
  When $\mm$ is a compact Riemannian manifold, if the derivative of $f$ has a Lipschitz continuous property in \cite[Assumption 1 (2)]{oviedo2022global}, the weak convergence of RetrLS algorithms using Zhang-Hager type nonmonotone Armijo stepsize was established in \cite[Thm. 1]{oviedo2022global}, following an approach similar as in \cite{zhang2004nonmonotone}, which is for the unconstrained optimization algorithms on the Euclidean space.
When the retraction is constructed using the projection,
in \cref{thm:nonmonotone-Armijo-stepsize-lower-bound-convergence-rate}(ii), we prove a more general result which allows for the presence of a normal component in \( \matr{H}_k\), which makes \( \matr{H}_k \) not necessarily tangent to $\mm$ at \( \matr{X}_k \).
\end{remark}

\subsection{Global convergence}\label{subsec:NA-global-conver}

\begin{lemma}\label{lem:nonmonotone-global-convergence}
Let $\mm'\subseteq\RR^m$ be an analytic submanifold and $g:\mm'\rightarrow\RR$ be a real analytic function.
Let $\{\vect{x}_k\}_{k\geq 0}\subseteq\mm'$ be a sequence having at least one accumulation point $\vect{x}_{*}$. Let $\{c_{k}\}_{k\geq 0}\subseteq\RR$ be an non-increasing sequence satisfying \(g(\vect{x}_k)\leq c_k\) for all $ k\geq 0 $. Suppose that\\
(i) there exist positive constants $\kappa, \psi>0$ such that for large enough $k$,
\begin{equation}\label{eq:nonmonotone-global-convergence-sufficient-decrease}
c_k-c_{k+1}\geq \kappa \| \grad g(\vect{x}_k) \|^2 + \psi \| \vect{x}_{k+1} - \vect{x}_k \|^2;
\end{equation}
(ii) \( \sum_{i=0}^{\infty} (c_k - g(\vect{x}_k))^{\theta} < +\infty\), where \(\theta\) is the exponent as shown in \eqref{eq:Lojasiewicz} at $ \vect{x}_{*} $.\\
Then $\lim_{k \to \infty}\vect{x}_k = \vect{x}_{*}$ and \( \vect{x}_{*} \) is a stationary point of \(g\).
\end{lemma}

\begin{proof}
If the sequence $ \{ c_k \}_{k \geq 0} $ has no lower bound, then it follows from $g(\vect{x}_k)\leq c_k$ and the monotony of $\{ c_k \}_{k \geq 0}$ that $ g(\vect{x}_k) \to -\infty $ as $ k \to \infty $, which contradicts the fact that $ \{ \vect{x}_k \}_{k \geq 0} $ has an accumulation point $ \vect{x}_{*} $. Thus, the non-increasing sequence $ \{ c_k \}_{k \geq 0} $ is lower bounded and convergent. Denote $c_{*} \eqdef \lim_{k \to \infty}c_k$. It follows from (ii) that $\lim_{k \to \infty}g(\vect{x}_k) = c_{*}$.

Without loss of generality, we assume that the conditions (i) and (ii) hold for all $ k \geq 0 $.
If there exists \(k_0\) such that \(c_{k_0} = c_{*}\), then \(c_k = c_{*}\) for all $ k\geq k_0 $ since \(\{ c_k \}_{k \geq 0}\) converges to $ c_{*} $ monotonically. It follows from \eqref{eq:nonmonotone-global-convergence-sufficient-decrease} that $ \vect{x}_{k+1} = \vect{x}_k $ for $ k \geq k_0 $, and so $ \{ \vect{x}_k \}_{k \geq 0} $ is convergent.
Now we consider the case where \(c_k > c_{*}\) for all $ k \geq 0 $. For simplicity, we assume that \(c_{*} = 0\). For all $ k \geq 0 $, it follows from the mean value theorem that there exists $ \bar{c}_{k} \in [c_{k+1}, c_k] $ such that
\begin{equation*}
c_k^{1-\theta} - c_{k+1}^{1-\theta} = (1-\theta)\frac{c_k - c_{k+1}}{\bar{c}_{k}^{\theta}} \geq (1-\theta)\frac{c_k-c_{k+1}}{c_{k}^{\theta}}.
\end{equation*}
Combing the above inequality with \eqref{eq:nonmonotone-global-convergence-sufficient-decrease}, we have
\begin{equation}\label{eq:proof-nonmonotone-global-convergence-1}
c_k^{1-\theta} - c_{k+1}^{1-\theta} \geq \left(1-\theta\right)\frac{c_k-c_{k+1}}{c_{k}^{\theta}} \geq \left(1-\theta\right) \frac{\kappa \| \grad g(\vect{x}_k) \|^2 + \psi \| \vect{x}_{k+1} - \vect{x}_k \|^2}{c_k^{\theta}}.
\end{equation}
By \cref{lemma-SU15}, there exists $\varepsilon,\varsigma > 0$ such that for $ \vect{x} \in \mm'\cap\openball{\vect{x}_{*}}{\varepsilon}$,
\begin{equation}\label{eq:proof-nonmonotone-global-convergence-KL-inequality}
|g(\vect{x})|^{\theta} = |g(\vect{x}) - g(\vect{x}_{*})|^{\theta} \leq \varsigma \| \grad g(\vect{x}) \|,
\end{equation}
where $\theta\in [\frac{1}{2}, 1)$.
Note that $ (a+b)^{\theta} \leq a^{\theta} + b^{\theta} $ for $a, b \geq 0$.
Then for $ \vect{x}_k $ satisfying $ \| \vect{x}_k - \vect{x}_{*} \| < \varepsilon $, we have
\[ c_k^{\theta} \leq (c_k - g(\vect{x}_k) + |g(\vect{x}_k)|)^{\theta} \leq (c_k - g(\vect{x}_k))^{\theta} + |g(\vect{x}_k)|^{\theta} \leq (c_k - g(\vect{x}_k))^{\theta} + \varsigma \| \grad g(\vect{x}_k) \|.  \]
Substituting the term \(c_k^{\theta}\) in \eqref{eq:proof-nonmonotone-global-convergence-1} by the right-hand side of the above inequality, we have
\begin{align*}
& (1-\theta) \left(\kappa \| \grad g(\vect{x}_k) \|^2 + \psi \| \vect{x}_{k+1} - \vect{x}_k \|^2 \right)\\
& \leq \left(c_k^{1-\theta} - c_{k+1}^{1-\theta}\right)\left((c_k - g(\vect{x}_k)^{\theta} + \varsigma \| \grad g(\vect{x}_k) \|\right) \\
& = \left(c_k^{1-\theta} - c_{k+1}^{1-\theta}\right)\left(c_{k} - g(\vect{x}_k)\right)^{\theta} + \varsigma \left(c_k^{1-\theta} - c_{k+1}^{1-\theta}\right) \| \grad g(\vect{x}_k) \|\\
& \leq \left(\frac{1}{4} + \frac{\varsigma^2}{4(1-\theta)\kappa}\right)\left(c_k^{1-\theta} - c_{k+1}^{1-\theta}\right)^{2} + (c_k - g(\vect{x}_k))^{2\theta} + (1-\theta)\kappa \| \grad g(\vect{x}_k) \|^2,
\end{align*}
where the last inequality follows from that
\begin{align*}
\left(c_k^{1-\theta} - c_{k+1}^{1-\theta}\right)(c_{k} - g(\vect{x}_k))^{\theta} &\leq \frac{1}{4} \left(c_k^{1-\theta} - c_{k+1}^{1-\theta}\right)^2 + (c_{k} - g(\vect{x}_k))^{2\theta},\\
\varsigma\left(c_k^{1-\theta} - c_{k+1}^{1-\theta}\right) \| \grad g(\vect{x}_k) \| &\leq \frac{\varsigma^2}{4(1-\theta)\kappa}\left(c_k^{1-\theta} - c_{k+1}^{1-\theta}\right)^2 + (1-\theta)\kappa \| \grad g(\vect{x}_k) \|^2.
\end{align*}
Then for $\vect{x}_k$ satisfying $\|\vect{x}_k - \vect{x}_{*} \| < \epsilon$, we have
\begin{align}
\| \vect{x}_{k+1} - \vect{x}_k \|
& \leq \frac{1}{\sqrt{(1-\theta)\psi}} \sqrt{\left(\frac{1}{4} + \frac{\varsigma^2}{4(1-\theta)\kappa}\right)\left(c_k^{1-\theta} - c_{k-1}^{1-\theta}\right)^{2} + \left(c_k - g(\vect{x}_k)\right)^{2\theta}} \notag \\
& \leq \frac{1}{\sqrt{(1-\theta)\psi}} \left(\sqrt{\frac{(1-\theta)\kappa + \varsigma^2}{4(1-\theta)\kappa}} \left(c_k^{1-\theta} - c_{k+1}^{1-\theta}\right) + \left(c_k - g(\vect{x}_k)\right)^{\theta}\right). \label{eq:proof-nonmonotone-global-convergence-x-diff}
\end{align}

Since $ \vect{x}_{*} $ is an accumulation point of $ \{ \vect{x}_k \}_{k \geq 0} $, and $\sum_{k=1}^{\infty} (c_k - g(\vect{x}_k))^{\theta} < +\infty$, there exists \(k_1\) such that
\[ \| \vect{x}_{k_1} - \vect{x}_{*} \| < \frac{\varepsilon}{3},\ \ \sqrt{\frac{(1-\theta)\kappa + \varsigma^2}{4(1-\theta)^2 \kappa\psi}}  \sum_{k=k_1}^{\infty}\left(c_k^{1-\theta} - c_{k+1}^{1-\theta}\right) < \frac{\varepsilon}{3},\ \ \frac{1}{\sqrt{(1-\theta)\psi}} \sum_{k=k_1}^{\infty} (c_k - g(\vect{x}_k))^{\theta}< \frac{\varepsilon}{3}.  \]
It can be shown by induction with the above inequalities and \eqref{eq:proof-nonmonotone-global-convergence-x-diff} that $\| \vect{x}_k - \vect{x}_{*} \| < \varepsilon$ for all $ k \geq k_1 $. By summing \eqref{eq:proof-nonmonotone-global-convergence-x-diff} up for $ k $ from \(k_1 \) to $ \infty $, we have $\sum_{k=k_1}^{\infty}\| \vect{x}_{k+1} - \vect{x}_k \| < +\infty$. Thus, $\{ \vect{x}_k \}_{k \geq 0}$ converges to its accumulation point \(\vect{x}_{*}\). Moreover, it follows from \eqref{eq:nonmonotone-global-convergence-sufficient-decrease} and \(\lim_{k \to \infty}c_k = c_{*}\) that \( \lim_{k \to \infty} \grad g(\vect{x}_k) = \vect{0}\), implying that \( \vect{x}_{*} \) is a stationary point of \( g \).
The proof is complete.
\end{proof}

\begin{theorem}\label{thm:nonmonotone-global-convergence}
Let $ \mm \subseteq \RR^{n\times r}$ be an analytic compact matrix submanifold and $f$ in \eqref{eq:objec_func_g} be an analytic function.
In TGP-NA algorithm, if $\matr{L}_k$ and $\matr{R}_k$ satisfy \cref{asp:H-tilde-grad-equiv}, $\matr{H}_{k}$ satisfies \cref{asp:H-boundedness}, and there exist constants \( \delta \in (0, \varrho_{*}] \) and \( \bar{\eta} \in [0, 1) \) such that $ \tau_k \| \hat{\matr{H}}_k \| \leq \varrho-\delta $ and $\eta_k \leq \min\{  \frac{1}{q_{k}((c_{k} - f(\matr{X}_{k+1}))(k+1)^{4} - 1)}, \bar{\eta} \} $ for sufficiently large \( k \), then the sequence \( \{ \matr{X}_k \}_{k \geq 0} \) converges to a stationary point \( \matr{X}_{\ast} \).
\end{theorem}

\begin{proof}
Since the sequence \( \{ \matr{X}_k \}_{k \geq 0} \subseteq \mm \), it has an accumulation point.
It is sufficient to verify the conditions (i) and (ii) in \cref{lem:nonmonotone-global-convergence}.
For sufficiently large \( k \), it follows from \eqref{eq:H-grad-norm-equiv} and \eqref{eq:first-order-boundedness} that
\[
\|\grad f(\matr{X}_k)\| \geq \frac{1}{\varpi} \| \tilde{\matr{H}}_k \| \geq \frac{1}{\tau_{k}L_0^{(\delta)}\varpi} \| \matr{X}_{k+1} - \matr{X}_{k} \| \geq \frac{1}{\hat{\tau}^{(u)}L_0^{(\delta)}\varpi}\| \matr{X}_{k+1} - \matr{X}_{k} \|.
\]
By our assumption, \( \eta_k \) is uniformly upper bounded.  Substituting it into \eqref{eq:proof-nonmonotone-convergence-rate-Ck-descent}, we have
\begin{equation*}
c_{k+1}-c_{k} \geq (1-\bar{\eta})\gamma\tilde{\tau}\upsilon \| \grad f(\matr{X}_k) \|^2 \geq   \frac{(1-\bar{\eta})\gamma \tilde{\tau}\upsilon}{(\hat{\tau}^{(u)}L_0^{(\delta)}\varpi)^2} \| \matr{X}_{k+1} - \matr{X}_k \|^2,
\end{equation*}
which implies that
\[
c_{k+1}-c_{k} \geq \frac{1}{2} (1-\bar{\eta})\gamma\tilde{\tau}\upsilon \| \grad f(\matr{X}_k) \|^2 + \frac{(1-\bar{\eta})\gamma \tilde{\tau}\upsilon}{2(\hat{\tau}^{(u)}L_0^{(\delta)}\varpi)^2} \| \matr{X}_{k+1} - \matr{X}_k \|^2.
\]
Thus, the condition (i) of \cref{lem:nonmonotone-global-convergence} is satisfied.
Note that \(c_{k+1} - f(\matr{X}_{k+1}) =  \frac{\eta_kq_{k}}{\eta_{k}q_{k}+1}(c_{k}  - f(\matr{X}_{k+1}))\). It follows from the direct computation that \(c_{k+1} - f(\matr{X}_{k+1})\leq 1 / (k+1)^4\) when $\eta_k \leq \frac{1}{q_{k}((c_{k} - f(\matr{X}_{k+1}))(k+1)^{4} - 1)}$. Since \(\theta \geq 1 / 2\) in \cref{lemma-SU15}, we have \( (c_k - f(\matr{X}_k))^{\theta} \leq 1/k^{2} \) for sufficiently large \( k \), implying that the condition (ii) of \cref{lem:nonmonotone-global-convergence} is also satisfied.
Then the proof is complete by \cref{lem:nonmonotone-global-convergence}.
\end{proof}

\begin{remark}
While the Zhang-Hager type nonmonotone Armijo stepsize has been used in the RetrLS and ProjLS algorithms on Riemannian manifold \cite{oviedo2019non,oviedo2022global,oviedo2023worst}, to our knowledge, their global convergence has not yet been studied in the literature\footnote{In \cite{oviedo2019non,oviedo2022global,oviedo2023worst,zhang2004nonmonotone}, the term ``global convergence'' has a different meaning with our paper; it means that every accumulation point of the iterates is a stationary point, referred to as ``weak convergence'' in our paper.}, even for the unconstrained nonconvex problems on the Euclidean space.
In \cref{thm:nonmonotone-global-convergence}, for the first time, we establish the global convergence of ProjLS algorithms using the nonmonotone Armijo stepsize.
It is easy to see that the global convergence of the Euclidean space case can be established similarly as in the proof of \cref{thm:nonmonotone-global-convergence}.
Moreover, as a nonmonotone analogue of \cite[Thm.  2.3]{SU15:pro}, \cref{lem:nonmonotone-global-convergence} we have proved in this paper can also contribute to establishing the global convergence of other nonmonotone algorithms.
\end{remark}

\section{TGP algorithms using a fixed stepsize}\label{sec:fixed_TGP}

In \cref{alg:TGP}, except the two types of stepsizes introduced in \cref{sec:TGP-A,sec:nonmonotone_TGP}, for a fixed positive constant \( \delta \in (0, \varrho_{*}] \), it is also possible to choose a fixed stepsize  \(\tau_{*}\) satisfying \(\tau_{*} \|\hat{\matr{H}}_k\| \leq \varrho_{*} - \delta\) for all \( k \in \NN \) and \(\tau_{*} < \frac{\upsilon}{\Gamma_1^{(\delta)} \varpi^2 + \Gamma_2^{(\delta)} \Delta_{\hat{\matr{H}}} \varpi}\). This can be equivalently expressed as
\[ \tau_{*} \in \left\{
\begin{array}{ll}
\left(0, \min \left\{\frac{\varrho_{*} - \delta}{\Delta_{\hat{\matr{H}}}}, \frac{\upsilon}{\Gamma_1 ^{(\delta)}\varpi^2 + \Gamma_2^{(\delta)} \Delta_{\hat{\matr{H}}} \varpi}\right\}\right) & \text{if } \Delta_{\hat{\matr{H}}} > 0; \\
\left(0 , \frac{\upsilon}{\Gamma_{1}^{(\varrho_{*})}\varpi^2}\right) & \text{if } \Delta_{\hat{\matr{H}}} = 0.
\end{array}
\right.  \]
In this case, we call it the \emph{Transformed Gradient Projection with a fixed stepsize} (TGP-F) algorithm.

\begin{lemma}\label{lem:TGP-F-Armijo-like-condition}
In the TGP-F algorithm, for all $ k\in\NN $, we have
\begin{equation}\label{eq:Riemannian-descent-lemma-proj}
f(\matr{X}_k) - f(\matr{X}_{k+1}) \geq \gamma_{*} \tau_{*}\langle \grad f(\matr{X}_k), \matr{H}_k \rangle,
\end{equation}
where \(\gamma_{*} \eqdef 1 - \tau_{*}(\Gamma_1^{(\delta)}\varpi^2 + \Gamma_2 ^{(\delta)}\Delta_{\hat{\matr{H}}}\varpi) / \upsilon \) satisfying \( \gamma_{*} \in (0, 1) \) by the definition.
\end{lemma}

\begin{proof}
The following calculations are similar to the proof of \cref{thm:Armijo-stepsize-lower-bound-convergence-rate}(i). Combining  \eqref{eq:H-grad-product} and \eqref{eq:Riemannian-descent-lemma-proj}, we have
\begin{align*}
&\tau_{*}\left(\Gamma_1^{(\delta)} \| \tilde{\matr{H}}_k \|^2 + \Gamma_2^{(\delta)} \| \grad f(\matr{X}_k) \| \| \tilde{\matr{H}}_k \| \| \hat{\matr{H}}_k \|\right) \leq \tau_{*}\left(\Gamma_1^{(\delta)} \varpi^2 + \Gamma_2^{(\delta)}\Delta_{\hat{\matr{H}}}\varpi\right) \| \grad f(\matr{X}_k) \|^2  \\
&\leq \tau_{*} \frac{\Gamma_1^{(\delta)} \varpi^2 + \Gamma_2^{(\delta)}\Delta_{\hat{\matr{H}}}\varpi}{\upsilon} \left\langle \grad f(\matr{X}_k), \matr{H}_k \right\rangle
= (1-\gamma_{*}) \langle \grad f(\matr{X}_k), \matr{H}_k \rangle.
\end{align*}
Then, it follows from \cref{thm:Riemannian-descent-lemma-proj} that
\begin{align*}
f(\matr{X}_k) - f(\matr{X}_{k+1}) & \geq \tau_{*}\left(\langle \grad f(\matr{X}_k), \matr{H}_k \rangle - \tau_{*}\left(\Gamma_1^{(\delta)} \| \tilde{\matr{H}}_k \|^2 + \Gamma_2^{(\delta)} \| \grad f(\matr{X}_k) \| \| \tilde{\matr{H}}_k \| \| \hat{\matr{H}}_k \|\right)\right)  \\
& \geq \gamma_{*}\tau_{*}\langle \grad f(\matr{X}_k), \matr{H}_k \rangle.
\end{align*}
The proof is complete.
\end{proof}

Using \cref{lem:TGP-F-Armijo-like-condition} and following the proofs of \cref{thm:Armijo-stepsize-lower-bound-convergence-rate,thm:Armijo-global-convergence}, we can obtain the following convergence results about TGP-F algorithm in a similar manner.

\begin{theorem}\label{thm:TGP-F-weak-convergence-convergence-rate}
  Let \( \mm \subseteq \RR^{n \times r} \) be a compact submanifold of class \(C^3\) and the cost function $f$ in \eqref{eq:objec_func_g} have a Lipschitz continuous gradient in the convex hull of $\mm$.
In TGP-F algorithm, if $\matr{L}_k$ and $\matr{R}_k$ satisfy \cref{asp:H-tilde-grad-equiv}, and $\matr{H}_{k}$ satisfies \cref{asp:H-boundedness}, then \\
(i) \( \lim_{k \to \infty} \| \grad f(\matr{X}_k) \| = 0\), implying that every accumulation point of \( \{ \matr{X}_k \}_{k \geq 0} \) is a stationary point;\\
(ii) for any $ K \in \NN $, we have that
\[ \min_{0 \leq k \leq K}\| \grad f(\matr{X}_k) \| \leq \sqrt{\frac{f(\matr{X}_{0}) - f^{*}}{\gamma_{*}\tau_{*} \upsilon(K+1)}}. \]
\end{theorem}

\begin{theorem}\label{thm:TGP-F-global-convergnce}
Let $ \mm \subseteq \RR^{n\times r}$ be an analytic compact matrix submanifold and $f$ in \eqref{eq:objec_func_g} be an analytic function.
In the TGP-F algorithm, if $\matr{L}_k$ and $\matr{R}_k$ satisfy \cref{asp:H-tilde-grad-equiv}, $\matr{H}_{k}$ satisfies \cref{asp:H-boundedness}, then the sequence \( \{ \matr{X}_k \}_{k \geq 0} \) converges to a stationary point and the estimation of the convergence speed in \eqref{eq:KL_convergence_rate} holds.
\end{theorem}

\begin{remark}\label{remark:tgp-f-conv}
In the RetrLS and ProjLS algorithms on a compact manifold including
\( \St(r, n) \) and $\Gr(p,n)$ as special cases, the fixed stepsize has been extensively used in the literature \cite{chen2009tensor,gao2018new,yang2019epsilon,li2019polar}, as well as the weak convergence and global convergence of the corresponding algorithms. For example, in \cite{li2019polar}, the convergence analysis of the case where \( \matr{H}_k = \nabla f(\matr{X}_{k}) \) or \( \matr{H}_k = \nabla f(\matr{X}_{k}) + \matr{X}_k/ \tau_{*} \) (corresponds to power method) is derived.
However, different from the fixed stepsize in these works, which uses the descent lemma to establish their convergence properties, in this paper, we use the geometric results in \cref{lem:first-second-order-boundedness} and a key inequality in \cref{thm:Riemannian-descent-lemma-proj} relevant to the tangent space and normal space of the submanifold. As a result, the conditions required for convergence in our framework are also different from those in \cite{chen2009tensor,gao2018new,yang2019epsilon,li2019polar}.
\end{remark}

\section{Numerical experiments}\label{sec:numer_exper}

For the question summarized in \cref{subsec:limitations_problem}, we have provided a \emph{positive} answer in the previous sections, by proposing the TGP algorithmic framework and establishing various convergence properties across three stepsizes in a general setting. 
Compared to the classical EGP and RGD algorithms, the search direction \( \matr{H}_k \) of the general TGP algorithmic framework is further endowed with scaling matrices \( \matr{L}_k \), \( \matr{R}_k \), and an additional normal vector \( \matr{N}_k \). 
This naturally leads to the following question:

\begin{itemize}
\item[] \emph{Could these components within the framework improve the numerical performance of the TGP algorithms?}
\end{itemize}

In this section, we will further provide a \emph{positive} answer to the above question, emphasizing the importance of the scaling matrices and the additional normal vector.
This demonstrates that introducing them is not merely for generality but can also lead to superior numerical performance in certain special problems.
We will focus on several concrete examples, and introduce new special-case algorithms, TGP-$\ast$-E/R in \eqref{eq:H_TGP_R_E}, TGP-$\ast$-DE/DF in \eqref{eq:TGP-DEF}, and TGP-A-Eigen in \eqref{eq:H-TGP-eigen}, from our TGP framework.
These algorithms distinguish themselves from the classical RGD and EGP algorithms, and we will provide both theoretical insights and numerical results to validate their effectiveness.

\subsection{Effect of normal components with a fixed tangent component}\label{subsec:tgp_diffe_norm}

In this subsection, we conduct numerical experiments to investigate the effect of the normal component of the search direction. Specifically, we compare different search directions that share the same tangent component, but differ in their normal components. Our results will show that, even with the same tangent component, different normal components of \( \matr{H}_k \) can lead to significantly different performance in practice. Consequently, in certain cases, the introduction of the additional component \( \matr{N}_k \) can be critical in enhancing performance.

To ensure a fair comparison with the classical RGD and EGP algorithms, we always set the tangent component of $\matr{H}_k$ as $\grad f(\matr{X}_k)$. The search directions we use are as follows:
\begin{equation}\label{eq:H_TGP_R_E}
  \begin{split}
  \text{TGP-$\ast$-R}:\quad & \matr{H}_k = \grad f(\matr{X}_{k}) + a \matr{X}_{k}\matr{S}_k,  \\
  \text{TGP-$\ast$-E}:\quad & \matr{H}_k = \nabla f(\matr{X}_{k}) + a \matr{X}_{k}\matr{S}_k = \grad f(\matr{X}_{k}) + \mathcal{P}_{\NormalSt{\matr{X}_k}}(\nabla f(\matr{X}_k)) + a\matr{X}_k\matr{S}_k,
  \end{split}
\end{equation}
where $a \in \RR$ is a hyper-parameter and $\matr{S}_k \in \symmm{\RR^{r \times r}}$ is chosen by us manually. 
When $a$ = 0, the above two TGP algorithms reduce to the classical RGD and EGP algorithms, respectively. 
Note that \( \matr{X}_k \matr{S}_k \in \NormalSt{\matr{X}_k}\) by \eqref{def-St-normal-space}.
As indicated in equation \eqref{eq:H_TGP_R_E}, the TGP algorithms with one of the above settings will be referred to as \emph{TGP-$\ast$-R} or \emph{TGP-$\ast$-E}, where the symbol $\ast$ represents the chosen stepsize, including \emph{A} (Armijo stepsize), \emph{NA} (Zhang-Hager type nonmonotone Armijo stepsize) and \emph{F} (fixed stepsize).

We consider three test problems on the Stiefel manifold.
For each problem, we conduct the following two types of experiments~\ref{exp1} and~\ref{exp2}:
\begin{enumerate}[label=Exp \arabic*, leftmargin=*]
\item \label{exp1} \textbf{Effect of normal component:} We study the impact of the normal component of \( \matr{H}_k \) by comparing TGP-A-R/E using different values of $a$. 
We will see that, even with the same tangent component \( \tilde{\matr{H}}_k = \grad f(\matr{X}_k) \), different choices of the normal component, such as \( \hat{\matr{H}}_k = a\matr{X}_k\matr{S}_k \) or \( \hat{\matr{H}}_k = \mathcal{P}_{\NormalSt{\matr{X}_k}}(\nabla f(\matr{X}_k)) + a\matr{X}_k\matr{S}_k \), can lead to significantly different performance in practice. 
In particular, since TGP-A-R/E with $a=0$ reduce to RGD and EGP, the experiments show that their performance can be improved by introducing the additional normal term $a \matr{X}_k \matr{S}_k$. 
This highlights the importance of the additional normal vector in our TGP framework.
\item \label{exp2} \textbf{Comparison with classical Riemannian optimization methods:} We implement
TGP-A/NA/F-E/R
algorithms in the framework of the {\tt manopt}\footnote{This package was downloaded from \url{https://www.manopt.org}.} package~\cite{JMLR:v15:boumal14a}, and compare them with three retraction-based Riemannian optimization methods from {\tt manopt}, including the {\emph{RGD}} (Riemannian gradient descent)\footnote{The algorithm is referred to as the \emph{Steepest Descent} solver in the {\tt manopt} package.}, 
\emph{CG} (Riemannian conjugate gradient) and \emph{BFGS} (Riemannian version of BFGS) algorithms.
It will be seen that TGP algorithms with a suitable choice of $a$ outperform these classical retraction-based algorithms, particularly in terms of the final \emph{solution quality}.
\end{enumerate}

\textbf{Performance metrics.} In each experiment, we randomly generate 500 instances, paired with corresponding initial points generated by projecting a random matrix onto $\St(r, n)$, and apply the relevant algorithms to solve these instances. The generation procedure will be clarified later.
We evaluate the performance of the algorithms using three metrics, including the \emph{average iteration number} (\emph{Niter}), \emph{average CPU time} (\emph{Time}), and \emph{solution quality}.
The solution quality is assessed as follows.
(i) For problems with a known global optimum, we simply report the number of random instances where the algorithm achieves the global solution\footnote{An algorithm is considered to have achieved a global solution if it terminates with an objective value less than $f_{*} + 10^{-4}$, where $f_{*}$ denotes the global minimum of this instance.}, denoted by \emph{NGlobal} to evaluate the solution quality. (ii) For problems with an unknown optimum, we select an algorithm as a baseline for comparison.
We record the number of instances where the algorithm finds a better or worse solution\footnote{In a random instance, if the baseline algorithm stops with an objective value \( f_b \), we consider another algorithm to have found a better solution if its final objective value \( < f_b - 10^{-4} \), and a worse solution if the value \( > f_b+10^{-4} \).} compared to the baseline algorithm, denoted by \emph{NBetter} and \emph{NWorse}, respectively.
We report the overall superior instance number by \emph{NSuper}, defined as $\emph{NSuper}\eqdef\emph{NBetter} - \emph{NWorse}$.

\textbf{Algorithmic details.} For each random instance, we first randomly generate a symmetric matrix \( \matr{S} \in \symmm{\RR^{r \times r}} \) with eigenvalues uniformly distributed in \( [0.5, 1.5] \), and fix \( \matr{S}_k\) in \eqref{eq:H_TGP_R_E} to be \( \matr{S} \) for all TGP variants. To ensure a fair comparison of computation time, we implement all TGP algorithms within the {\tt manopt} framework. 
In {\tt manopt}, both RGD and CG employ an adaptive line-search scheme, which updates the trial stepsize $\hat{\tau}_k$ using information from previous iterations. However, this strategy is not well-suited for TGP algorithms, where the search direction may not be tangent. To address this issue, we adopt a simple update rule: after determining  \(\tau_k\) via backtracking, if no backtracking occurs, i.e. \(\tau_k = \hat{\tau}_k\), we set $ \hat{\tau}_{k+1} = \min\{ 1.1 \cdot \hat{\tau}_k, \hat{\tau}_0 \}$; otherwise, we set $ \hat{\tau}_{k+1} = 0.9 \cdot \hat{\tau}_k$. For TGP-A and TGP-NA, without additional specification, the backtracking parameters are fixed as \(\gamma = \beta = 0.5\), and \(\hat{\tau}_0 = 1\). The maximum number of backtracking steps is set to 10. For TGP-F, the stepsize is specified for each test problem. 
For the retraction-based Riemannian optimization methods from {\tt manopt}, we use the default settings, except that the maximum number of iterations \texttt{maxiter}, the maximum time \texttt{maxtime}, and the gradient norm tolerance \texttt{tolgradnorm} are set to match those of TGP. 
We also disable the default stopping criterion \texttt{minstepsize}, and use the polar decomposition as the retraction.
All algorithms stop when any one of the following stopping criteria is met: (i) the norm of the Riemannian gradient is less than \texttt{tolgradnorm}, set as $10^{-4}$, (ii) the number of iterations exceeds \texttt{maxiter}, set as $10^4$, or (iii) the runtime exceeds \texttt{maxtime}, set as $5$ seconds. In the latter two cases, the algorithm is considered to have failed, and the number of such instances is denoted by \emph{NFail}.
All the computations are done using MATLAB R2025b and the Tensor Toolbox version 3.7 \cite{TTB_Software}. The MATLAB code generating the results in this paper is available upon reasonable request.

\begin{example}\label{exa:qp_inhomo}
We consider the following \emph{inhomogeneous quadratic optimization problem with orthogonality constraints}:
\begin{equation}\label{eq:problem-qp-inhomo}
 \min f(\matr{X}) = \frac{1}{2} \tr\left((\matr{X} - \matr{X}^{*})^{\T}\matr{A}(\matr{X} - \matr{X}^{*})\right), \ \matr{X}\in \St(r, n),
\end{equation}
where \( \matr{A} \in \symmm{\RR^{n \times n}} \) is positive semi-definite and \( \matr{X}^{*} \in \St(r, n) \). The global minimum of problem \eqref{eq:problem-qp-inhomo} is clearly \( 0 \). We set \( n = 3\) and \( r = 2 \).
For each random instance, the matrix $\matr{X}^{*}$ in~\eqref{eq:problem-qp-inhomo} is generated by projecting a random matrix onto $\St(r,n)$. 
The matrix $\matr{A}$ is generated in two different ways due to the huge difference between their numerical results.
\begin{enumerate}[label=Case \arabic*, leftmargin=*]
  \item \label{qp_inhomo_set_1} \( \matr{A} \) is generated randomly: $\matr{A} = \matr{B}\matr{B}^{\T}$, where $\matr{B} = \texttt{randn}(n, n)$.
\item \label{qp_inhomo_set_2} \( \matr{A} \) is generated randomly with eigenvalues between \( 9.9 \) and \( 10.1 \): \( \matr{A} = \matr{Q}^{\T}\texttt{diag}(\vect{d})\matr{Q} \), where \( \vect{d} = 9.9 + 0.2 \cdot \texttt{rand}(n, 1)\), and \( \matr{Q} = \mathcal{P}_{\ON{n}}(\texttt{randn}(n, n)) \).
\end{enumerate}
\end{example}
In \ref{exp2}, we set \( a = 1.1 \), \(\gamma = 10^{-4}\), \(\eta_k = 0.1\) for \ref{qp_inhomo_set_1}, and fix the stepsize of TGP-F-R/E to 0.055 to ensure convergence. For~\ref{qp_inhomo_set_2}, we set \( a = 0.7 \), \(\gamma = 0.5\), \(\eta_k = 0.1\), and use the stepsize $0.05$ for TGP-F-R/E. 
All other parameters are kept at their default values.

The results of \ref{exp1} under both settings of \( \matr{A} \) are shown in \cref{fig:qp_inhomo_exp1}. Each figure shows \emph{Niter} and \emph{NGlobal} of the TGP-A algorithm using \( \matr{H}_k \) as in \eqref{eq:H_TGP_R_E}, with varying values of $a$.
As shown in \cref{fig:qp_inhomo_exp1_rg,fig:qp_inhomo_exp1_eg},  different values of $a$ in \ref{qp_inhomo_set_1} result in significantly different performance of the TGP algorithms, both in terms of \emph{Niter} and \emph{NGlobal}. 
In particular, a small positive value of $a$ improves the performance compared to $a = 0$. 
In contrast, in \ref{qp_inhomo_set_2}, all TGP algorithms reach the global minimum. However, \emph{Niter} is still related to the value of $a$.
Overall, the variation in the value of $a$, which is related to the normal component in \( \matr{H}_k \), significantly affects the practical performance of the TGP algorithms.


The results of \ref{exp2} are presented in \cref{table:qp_inhomo_exp2}.
In \ref{qp_inhomo_set_1}, where \( \matr{A} \) is randomly generated and typically ill-conditioned, TGP-A/NA-E achieve the highest \emph{NGlobal} (414–418 out of 500) with \emph{Niter} and \emph{Time} comparable to RGD. 
By contrast, BFGS requires fewer iterations (\emph{NIter}) but more computation time (\emph{Time}), and attains fewer global solutions than TGP-A/NA-E.
The Zhang–Hager-type nonmonotone Armijo line search in TGP-NA-E/R variants provides a modest improvement in robustness over their monotone counterparts.
In \ref{qp_inhomo_set_2}, where \( \matr{A} \) has nearly uniform eigenvalues, all algorithms successfully reach the global minimum in all 500 instances. 
However, TGP-F-E requires the fewest iterations (\emph{NIter}) and the shortest runtime (\emph{Time}), outperforming all the retraction-based Riemannian algorithms in terms of efficiency. 
These results suggest that the TGP framework, by incorporating the normal component into the search direction, can improve convergence behavior without sacrificing generality.

\begin{figure}[htbp]
  \centering
  \begin{subfigure}{0.42\linewidth}
    \centering
    \includegraphics[width=\textwidth]{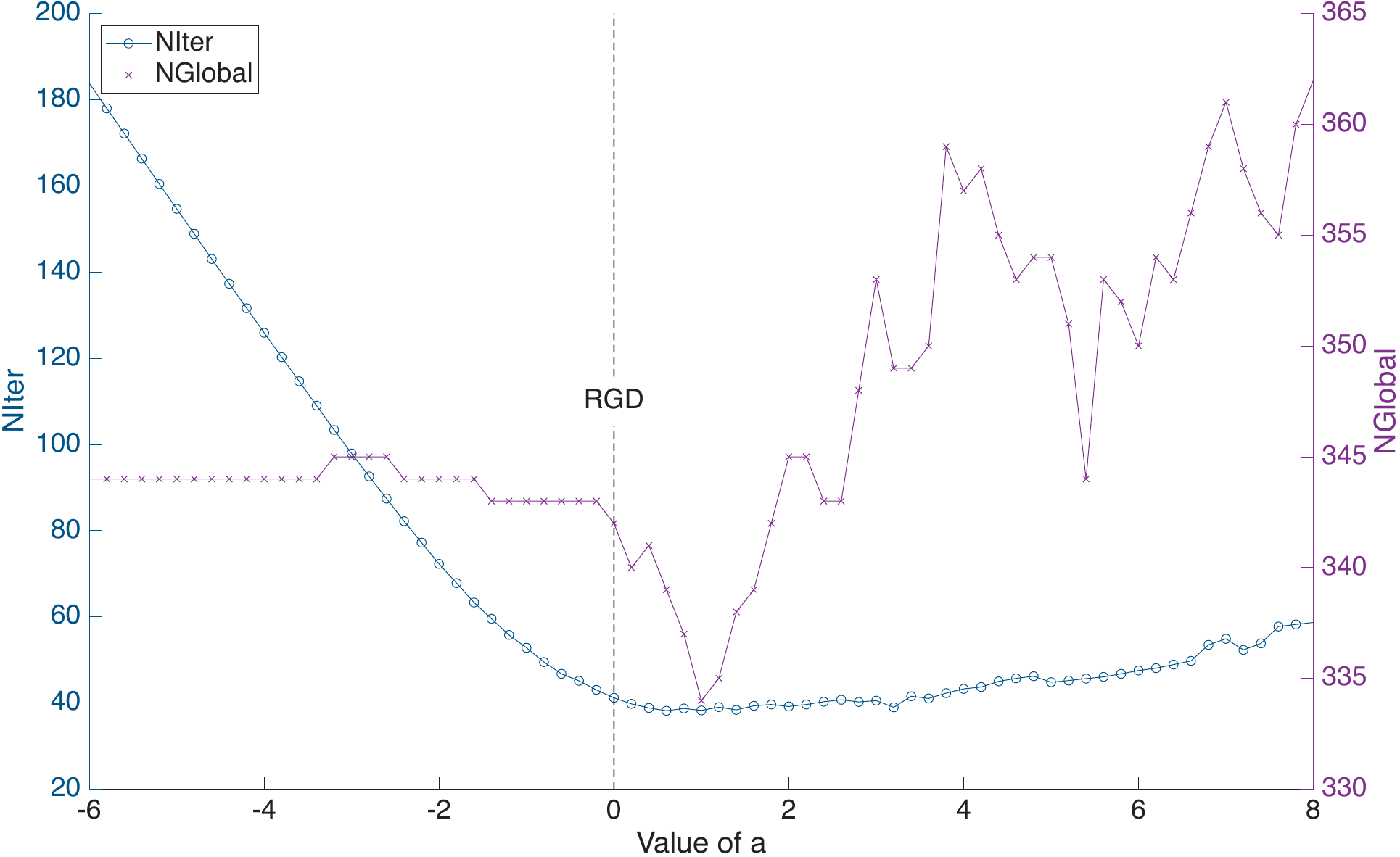}
    \caption{TGP-A-R with varying $a$ in~\ref{qp_inhomo_set_1}}
    \label{fig:qp_inhomo_exp1_rg}
  \end{subfigure}
  \hspace{0.04\textwidth} 
  \begin{subfigure}{0.42\linewidth}
    \centering
    \includegraphics[width=\textwidth]{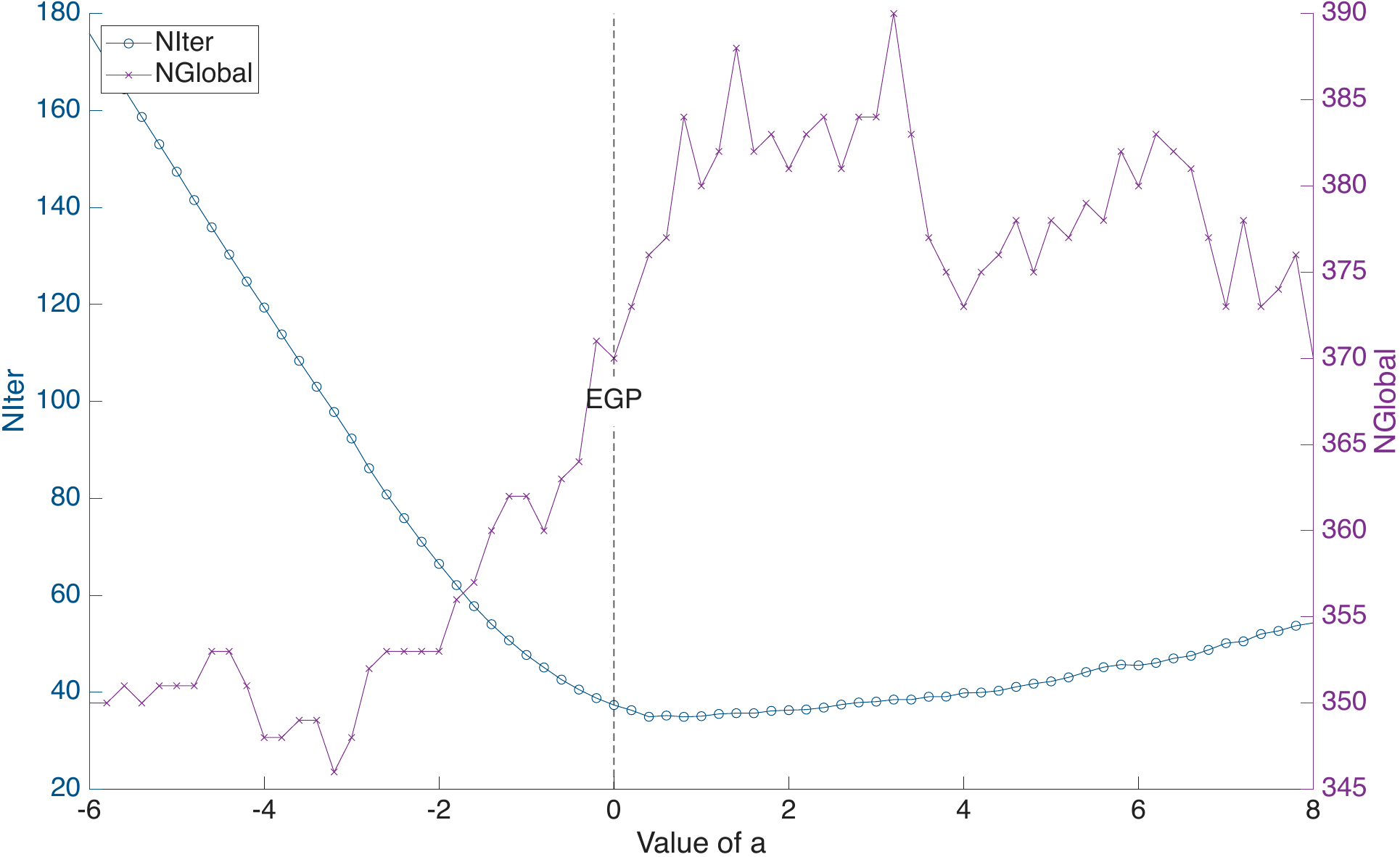}
    \caption{TGP-A-E with varying $a$ in~\ref{qp_inhomo_set_1}}
    \label{fig:qp_inhomo_exp1_eg}
  \end{subfigure}\\
  \begin{subfigure}{0.42\linewidth}
    \centering
    \includegraphics[width=\textwidth]{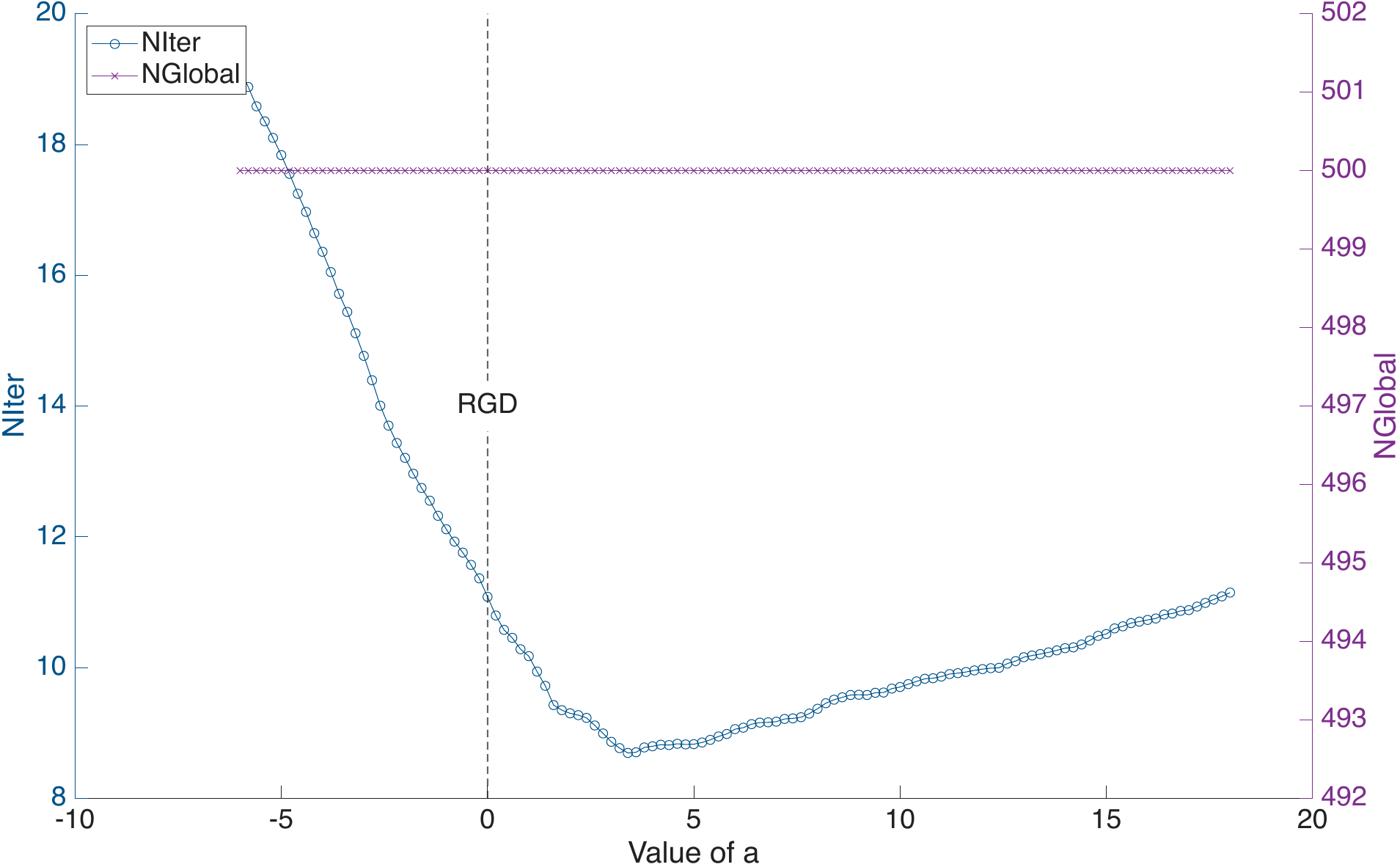}
    \caption{TGP-A-R with varying $a$ in~\ref{qp_inhomo_set_2}}
    \label{fig:qp_inhomo_99_02_3_2_exp1_rg}
  \end{subfigure}
  \hspace{0.04\textwidth} 
  \begin{subfigure}{0.42\linewidth}
    \centering
    \includegraphics[width=\textwidth]{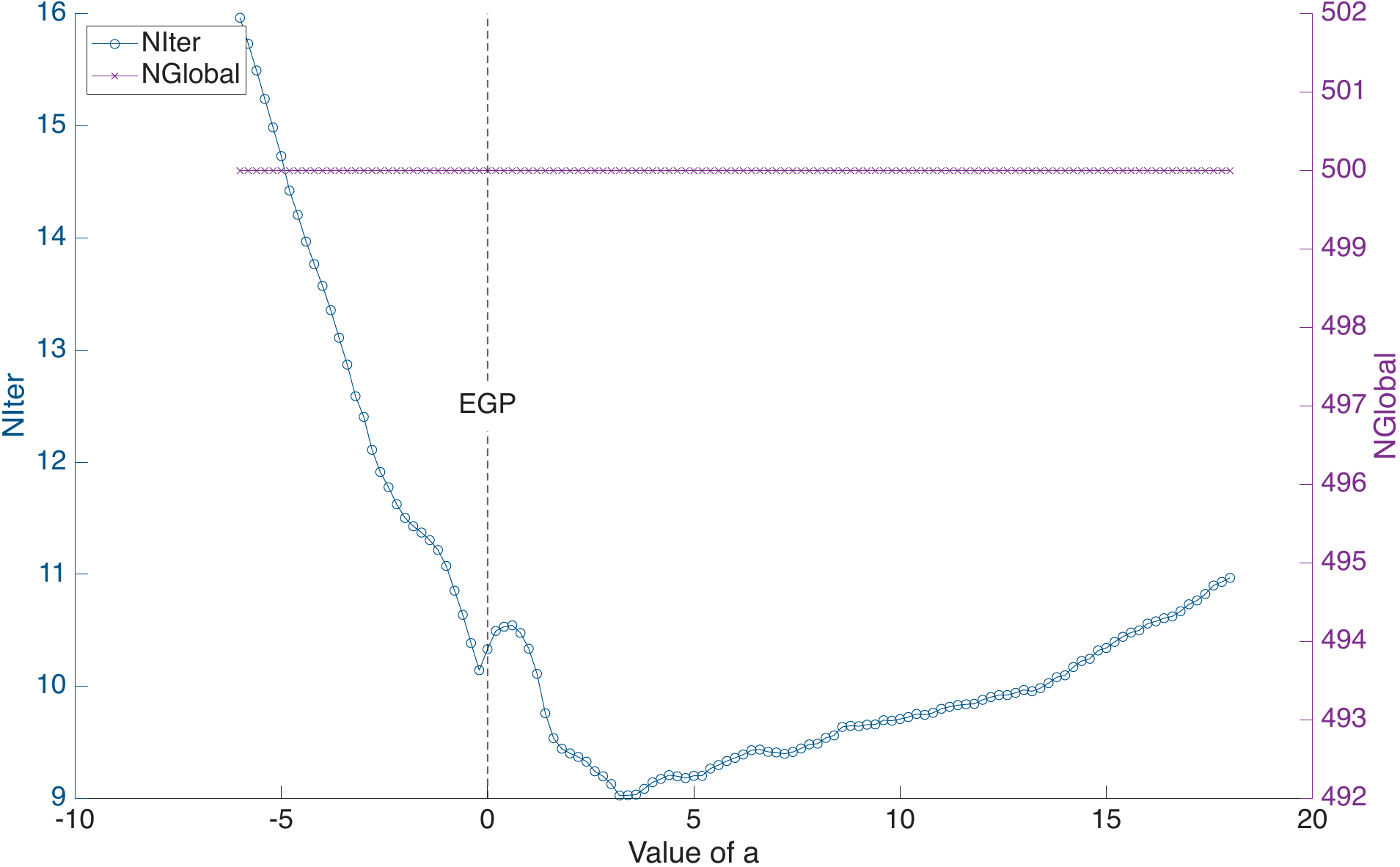}
    \caption{TGP-A-E with varying $a$ in~\ref{qp_inhomo_set_2}}
    \label{fig:qp_inhomo_99_02_3_2_exp1_eg}
  \end{subfigure}
  \caption{Results of \ref{exp1} for \cref{exa:qp_inhomo}.
The blue curves represent the average number of iterations (\emph{Niter}), and the purple curves represent the number of instances that reach the global minimum (\emph{NGlobal}), both as functions of \(a\). 
The black dashed lines correspond to the case \(a = 0\), i.e., RGD or EGP algorithm.}
  \label{fig:qp_inhomo_exp1}
\end{figure}

\begin{table}[htbp]
\centering
\footnotesize
\begin{tabular}{lrrrrrrrr}
\toprule
& \multicolumn{4}{c}{$\matr{A}$ from~\ref{qp_inhomo_set_1}} 
& \multicolumn{4}{c}{$\matr{A}$ from~\ref{qp_inhomo_set_2}} \\
\cmidrule(lr){2-5} \cmidrule(lr){6-9}
Algorithm & NIter & Time (s) & NGlobal & NFail 
          & NIter & Time (s) & NGlobal & NFail \\ 
\midrule
  TGP-A-R  & 43.0     & 0.0135     & 354     & 0  & 10.4    & 0.0038     & 500 & 0 \\
  TGP-NA-R & 44.9     & 0.0138     & 355     & 0  & 14.3    & 0.0049     & 500 & 0 \\
  TGP-F-R  & 206.9    & 0.0671     & 342     & 0  & 7.9     & 0.0020     & 500 & 0 \\
  TGP-A-E  & 38.1     & 0.0115     & 414     & 0  & 10.5    & 0.0038     & 500 & 0 \\
  TGP-NA-E & 39.8     & 0.0119     & \bf 418 & 0  & 12.4    & 0.0042     & 500 & 0 \\
  TGP-F-E  & 203.9    & 0.0651     & 345     & 0  & \bf 5.2 & \bf 0.0012 & 500 & 0 \\
\addlinespace[2pt]
\cmidrule(lr){1-9}
\addlinespace[2pt]
  RGD      & 39.8     &\bf 0.0114     & 346     & 0  & 14.6    & 0.0046     & 500 & 0 \\
  CG       & 46.8     & 0.0318     & 340     & 2  & 13.5    & 0.0073     & 500 & 0 \\
  BFGS     & \bf 13.6 & 0.0383     & 341     & 0  & 8.9     & 0.0139     & 500 & 0 \\
\bottomrule
\end{tabular}
\caption{Results of \ref{exp2} for \cref{exa:qp_inhomo}.}
\label{table:qp_inhomo_exp2}
\end{table}


\textbf{An intuitive explanation.}
We illustrate why TGP algorithms can find more global solutions in \ref{qp_inhomo_set_1} through a simple example. Consider \( \matr{A} = \diag{5, 2} \), \( \vect{x}^{*} = (0, 1)^\top \) and the initial point \(\vect{x}_0 =  (-\frac{1}{\sqrt{2}}, -\frac{1}{\sqrt{2}})^\top\). 
We apply RGD and TGP-A-R with $a = 2$ and plot their trajectories in \cref{fig:exa_QPInhomo}. 
We observe that RGD converges to a local minimizer \( (0, -1)^\top \), while TGP-A-R successfully reaches the global minimizer \( (0, 1)^\top \).
This difference arises from their different search directions. 
Since the search direction of RGD is a tangent vector, RGD moves on the geodesics of the manifold. 
Therefore, its trajectory is inevitably attracted to the nearby local minimizer. 
In contrast, TGP incorporates a normal vector in its search direction, allowing it to ``jump'' across the manifold and escape from a local minimum. For instance, in the first iteration, the backtracking procedure of TGP yields stepsize \( \tau_0 = 1 \), leading to \( \vect{x}_1 = \mathcal{P}_{\mathbb{S}^1}(\vect{x}_0 - (\grad f(\vect{x}_0) + 2 \vect{x}_0) = \mathcal{P}_{\mathbb{S}^1}(- \vect{x}_0 - \grad f(\vect{x}_0)) \), which is significantly different from \( \vect{x}_1 = \retr_{\vect{x}_0}(-\tau_0 \grad f(\vect{x}_0)) \) generated by RGD.
This example illustrates that the additional normal vector \( 2 \vect{x}_0 \) in \( \matr{H}_0 \) enables the TGP algorithm to explore a broader region and reach a better solution in practice.
\begin{figure}[ht]
  \centering
  \includegraphics[width=0.5\linewidth]{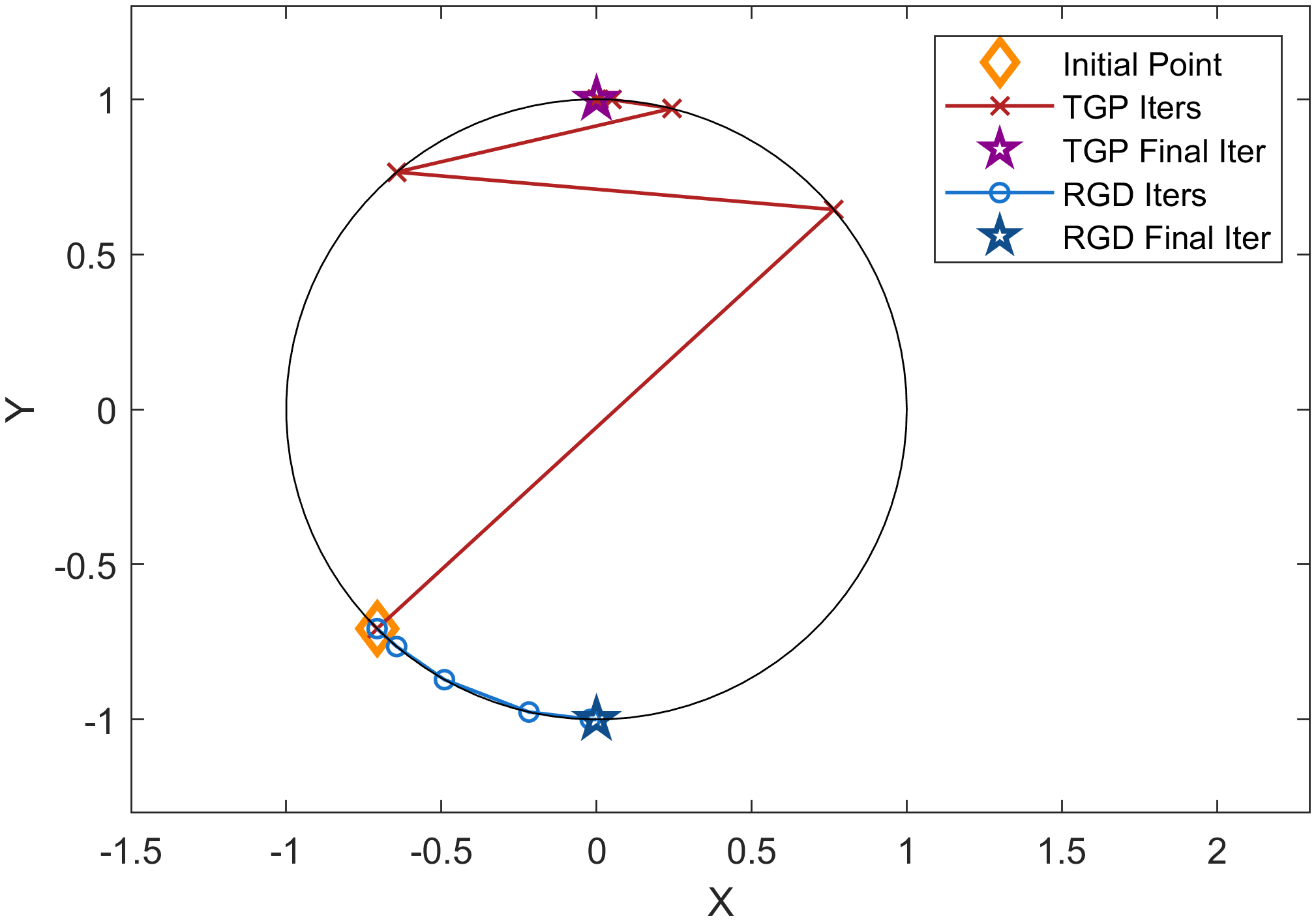}
  \caption{Paths of the iterates generated by RGD and TGP-A-R}\label{fig:exa_QPInhomo}
\end{figure}

\begin{example}\label{exa:jamd2}
We next consider the \emph{jointly approximate symmetric matrix diagonalization} (JAMD-S) problem \cite{LUC2017globally,ULC2019,li2019polar,ding2024projectively} on $\St(r,n)$:
\begin{equation}\label{eq:problem-jamd}
\min\ f(\matr{X}) = -\sum\limits_{\ell=1}^{L}  \left\| \diag{\matr{X}^{\T}\matr{A}^{(\ell)}\matr{X}} \right\|^2,\ \ \matr{X} \in \St(r, n).
\end{equation}
Here \( \matr{A}^{(\ell)} \in \symmm{\RR^{n \times n}} \) for all \( \ell \), and \( \diag{\matr{W}} \eqdef (W_{11}, W_{22} , \ldots , W_{rr})^{\T} \) represents the vector composed of the diagonal elements of \( \matr{W}\in\RR^{r\times r} \).
We set \( n = 3 \), \( r = 2 \), and \( L = 3 \).
Each matrix \( \matr{A}^{(\ell)} \) is generated as \( \matr{A}^{(\ell)} = \matr{Q}^{\T} \matr{D}^{(\ell)}\matr{Q} + \matr{E}^{(\ell)} \), where \( \matr{Q} = \mathcal{P}_{\ON{n}}(\texttt{randn}(n,n)) \), \( \matr{D}^{(\ell)} = \texttt{diag}({\texttt{randn}(n, 1)}) \), and \( \matr{E}^{(\ell)} = 0.02 \cdot \symmo{\texttt{randn}(n, n)} \) represents random noise. 
Due to the existence of the random noise matrices \( \matr{E}^{(\ell)} \), the global minimum is unknown. 
Therefore, we use \emph{NSuper} to evaluate the solution quality instead of \emph{NGlobal}.
In \ref{exp2}, we set $a = 2$ for TGP-$\ast$-R, $a = 10.8$ for TGP-$\ast$-E, \(\eta_k = 0.2\) for TGP-NA-R/E, a fixed stepsize of 0.025 for TGP-F-R and 0.059 for TGP-F-E.

In \ref{exp1}, we study the influence of the normal component \( \hat{\matr{H}}_k \), controlled by the parameter \( a \) in \eqref{eq:H_TGP_R_E}. 
Specifically, in \cref{fig:jamd2_exp1_rg}, the baseline is TGP-A-R with \( a = 0 \), i.e., RGD, and in \cref{fig:jamd2_exp1_eg}, the baseline is TGP-A-E with \( a = 0 \), i.e., EGP. 
As shown, \emph{Niter} and \emph{NSuper} vary significantly with the choice of \( a \), which is consistent with the results in \cref{fig:qp_inhomo_exp1}.

In \ref{exp2}, we use RGD as the baseline to compute \emph{NSuper}.
The average results over 500 random instances are reported in \cref{table:jamd_exp2}. 
Among all methods, TGP-NA-E achieves the best overall solution quality with \emph{NSuper} = 79, and TGP-A-R is the fastest.
\end{example}

\begin{figure}[!htp]
\centering
\begin{subfigure}{0.42\linewidth}
\centering
  \includegraphics[width=\textwidth]{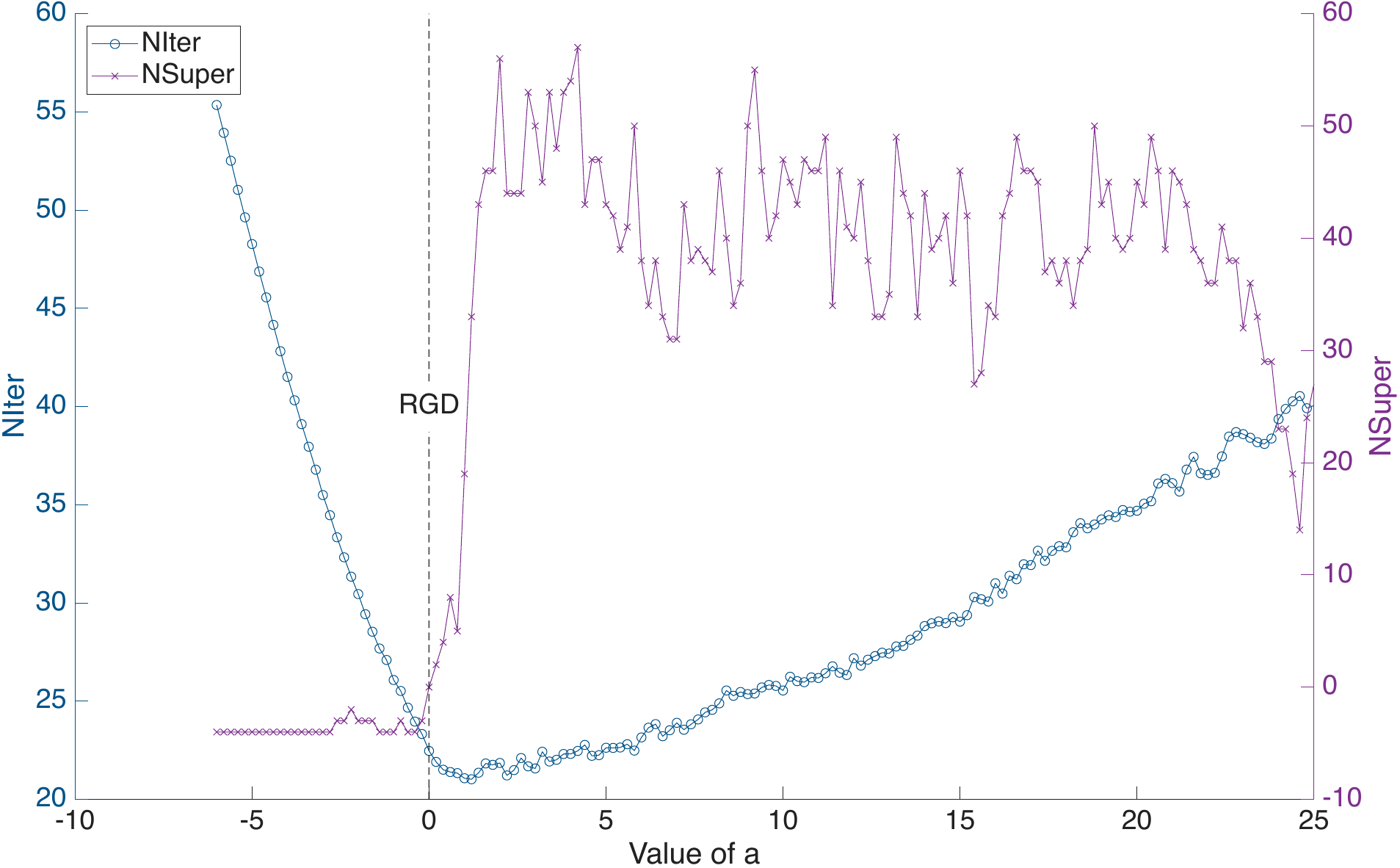}%
  \caption{TGP-A-R with varying $a$. The case $a=0$ serves as the baseline.}
  \label{fig:jamd2_exp1_rg}
\end{subfigure}
  \hspace{0.04\textwidth} 
\begin{subfigure}{0.42\linewidth}
\centering
  \includegraphics[width=\textwidth]{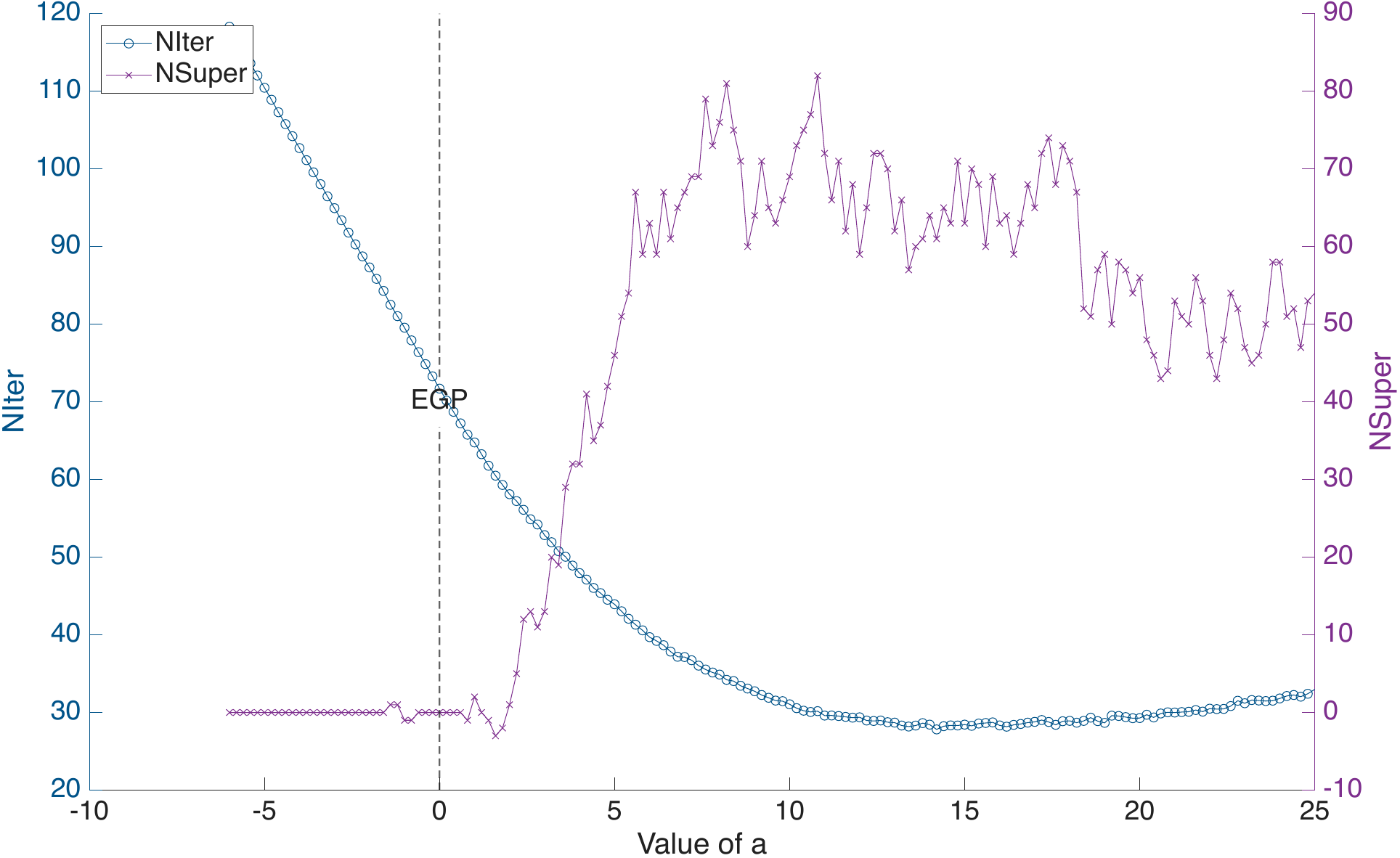}%
  \caption{TGP-A-E with varying $a$. The case $a=0$ serves as the baseline.}
  \label{fig:jamd2_exp1_eg}
\end{subfigure}
\caption{Results of \ref{exp1} for \cref{exa:jamd2}.
The blue curves represent the average number of iterations (\emph{Niter}), and the purple curves represent the solution quality (\emph{NSuper}), both as functions of \(a\). 
The black dashed lines correspond to the case \(a = 0\), i.e., RGD or EGP algorithm.}
\label{fig:jamd_exp1}
\end{figure}

\begin{table}[htb]
\centering
\begin{tabular}{lrrrrrr}
  \toprule
  Algorithm & NIter      & Time (s) &NBetter &NWorse         & NSuper   & NFail \\
  \midrule
  TGP-A-R  & 21.9  & \textbf{0.0076} & 69           & 17          & 52          & 0 \\
  TGP-NA-R & 25.9  & 0.0088 & 72           & 17          & 55          & 0 \\
  TGP-F-R  & 123.0 & 0.0375 & 5            & 15 & -10         & 0 \\
  TGP-A-E  & 30.2  & 0.0101 & 96           & 24          & 72          & 0 \\
  TGP-NA-E & 27.4  & 0.0088 & \textbf{101} & 22          & \textbf{79} & 0 \\
  TGP-F-E  & 45.8  & 0.0128 & 8            & 17          & -9          & 0 \\
\addlinespace[2pt]
\cmidrule(lr){1-7}
\addlinespace[2pt]
  RGD  & 26.4          & 0.0079          & ---  & ---  & ---   & 0 \\
  CG   & 86.7          & 0.0597          & 4  & 16 & -12 & 5 \\
  BFGS & \textbf{11.4} & 0.0259          & 13 & \textbf{14} & -1  & 0 \\
  \bottomrule
\end{tabular}
\caption{Results of~\ref{exp2} for \cref{exa:jamd2}. RGD serves as the baseline.}
\label{table:jamd_exp2}
\end{table}


\begin{example}\label{exa:jamd3}
We further consider the \emph{jointly approximate symmetric tensor diagonalization} (JATD-S) problem \cite{LUC2017globally,ULC2019,li2019polar} on \( \St(r, n) \):
\begin{equation}\label{problem:jamd3}
\min\ f(\matr{X}) = -\sum\limits_{\ell=1}^{L} \left\| \diag{\tenss{A}^{(\ell)}\contr{1}\matr{X}^{\T}\contr{2}\matr{X}^{\T}\contr{3}\matr{X}^{\T}} \right\|^2,\ \ \matr{X} \in \St(r, n),
\end{equation}
where \( \tenss{A}^{(\ell)} \in \RR^{n \times n \times n} \) is a third-order symmetric tensor, and \( \diag{\tenss{W}} \eqdef (W_{111}, W_{222} , \ldots , W_{rrr})^{\T} \) denotes the vector composed of the diagonal elements of a tensor \( \tenss{W}\in \RR^{r \times r \times r} \).
We fix \( n = 3, r = 2\) and \(L = 1\).
In each random instance, \( \tenss{A}^{(\ell)} \) is generated as \( \tenss{A}^{(\ell)} = \texttt{symmetrize}(\texttt{tensor}(\texttt{randn}(n, n, n))) \). 
Therefore, the global minimum is also unknown. 
In \ref{exp2}, we set $a =9.6$ for TGP-$\ast$-R and $a = 12.4$ for TGP-$\ast$-E. For TGP-NA-R/E, we fix \(\eta_k = 0.1\).
The stepsizes of TGP-F-R and TGP-F-E are set to 0.01 and 0.035, respectively.
The maximum time \texttt{maxtime} is set to 10 seconds.

The results of \ref{exp1} are reported in \cref{fig:jamd3_exp1}. 
We observe that both \emph{NSuper} and \emph{Niter} vary as $a$ changes. In particular, compared to the case $a=0$, moderately large values of $a$ tend to improve the performance of TGP-A algorithms.
The results of \ref{exp2} are summarized in \cref{table:jamd3_exp2}.
Among all variants, TGP-NA-E consistently finds more superior solutions than the other algorithms, and also achieves the fastest convergence in terms of \emph{Time}.
In particular, TGP-A/NA-E significantly outperform RGD with respect to \emph{Niter}, \emph{Time}, and \emph{NSuper}.
\end{example}

\begin{figure}[!htp]
\centering
\begin{subfigure}{0.42\linewidth}
\centering
  \includegraphics[width=\textwidth]{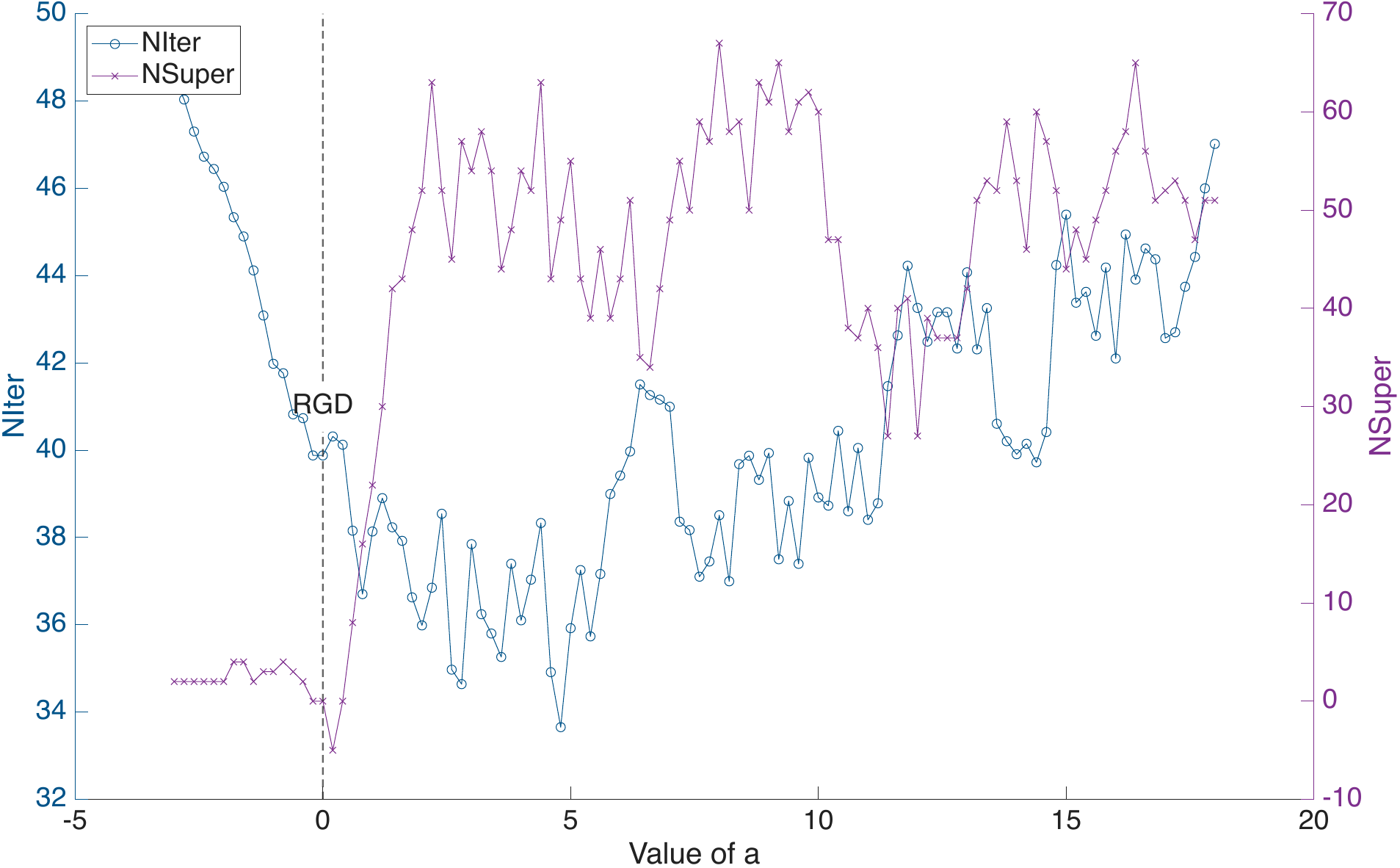}%
  \caption{TGP-A-R with varying $a$. The case $a = 0$ serves as the baseline.}
  \label{fig:jamd3_exp1_rg}
\end{subfigure}
 \hspace{0.04\textwidth} 
\begin{subfigure}{0.42\linewidth}
\centering
  \includegraphics[width=\textwidth]{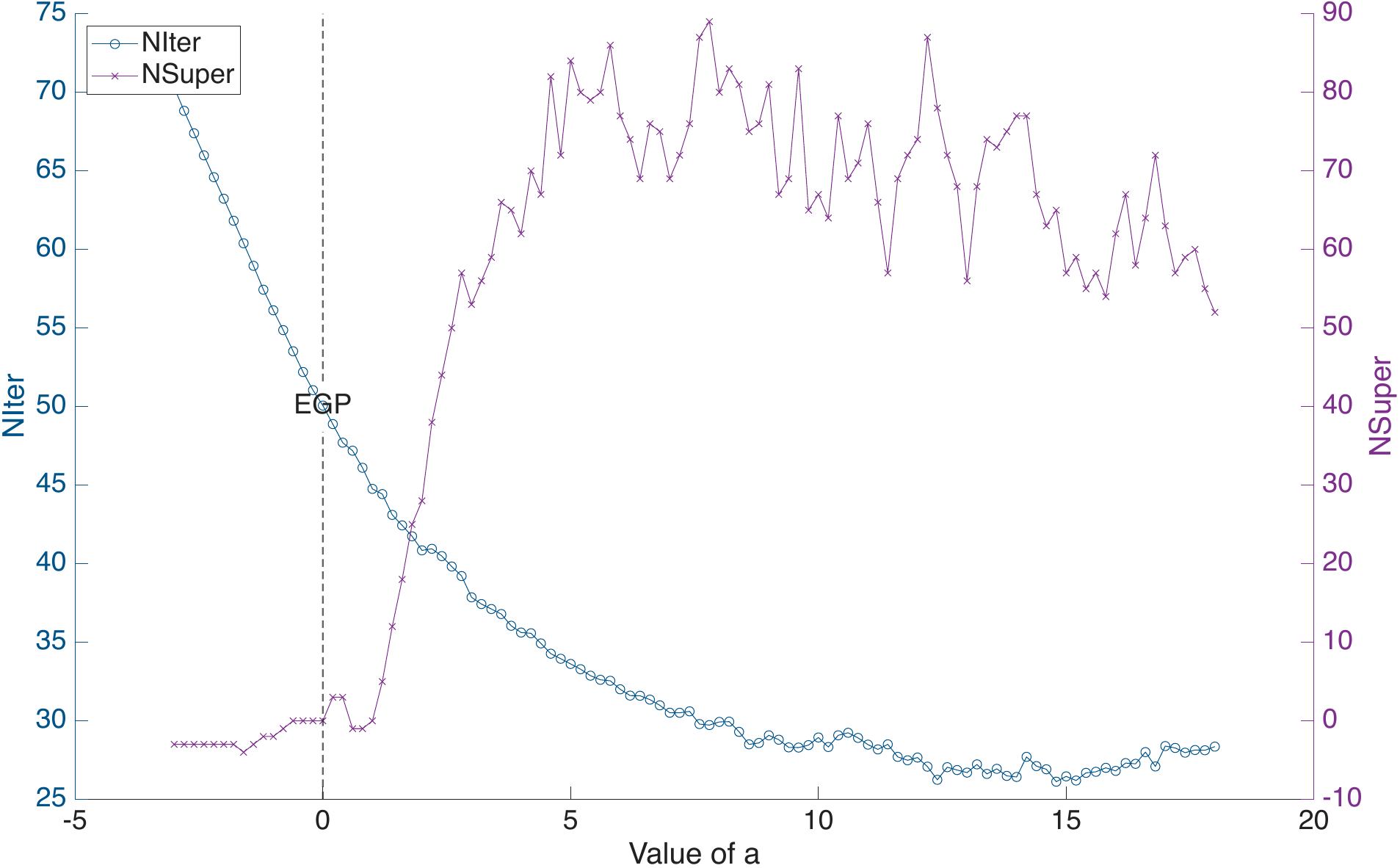}%
  \caption{TGP-A-E with varying $a$. The case $a=0$ serves as the baseline.}
  \label{fig:jamd3_exp1_eg}
\end{subfigure}
\caption{Results of \ref{exp1} for \cref{exa:jamd3}.
The blue curves represent the average number of iterations (\emph{Niter}), and the purple curves represent the solution quality (\emph{NSuper}), both as functions of \(a\). 
The black dashed lines correspond to the case \(a = 0\), i.e., RGD or EGP algorithm.}
\label{fig:jamd3_exp1}
\end{figure}

\begin{table}[htb]
\centering
\begin{tabular}{lrrrrrr}
  \toprule
  Algorithm &NIter&Time (s)& NBetter & NWorse & NSuper & NFail \\ \midrule
  TGP-A-R  & 37.4  & 0.0221 & 85           & 30          & 55          & 0 \\
  TGP-NA-R & 39.8  & 0.0233 & 88           & 31          & 57          & 0 \\
  TGP-F-R  & 284.3 & 0.1466 & 17           & \textbf{20} & -3          & 0 \\
  TGP-A-E  & 26.2  & 0.0159 & 113          & 33          & 80          & 0 \\
  TGP-NA-E & 22.6  & \textbf{0.0131} & \textbf{114} & 32          & \textbf{82} & 0 \\
  TGP-F-E  & 65.0  & 0.0289 & 17           & 23          & -6          & 0 \\
\addlinespace[2pt]
\cmidrule(lr){1-7}
\addlinespace[2pt]
  RGD  & 43.1          & 0.0190          & --- & --- & ---  & 0 \\
  CG   & 139.8         & 0.1548          & 11 & 31 & -20 & 7 \\
  BFGS & \textbf{13.1} & 0.0384          & 23 & 22 & 1   & 0 \\
  \bottomrule
\end{tabular}
\caption{Results of \ref{exp2} for \cref{exa:jamd3}. RGD serves as the baseline.}
\label{table:jamd3_exp2}
\end{table}

\textbf{Effect of normal components.} 
From the numerical experiments in \cref{exa:qp_inhomo,exa:jamd2,exa:jamd3}, we observe that the performance of TGP algorithms is strongly affected by the normal component of $\matr{H}_k$, such as the parameter $a$ in \eqref{eq:H_TGP_R_E}.
With a suitable normal component, TGP algorithms outperform the RGD, especially in terms of solution quality.
In particular, for \cref{exa:jamd3}, where the optimization landscape is more intricate than in \cref{exa:qp_inhomo,exa:jamd2}, the superiority of TGP algorithms in the quality of the final solutions becomes more obvious.
Furthermore, as shown in \cref{table:qp_inhomo_exp2,table:jamd_exp2,table:jamd3_exp2}, within the TGP family, the advantage of Zhang–Hager-type nonmonotone Armijo stepsizes over monotone ones becomes increasingly significant as the problem landscape grows more complex.

\subsection{Effect of tangent components with a fixed normal component}\label{subsec:tgp_diffe_scale}

In this subsection, we investigate how the tangent component of the search direction \( \matr{H}_k \) influences the practical performance of the TGP algorithms. 
Specifically, we consider two modified search directions, both based on TGP-$\ast$-E in~\eqref{eq:H_TGP_R_E} but with different tangent components:
\begin{equation}\label{eq:TGP-DEF}
  \begin{split}
  \text{TGP-$\ast$-DE}:\quad & 
  \matr{H}_k = \matr{D}_{\rho}(\matr{X}_k) + 
  \mathcal{P}_{\NormalSt{\matr{X}_k}}(\nabla f(\matr{X}_k)) + 
  a \matr{X}_{k}\matr{S}_k,\\
  \text{TGP-$\ast$-DF}:\quad &
  \matr{H}_k = \matr{L}(\matr{X}_k)\grad f(\matr{X}_k)\matr{R}(\matr{X}_k) + 
  \mathcal{P}_{\NormalSt{\matr{X}_k}}(\nabla f(\matr{X}_k)) + 
  a \matr{X}_{k}\matr{S}_k,
  \end{split}
\end{equation}
where \( a \in \RR \) and \( \matr{S}_k \in \symmm{\RR^{r \times r}} \) are defined as in~\cref{subsec:tgp_diffe_norm}.  
For simplicity, all algorithms in this subsection use the Armijo stepsize (\emph{A}).
In TGP-A-DE, we replace the tangent component \( \grad f(\matr{X}_k) \) of TGP-A-E by a more general form \( \matr{D}_\rho(\matr{X}_k) \), where different values of \( \rho \) yield different tangent directions. 
Note that when \( \rho = 0.25 \), \( \matr{D}_{\rho}(\matr{X}_k) \) reduces to \( \grad f(\matr{X}_k) \), and thus TGP-A-DE degenerates to TGP-A-E. 
We will set an appropriate value of \(\rho\) to improve the performance of TGP-A-E. 
We then investigate whether the performance of TGP-A-DE can be further enhanced by incorporating a more flexible tangent component.
To this end, we consider the generalized form proposed in~\cref{exa:LR_St}, leading to the variant TGP-A-DF, where the tangent component is scaled by two matrices \( \matr{L}(\matr{X}_k) \) and \( \matr{R}(\matr{X}_k) \).  
As discussed in~\cref{rem:LR_St}, \( \matr{D}_{\rho}(\matr{X}_k) \) can be viewed as a special case of \( \matr{L}(\matr{X}_k)\grad f(\matr{X}_k)\matr{R}(\matr{X}_k) \) by setting the scaling matrices as in~\eqref{eq:LR_St} with \( \matr{E} = \matr{I}_r \), \( \mu = 4\rho - 1 \), and \( \matr{F} = \matr{0} \).

We first examine the effect of \(\rho\) in TGP-A-DE. 
We use the same experimental settings as in~\ref{exp1}, but choose $a$ as in~\ref{exp2} for~\cref{exa:jamd2,exa:jamd3}, and vary \(\rho\).
The results in~\cref{fig:TGP-DE} show that varying \( \rho \) affects both the number of iterations (\emph{Niter}) and the number of superior solutions (\emph{NSuper}), even when the same normal component is used. 
This suggests that the tangent component plays a crucial role in determining the efficiency of the TGP algorithms.
  
For TGP-A-DF, we consider a simple case: we retain \( \matr{E} = \matr{I}_r \) and \( \mu = 4\rho - 1 \), where \(\rho\) is chosen as in TGP-A-DE, and modify only \( \matr{F} \) to adjust the tangent scaling.  
For~\cref{exa:jamd2}, we set \( \rho = 0.35 \) and \( F = 0.05 \); 
for~\cref{exa:jamd3}, we set \( \rho = 0.23 \) and \( F = 0.25 \), both with \( a \) chosen as in the previous experiment..  

The results in~\cref{table:TGP-tangent} indicate that incorporating an additional scaling term \( \matr{F} \) can further improve the overall efficiency and stability of the algorithm. 
Although the performance gains depend on the specific choice of \( \matr{F} \) and \( \rho \), these results demonstrate that an appropriately designed tangent component can enhance the behavior of TGP beyond the standard choice \( \grad f(\matr{X}) \).  

\begin{figure}[!htp]
\centering
\begin{subfigure}{0.42\linewidth}
\centering
  \includegraphics[width=\textwidth]{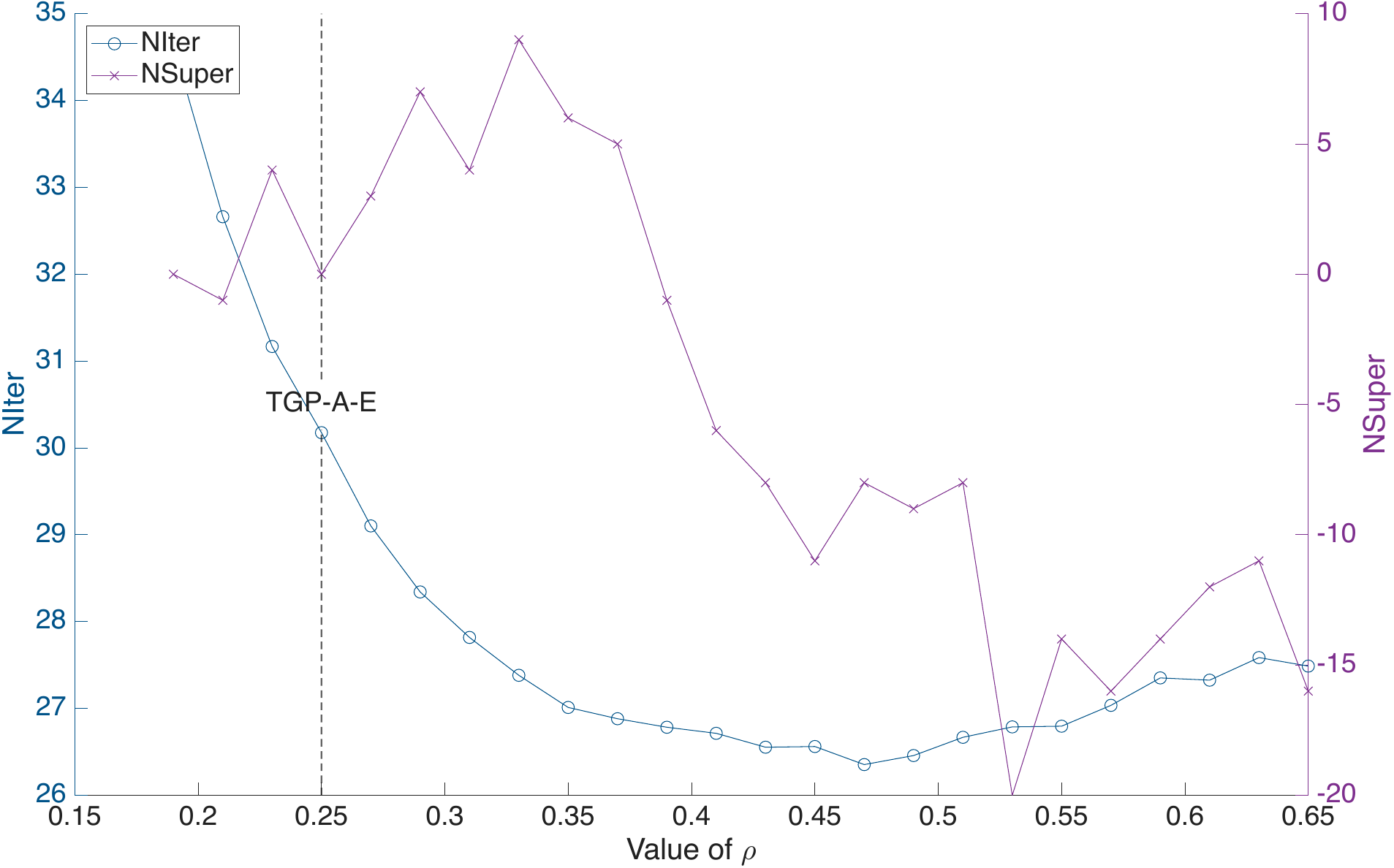}%
  \caption{\cref{exa:jamd2}}
  \label{fig:jamdda_eg}
\end{subfigure}
\hspace{0.04\textwidth} 
\begin{subfigure}{0.42\linewidth}
\centering
  \includegraphics[width=\textwidth]{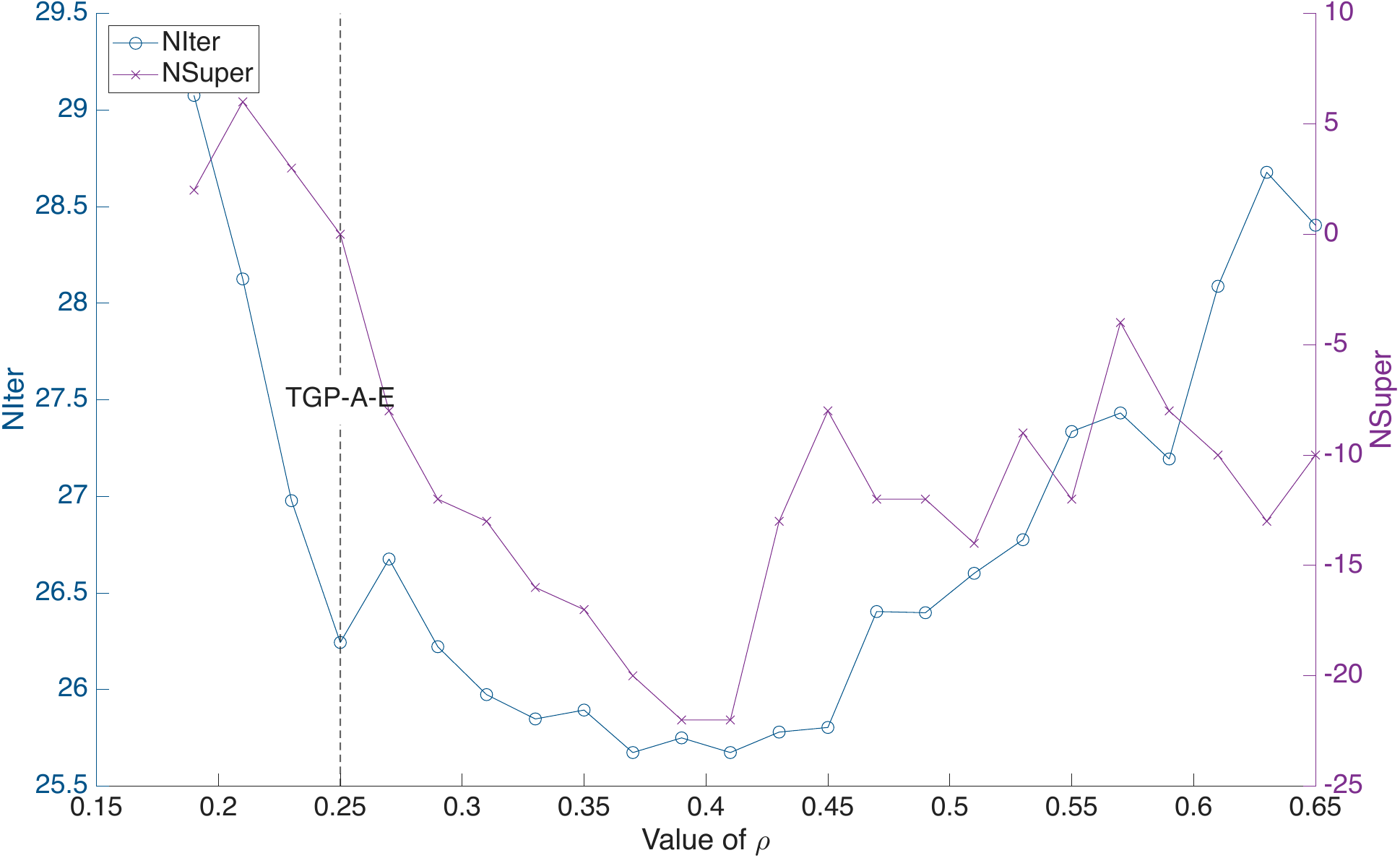}%
  \caption{\cref{exa:jamd3}.}
  \label{fig:jatdda_eg}
\end{subfigure}
\caption{TGP-A-DE with varying \(\rho\).
The blue curves represent the average number of iterations (\emph{Niter}), and the purple curves represent the solution quality (\emph{NSuper}), both as functions of \(a\). 
The black dashed lines correspond to the case \(\rho = 0.25\), i.e., TGP-A-E, which serves as the baseline.}
\label{fig:TGP-DE}
\end{figure}

\begin{table}[htbp]
\centering
\begin{tabular}{lrrrrrrrrrr}
\toprule
& \multicolumn{5}{c}{Problem~\eqref{eq:problem-jamd}} 
& \multicolumn{5}{c}{Problem~\eqref{problem:jamd3}} \\
\cmidrule(lr){2-6} \cmidrule(lr){7-11}
Algorithm & NIter & Time (s) & NBetter & NWorse & NSuper  
          & NIter & Time (s) & NBetter & NWorse & NSuper \\ 
\midrule
  TGP-A-E  & 30.2     & 0.0106     & --- & --- & ---     & \bf 26.2 & \bf 0.0155 & --- & --- & ---      \\
\addlinespace[2pt]
\cmidrule(lr){1-11}
\addlinespace[2pt]
  TGP-A-DE & 27.0     & \bf 0.0095 & 20 & \textbf{14} & 6     & 27.0     & 0.0159   & 13 & 10   & 3      \\
  TGP-A-DF & \bf 26.6 & 0.0096 & \textbf{23} & 15 & \textbf{8} & 27.4     & 0.0163  & \textbf{23} & \textbf{8}    & \textbf{15} \\
\bottomrule
\end{tabular}
\caption{Comparison of TGP with different tangent components for Problems~\eqref{eq:problem-jamd} and~\eqref{problem:jamd3}; TGP-A-E serves as the baseline.}
\label{table:TGP-tangent}
\end{table}




\subsection{A specific example: analysis of the eigenvalue problem}\label{subsec:eigenvalue-problem}

In \cref{subsec:tgp_diffe_norm,subsec:tgp_diffe_scale}, we presented several numerical results. In this subsection, we focus on the following eigenvalue problem \eqref{eq:eigenvalue_problem} as a specific example, due to its simplicity, making it easier to study the algorithm's behavior. We provide a theoretical analysis of this case, and complement it with numerical results from a specific instance to illustrate the impact of the scaling matrices and the normal component on the algorithm's performance. 

\begin{example}
The \emph{eigenvalue problem} can be formulated as 
\begin{equation}\label{eq:eigenvalue_problem}
  \begin{array}{ll}
    \min_{\vect{x}} & \frac{1}{2}\vect{x}^{\top}\matr{A}\vect{x},  \\
    \mbox{s.t.} & \vect{x} \in \mathbb{S}^{n-1}.
  \end{array}
\end{equation}
\end{example}
Without loss of generality, we assume that $\matr{A}=\diag{\lambda_1 , \lambda_2 \ldots , \lambda_n}$ with \(\lambda_1 \geq \lambda_2 \geq \ldots \geq \lambda_{n-1} > \lambda_n \). It follows that the global minimizers are \( (0 , \ldots , 0, \pm 1)^\top \).
\subsubsection{New scaling matrices to achieve the global solution}

We now demonstrate that, under specific conditions, the scaling matrices $\matr{L}_k$ and $\matr{R}_k$ introduced in \cref{exa:LR_St} can address certain issues encountered in classical RGD and EGP algorithms when solving problem \eqref{eq:eigenvalue_problem}.
Suppose we begin with an initial point \(\vect{x}_0\) satisfying that the last entry is equal to 0, \emph{i.e.}, \((\vect{x}_0)_n = 0\). Direct computation shows that \((\nabla f(\vect{x}_0))_n = (\matr{D}_{\rho}(\vect{x}_0))_n = 0\). 
Therefore, in both RGD using projection as retraction, \(\vect{x}_{k+1} = \mathcal{P}_{\mathbb{S}^{n-1}}(\vect{x}_k - \tau_k \grad f(\vect{x}_k))\), and EGP, \( \vect{x}_{k+1} = \mathcal{P}_{\mathbb{S}^{n-1}}(\vect{x}_k - \tau_k \nabla f(\vect{x}_k)) \), it is straightforward to verify that \((\vect{x}_k)_n = 0\) holds for all \( k\geq 0 \). 
As a result, both methods fail to reach the global minimum, as the iterates remain confined to the subspace spanned by the eigenvectors corresponding to the remaining eigenvalues.
In fact, for $\vect{x}_k$ satisfying \( (\vect{x}_k)_n = 0  \), it can be verified that if \( \matr{F} = \matr{0} \) in \eqref{eq:LR_St}, then the $n$th entry of the search direction \( \matr{H}_k \) will always remain zero. This leads to a situation where classical algorithms are unable to find the global solution, as no progress is made in that direction. 
However, this issue can be resolved by selecting a suitable nonzero matrix $\matr{F}$, which ensures that  $(\matr{H}_k)_n$ becomes nonzero, even when $(\vect{x}_k)_n = 0$. 
For instance, consider the case where \( \matr{A} = \diag{4, 2, -2} \) and the initial point is \( \vect{x}_0 = (\sqrt{3}/ 2, 1 / 2, 0)^\top \). To address the above issue of RGD and EGP, we introduce a new search direction using a nonzero matrix \( \matr{F} \) from \cref{exa:LR_St} for this special eigenvalue problem.
Utilizing this new direction and the Armijo stepsize rule, we construct a new algorithm called the \emph{TGP-A-Eigen} algorithm, defined as follows: 
\begin{equation}\label{eq:H-TGP-eigen}
 \text{TGP-A-Eigen:}\quad \matr{H}_k = \matr{L}_k \grad f(\vect{x}_k) \matr{R}_k, \text{ where } \matr{L}_k = \matr{I} + 0.05 \matr{X}_{\perp} \begin{pmatrix}1 & 1\\ 1 & 1\end{pmatrix} \matr{X}_{\perp}^\top, \matr{R}_k = 1. 
\end{equation}
Here \( \matr{L}_k \) and \( \matr{R}_k \) are constructed from \cref{exa:LR_St} with \( \matr{E} = 1 \), \(\matr{F} = 0.05 \cdot \begin{pmatrix}1 & 1\\ 1 & 1\end{pmatrix} \) and \(\mu = 0\). Note that the direction \eqref{eq:H-TGP-eigen} is in the form of \eqref{eq:Trans_RGP} with \( \matr{N}(\matr{X}) = \matr{0} \). 
Paths of RGD with the Armijo stepsize and TGP-A-Eigen with the initial trial stepsize \(\hat{\tau}_k = 0.5\) are shown in \cref{fig:eigen_TGP_Eigen}.
It is clear from \cref{fig:eigen_TGP_Eigen} that TGP-A-Eigen successfully converges to the global minimizer \( (0,0,1)^\top \), while RGD converges to a saddle point \( (0, 1, 0)^\top \).
\begin{figure}[ht]
  \centering
  \includegraphics[width=0.5\linewidth]{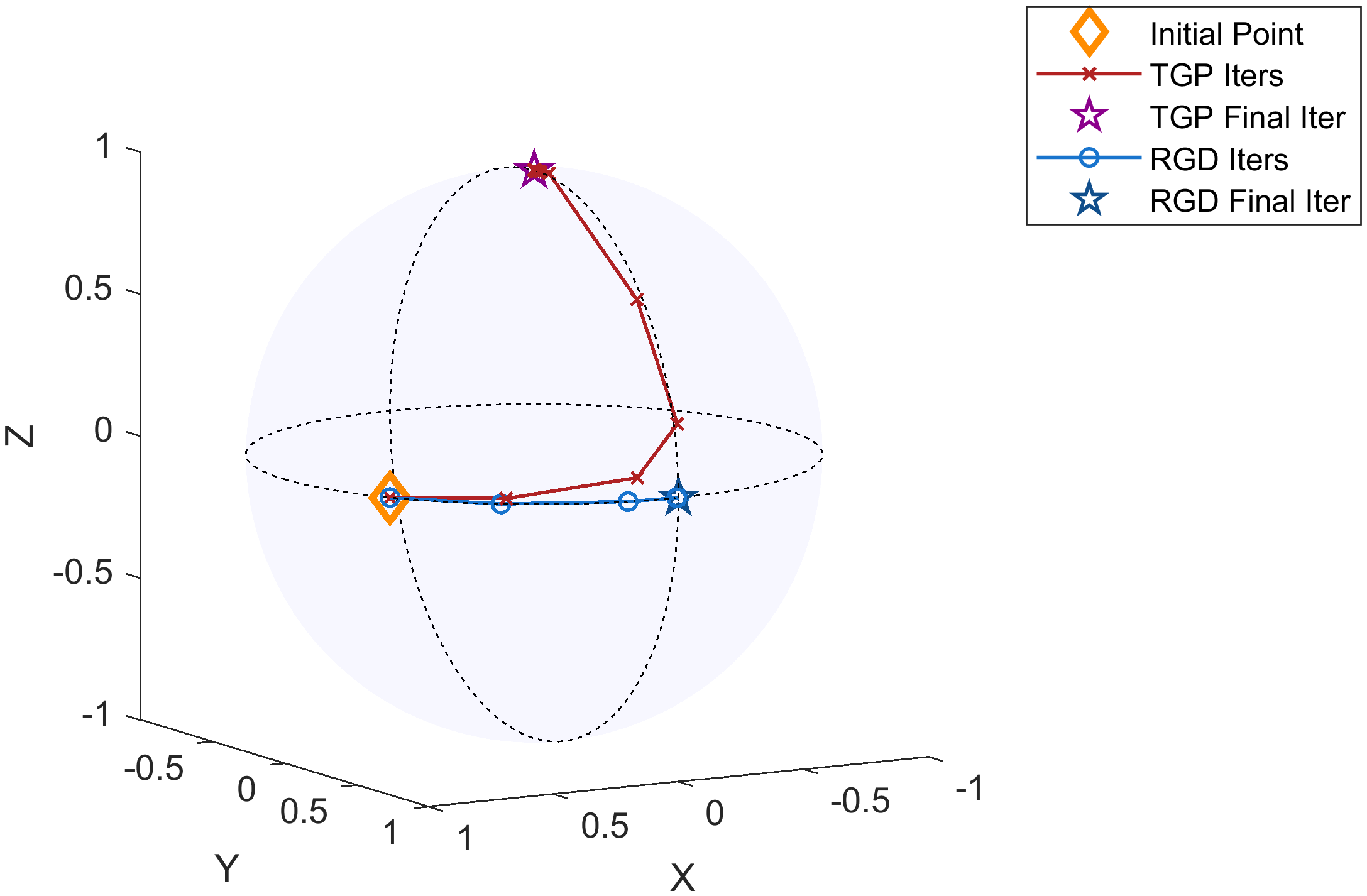}
  \caption{Paths of RGD and TGP-A-Eigen using the new search direction \eqref{eq:H-TGP-eigen}.}\label{fig:eigen_TGP_Eigen}
\end{figure}
\subsubsection{Effect of the normal vector}
We now demonstrate that, in several scenarios, the normal component of the search direction \( \matr{H}_k \) can accelerate the convergence of the TGP algorithms through a direct theoretical approach. 
As an illustrative example, we consider the eigenvalue problem \eqref{eq:eigenvalue_problem} and compare the convergence speed of three methods: RGD, EGP, and the Shifted PM. Since the primary difference between these methods lies in their treatment of the normal component, as shown later in \eqref{eq:H-RGD-GP-PM}, the differences in their performance reflect the impact of the normal component.
When framed within the TGP algorithmic framework, the search direction \( \matr{H}_k \) of these methods (denoted as $\matr{H}_k^{(R)}$,  $\matr{H}_k^{(G)}$, and $\matr{H}_k^{(P)}$, respectively) can be expressed as: 
 \begin{equation}\label{eq:H-RGD-GP-PM}
\begin{aligned}
  \matr{H}_k^{(R)} & = \grad f(\vect{x}_k) = ((\lambda_1 - 2f(\vect{x}_k)) (\vect{x}_k)_1 , \ldots , (\lambda_n - 2 f(\vect{x}_k))(\vect{x}_{k})_n),\\
  \matr{H}_k^{(G)} & = \nabla f(\vect{x}_k) = (\lambda_1 (\vect{x}_{k})_1 , \ldots , \lambda_n (\vect{x}_{k})_n),\\
  \matr{H}_k^{(P)} & = \nabla f(\vect{x}_k) + (1-s) \vect{x} = \left(\left(\lambda_1-s +1\right) (\vect{x}_{k})_1 , \ldots , \left(\lambda_n -s  + 1\right) (\vect{x}_{k})_n\right).
\end{aligned}
 \end{equation}
Suppose that we start from \( \vect{x}_0 \) with \(( \vect{x}_0 )_n \neq 0\). 
In this case, the expression
\( \max_{l \in [n-1]} \left\lvert\frac{(\vect{x}_k)_l}{(\vect{x}_k)_n}\right\rvert\) serves as a metric for assessing the convergence progress of the algorithm, while the ratio \(  \max_{l \in [n-1]} \left\lvert\frac{(\vect{x}_{k+1})_l}{(\vect{x}_{k+1})_n} \right\rvert/  \left\lvert\frac{(\vect{x}_k)_l}{(\vect{x}_k)_n}\right\rvert \) can be interpreted as an indicator of convergence speed.
Let us compare RGD and EGP, both using a fixed stepsize \(\tau\). For \( l \in [n-1] \), we denote \(b_{k, l}^{(R)}\), \(b_{k, l}^{(G)}\), and \( b_{k, l}^{(P)} \) as \( \left\lvert\frac{(\vect{x}_{k+1})_l}{(\vect{x}_{k+1})_n} \right\rvert/  \left\lvert\frac{(\vect{x}_k)_l}{(\vect{x}_k)_n}\right\rvert \) in RGD, EGP, and Shifted PM, respectively. It follows from direct computations that
\begin{align*}
  b_{k,l}^{(R)}= \left\lvert\frac{1-\tau\lambda_l + 2\tau f(\vect{x}_k)}{1-\tau\lambda_n + 2 \tau f(\vect{x}_k)}\right\rvert,
  \ b_{k,l}^{(G)}= \left\lvert\frac{1-\tau\lambda_l}{1-\tau\lambda_n}\right\rvert,
  \ b_{k,l}^{(P)}= \left\lvert\frac{ s-\lambda_l} { s-\lambda_n}\right\rvert.
\end{align*}
There are also numerous instances where \(b_{k, l}^{(G)} < b_{k, l}^{(R)}\), indicating that EGP converges faster than RGD. This is particularly evident in the following two scenarios: (i) when \( f(\vect{x}_k) > 0 \) and \( 1- \tau \lambda_1 > 0 \); (ii) when \( \lvert f(\vect{x}_k) \rvert \) is extremely large, such that \(b_{k, l}^{(R)} \approx 1\) while \(b_{k, l}^{(G)} < 1\) remains a constant. Moreover, it is clear that when \( \lambda_1 < s < 1 / \tau \), it always holds that \( b_{k, l}^{(P)} < b_{k, l}^{(G)} \).
The results of these three algorithms 
for case (i) are illustrated in \cref{fig:eigen_convergence_rate}, where we set \( \matr{A} = \diag{3, 3, 2} \), \( \vect{x}_0 = (\frac{1}{\sqrt{3}}, \frac{1}{\sqrt{3}}, \frac{1}{\sqrt{3}})^\top \), the fixed stepsize \(\tau = 0.2 \), and \( s = 4 \).
\begin{figure}[htbp]
  \centering
  \begin{subfigure}{0.44\linewidth}
    \centering
    \includegraphics[width=\textwidth]{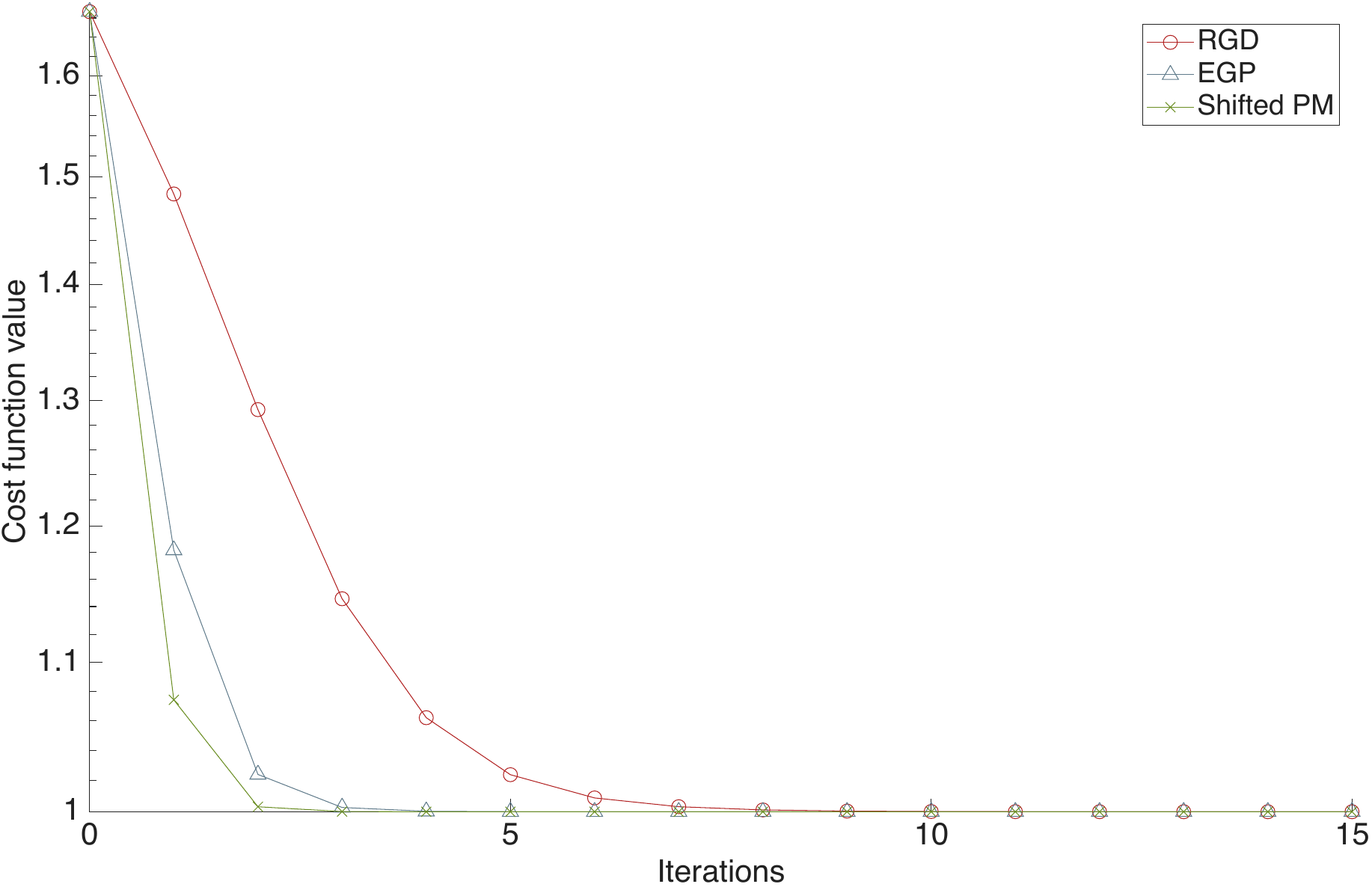}
  \end{subfigure}
 \hspace{0.04\textwidth} 
  \begin{subfigure}{0.44\linewidth}
    \centering
    \includegraphics[width=\textwidth]{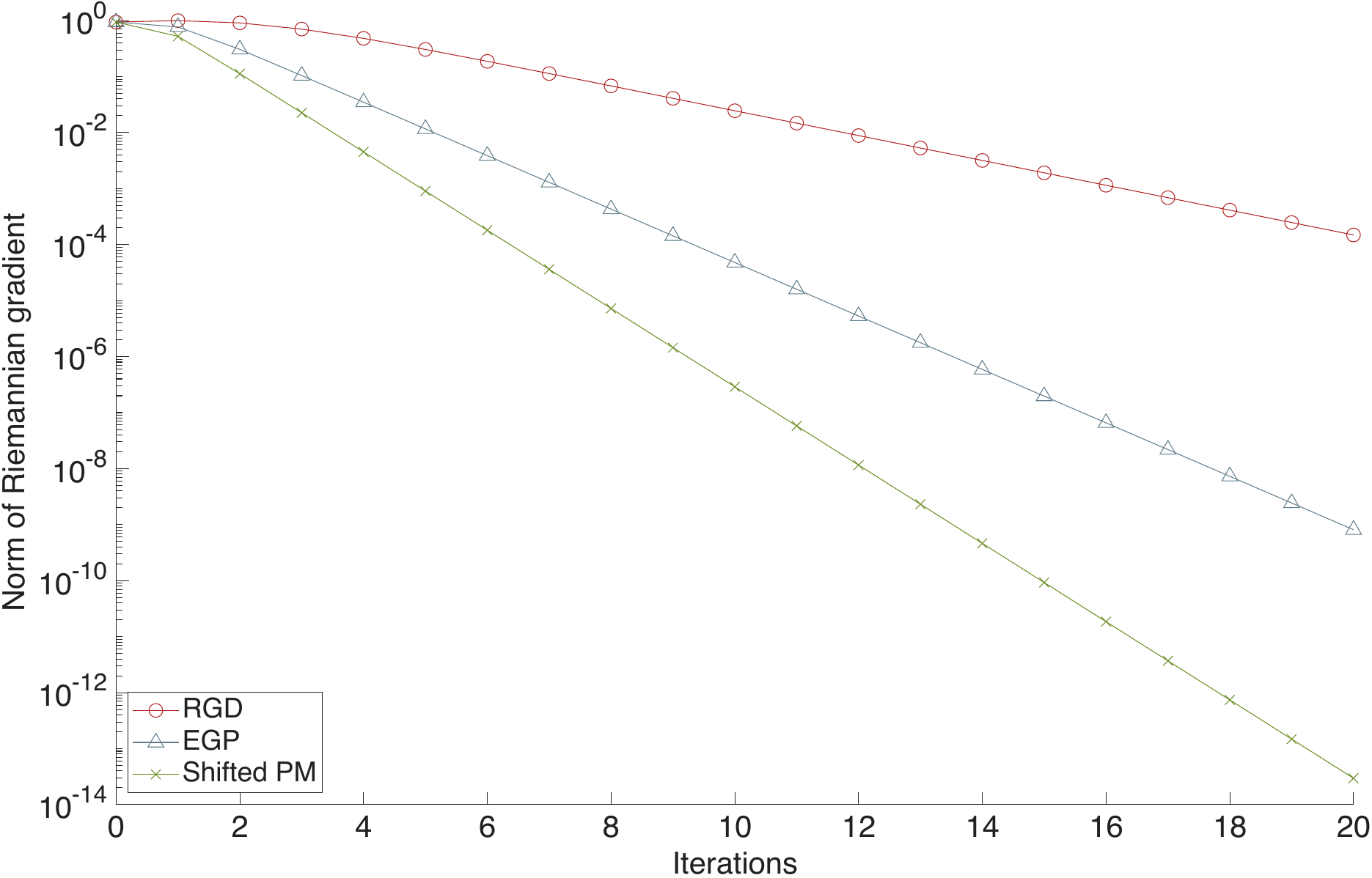}
  \end{subfigure}
  \caption{Comparison of RGD, EGP, and Shifted PM for the eigenvalue problem \eqref{eq:eigenvalue_problem}}\label{fig:eigen_convergence_rate}
\end{figure}

\section{Conclusions}\label{sec:conclu}
In this paper, using the projection onto a compact matrix manifold, we propose a general TGP algorithmic framework to solve the problem \eqref{eq:objec_func_g}. 
Our framework not only covers numerous existing algorithms in the literature, but also introduces several new special cases.
For this new general algorithmic framework with various stepsizes, including the Armijo stepsize, the Zhang-Hager type nonmonotone Armijo stepsize, and the fixed stepsize, we establish their weak convergence and iteration complexity results. A key aspect of our analysis is the exploration of the projection onto a compact submanifold, which is crucial for our convergence results and may also be of independent interest. 
Furthermore, by exploiting the \L{}ojasiewicz property, we establish the global convergence of our algorithmic framework and obtain several novel contributions. For example, to our knowledge, the global convergence of the Zhang-Hager type nonmonotone Armijo stepsize has not been established for nonconvex problems in the literature, even in Euclidean spaces. 

We would like to emphasize that, although we choose the tangent component of the search direction $\matr{H}_k$ in the form of a transformed Riemannian gradient in \eqref{eq:c_tilde_hat_transf}, our convergence analysis can be easily extended to more general ProjLS algorithms in \eqref{eq:iter_prjec-h}.
To be more specific, the weak convergence and iteration complexity can be established if the tangent component of the search direction $\matr{H}_k$ is scale-equivalent to and directionally aligned with the Riemannian gradient, \emph{i.e.}, if $\matr{\tilde{H}}_k$ satisfies \eqref{eq:H-grad-norm-equiv} and \eqref{eq:H-grad-product}. Moreover, global convergence under \L{}ojasiewicz property holds if the normal component of $\matr{H}_k$ is not too large, \emph{i.e.}, there exists \(\delta > 0\) such that $\tau_k \| \matr{\hat{H}}_k\| < \varrho_{*} - \delta$.

We conclude this paper by discussing several meaningful future research directions concerning the proposed TGP algorithmic framework.
\begin{itemize} 
\item \emph{Designing effective search direction components in the TGP framework}: As shown in \cref{sec:numer_exper}, the choice of scaling matrices $\matr{L}(\matr{X}_k)$, $\matr{R}(\matr{X}_k)$, and the additional normal vector $\matr{N}(\matr{X}_k)$ significantly influences the practical performance of TGP algorithms across different cost functions. Designing these components more effectively, either to improve convergence speed or to achieve better solution quality, is an important direction for future research. Although theoretical analysis in the general case is challenging, obtaining meaningful results for specific cost functions remains a valuable and feasible goal.  
\item \emph{Developing more advanced stepsize strategies within the TGP framework, \emph{e.g.}, the BB-type stepsize} \cite{iannazzo2018riemannian}: 
Due to space limitations, we did not develop and analyze more advanced stepsize strategies in the TGP framework, such as BB-type rules.
In fact, this is a promising yet challenging direction for future research, with the potential to enhance the practical performance and theoretical understanding of TGP algorithms.
\item \emph{Exploring second-order information in the TGP framework}: While the proposed TGP algorithmic framework primarily relies on first-order information through the Riemannian gradient and scaling matrices, a promising direction for future research is to incorporate partial or full second-order information. For example, designing TGP-type search directions based on approximations of the Riemannian Hessian, or employing quasi-Newton updates that are compatible with projection steps, may further accelerate convergence and improve solution accuracy, particularly for ill-conditioned problems. Theoretical analysis in this setting, especially concerning weak and global convergence, remains largely unexplored and could lead to a new line of research.
\end{itemize}

\section{Declarations}

\textbf{Competing interests}: The authors have no competing interests to declare that are relevant to the content of this article.

\bibliographystyle{siamplain}
\bibliography{ReferenceTensor}

\end{document}